\documentclass[runningheads,a4paper]{elsarticle}
\usepackage{amsmath,amsfonts,amsthm}
\usepackage{amssymb,amscd}
\usepackage[enableskew]{youngtab}
\usepackage{graphicx}

\def\muntz{M\"untz }

\newtheorem{theorem}{Theorem}
\newtheorem{lemma}{Lemma}
\newtheorem{proposition}{Proposition}
\newtheorem{definition}{Definition}
\newtheorem{corollary}{Corollary}
\newtheorem{example}{Example}
\newdefinition{remark}{Remark}
\begin{document}

\title{Chebyshev Blossom in \muntz Spaces: Toward Shaping
with Young Diagrams}

\author[rvt,els]{Rachid Ait-Haddou\corref{cor1}}
\ead{rachid@bpe.es.osaka-u.ac.jp}
\author[focal]{Yusuke Sakane}
\author[rvt,els]{Taishin Nomura}

\cortext[cor1]{Corresponding author}

\address[rvt]{The Center of Advanced Medical Engineering and Informatics,
Osaka University, 560-8531 Osaka, Japan}
\address[focal]{Department of Pure and Applied Mathematics,
Graduate School of Information Science and Technology,
Osaka University, 560-0043 Osaka, Japan}
\address[els]{Department of Mechanical Science and Bioengineering
Graduate School of Engineering Science,
Osaka University, 560-8531 Osaka, Japan}

\begin{abstract}
The notion of blossom in extended Chebyshev spaces offers adequate
generalizations and extra-utilities to the tools for free-form 
design schemes. Unfortunately, such advantages are often overshadowed
by the complexity of the resulting algorithms.
In this work, we show that for the case of \muntz spaces with 
integer exponents, the notion of Chebyshev blossom leads to elegant algorithms
whose complexities are embedded in the combinatorics of Schur functions.
We express the blossom and the pseudo-affinity property in \muntz 
spaces in term of Schur functions. We derive an explicit expression of the 
Chebyshev-Bernstein basis via an inductive argument 
on nested \muntz spaces. We also reveal a simple algorithm for the 
dimension elevation process. Free-form design schemes in \muntz spaces
with Young diagrams as shape parameter will be discussed.
\end{abstract}      

\maketitle
%%%%%%%%%%%%%%%%%%%%%%%%%%%%%%%%%%%%%%%%%%%%%%%%%%%%%%%%%%%%%%%%%%%%%%%%%%%%%%%%%%%%%%%%%%%%%%%%%%
%%%%%%%%%%%%%%%%%%%%%%%%%%%%%%%%%%%%%%%%%%%%%%%%%%%%%%%%%%%%%%%%%%%%%%%%%%%%%%%%%%%%%%%%%%%%%%%%%%
\section{Introduction}
Representing a polynomial on an interval by its B\'ezier points is 
a common practice in the field of computer aided geometric design \cite{Farin}.
Namely, a polynomial $F$ of degree $n$, can be written as
$$
F(t) = \sum_{i=0}^n B_i^n (t) P_i; \qquad t \in [a,b]
$$
where $B_i^n (t) = \binom{n}{i} \alpha (t)^i \beta (t)^{n-i}, i=0, ..., n$
is the Bernstein basis in the space $\mathcal{P}_n$ of polynomials of degree $n$
with respect to the interval $[a,b]$, $\alpha(t)$ and $\beta(t)$ 
are the barycentric coordinates of the point $t$ with respect to the interval $[a,b]$, i.e.,
$\alpha(t) = (t-a) / (b-a)$ and $\beta(t) = (b-t) / (b-a)$.
Important features of this representation are that the piecewise linear
interpolant of the B\'ezier points $(P_0, P_1,..., P_n)$ reflects,
to a certain extent, the shape of the polynomial curve and that 
the end-segments $[P_0,P_1]$ and $[P_{n-1},P_n]$ are tangents to the curve
at the point $F(a)$ and $F(b)$ respectively. Furthermore, the curve lies 
in the convex hull of the control polygon and the de Casteljau algorithm
leads to an efficient method for the evaluation of the polynomial from its 
control points. The total positivity of the Bernstein basis gives rise to many 
shape preserving properties. For example, the diminishing variation property 
\cite{Ait-Haddou1} ensures that the number of times an arbitrary hyperplane 
crosses the curve is no more than the number of times that hyperplane
crosses the control polygon. The notion of blossom introduced by Ramshaw \cite{Ramshaw}
offers an elegant and unifying approach to the understanding of the many 
aspects of the theory of B\'ezier curves. The fundamental idea of blossoming 
is that for any polynomial function $F$ of degree $n$ there exists a unique function
$f(u_1,u_2,...,u_n)$ that is $n$-affine (i.e., $f$ is a polynomial of degree less than 
or equal to $1$ with respect to each separate variable), symmetric (i.e., $f(u_1,u_2,...,u_n)
= f(u_{\sigma(1)},u_{\sigma (2)},...,u_{\sigma(n)})$ for any permutation 
$\sigma$ on $\{1,2,...,n\}$) and satisfies $f(t,t,..., t)=F(t)$ for every $t \in \mathbb{R}$.
The function $f$ is called the blossom or polar form of $F$. The control points of the 
polynomial $F$ with respect to the interval $[a,b]$ are then expressed in term of the 
blossom as $P_i=f (a^{n-i},b^i)$ for $i = 0, ..., n$. The multi-affinity of the blossom 
leads in a very natural way to the de Casteljau algorithm.
Moreover, the blossom of a polynomial has a simple expression, namely if the polynomial
$F$ is expressed in the monomial basis as $F(t) = \sum_{i = 0}^n a_i t^i$ then 
its blossom is given by
$$
f (u_1,u_2, ...,u_n ) = \sum_{i=0}^n \frac{a_i}{\binom{n}{k}} e_k (u_1,u_2,...,u_n ),
$$
where $e_k(u_1,u_2,...,u_n)$ is the $k$th elementary symmetric
function in the variables $u_1,...,u_n$. The concept of blossoming was extended
by Pottmann \cite{Pottmann} to include any linear space $E = span \left(1,\phi_1,\phi_2, ...,
\phi_n \right)$ such that $span \left( \phi_1',
\phi_2', ..., \phi_n' \right)$ is an extended Chebyshev space of order $n$ on an interval.
The proposed extension bears striking similarities to the polynomial framework and in which notions of
control points, de Casteljau algorithm and generalized Bernstein basis can be defined.
Moreover, the emergence of the interval of interest as a shape parameter makes this extension 
fundamental in free-form curve design. However, while the expression of the blossom in the 
space of polynomials is simple, the expression of the blossom in a generic extended 
Chebyshev space is much more complicated in general. Thereby, leading to more complicated subdivision schemes. 
The main objective of this paper is to show that at least for the case of \muntz spaces
with integer exponents, the notion of blossoming provides us with an elegant theory in which
the resulting algorithms could be understood and made easy once we invoke the 
notion of Schur functions.
The paper is organized as follows : In the second section, 
we review the basic properties of Chebyshev blossoming \cite{Mazure1,Mazure2}.
Emphasis will be given to the notions that will be needed within 
this work, such as the definition of Chebyshev blossom, 
the pseudo-affinity property, characterizations of 
Chebyshev-Bernstein basis and the process of dimension elevation. 
In section 3, we review the fundamentals of Chen iterated integrals. 
Introducing Chen iterated integral has a twofold aims. Firstly 
it will allow us to give an interesting determinantal expression
of the Chebyshev blossom of Chebyshev functions defined
in terms of the so-called weight functions \cite{Mazure5}, thereby, 
allowing the main result of the section to be used in different contexts 
than the one of \muntz spaces. Second, the determinantal expression will provide 
us, in section 5, with the Chebyshev blossom of \muntz spaces with integer 
exponents without resorting to solving linear systems. The relevant properties of 
Schur functions will be recalled in section 4. In section 5,
we give the expression of the Chebyshev blossom of \muntz spaces with integer exponents 
in terms of Schur functions. The main result of this section is essentially the 
same as in \cite{Mazure3,Mazure4} in which a different convention was adopted and the connection 
with Schur functions seems to not to have been noticed. 
Several fundamental examples that will guide us throughout this work 
will be given. Using the Dodgson condensation formula, 
an expression of the pseudo-affinity property in terms
of Schur functions will be given in Section 6. Such expression will 
be fundamental, through section7, in deriving an explicit expression
of the Chebyshev-Bernstein basis in any \muntz space with integer exponents. 
Note that there is only a single case in which an explicit expression of 
Chebyshev-Bernstein basis is known, namely the space  
$span\{1,t^{k+1},t^{k+1},t^{k+2},...,t^{k+n} \}$ where $k$ is a positive integer \cite{Mazure4}.
Our strategy for deriving such an explicit expression consists of two steps.
First, we show that, although the de Casteljau algorithm is not able to provide us with
meaningful expressions of the Chebyshev-Bernstein basis, 
it will allow us to gain extra information on the derivatives of these bases.
Then, the explicit expression will be obtained via a dimension elevation process
and some combinatorial manipulations on nested \muntz spaces. The Chebyshev-Bernstein 
bases can be defined without resorting to the notion of Chebyshev blossom. Therefore,
this section shows, in particular, the importance of the notion of Chebyshev blossom
in solving this specific problem.  
In section 8, we give a simple algorithm for the dimension elevation process. 
The idea of using Young diagrams as shape parameter for free-from design and for 
the problem of continuity of composite Chebyshev-B\'ezier curves will be discussed.
Expression for the derivative of the Chebyshev-Bernstein basis will also be given.  
%%%%%%%%%%%%%%%%%%%%%%%%%%%%%%%%%%%%%%%%%%%%%%%%%%%%%%%%%%%%%%%%%%%%%%%%%%%%%%%%%%%%%%%%%%%%%%%%%%
%%%%%%%%%%%%%%%%%%%%%%%%%%%%%%%%%%%%%%%%%%%%%%%%%%%%%%%%%%%%%%%%%%%%%%%%%%%%%%%%%%%%%%%%%%%%%%%%%%
\section{Chebyshev blossom and Chebyshev-Bernstein basis}
In this section, we review the basic properties of Chebyshev blossoming and 
provide the relevant informations that will be used within this work.
We will mainly follow the terminology and the notations of the excellent report 
\cite{Mazure1}. Although, there is optimal smoothness conditions on Chebyshev 
functions in order to define the Chebyshev blossom, we will assume here, for simplicity,
that all the functions that we encounter are infinitely differentiable.   
%%%%%%%%%%%%%%%%%%%%%%
\vskip 0.2 cm 
\noindent{{\bf {Chebyshev blossom: }}}
Let $I$ denote a non-empty real interval, and let $\phi=(\phi_1,
\phi_2, ...,\phi_n)^{T}$ be a $C^{\infty}$ function from the
interval $I$ into $\mathbb{R}^{n}$ (the space $\mathbb{R}^{n}$ is 
viewed as an $n$-dimensional affine space). Let us assume that the linear space 
$D\mathcal{E}(\phi)=span(\phi_1',\phi_2',...,\phi_n')$
is an $n-$dimensional extended Chebyshev space on $I$, i.e.,
each non-zero element of this space vanishes (counting multiplicities)
at most $n-1$ times on $I$. In this case, we say that the function $\phi$ 
is a {\it{Chebyshev function of order $n$ on $I$}}. The linear space
$\mathcal{E}(\phi) = span(1,\phi_1,\phi_2,...,\phi_n)$ is an $(n+1)$-dimensional 
extended Chebyshev space that we call the Chebyshev space associated 
with the Chebyshev function $\phi$.
If for any real number $t$ in $I$, we denote by $Osc_{i}\phi(t)$ the osculating 
flat of order $i$ of the function $\phi$ at the point $t$, i.e.,
$$
Osc_i\phi(t) = \{\phi(t)+\alpha_1 \phi'(t)+... + \alpha_i \phi^{(i)}(t) 
\quad | \quad 
\alpha_1,...,\alpha_i \in \mathbb{R} \}, 
$$
then the assumption that $\phi$ is a Chebyshev function of order $n$
imply that for all $t \in I$ and for all $i = 0, ..., n$, the osculating
flat $Osc_i\phi(t)$ is an $i$-affine dimensional space {\cite{Mazure1}.
Moreover, it can be shown that for all distinct points 
$\tau_1,..., \tau_r$ in the interval $I$ and all positive integers 
$\mu_1, ..., \mu_r$ such that $\sum_{k = 1}^r \mu_k = m \leq n$, we
have 
\begin{equation}\label{intersection}
dim \cap_{k = 1}^r Osc_{n-\mu_k}\phi(\tau_k) = n-m.
\end{equation}
In particular, if in equation (\ref{intersection}) we have $m=n$,
then the intersection consists of a single point in $\mathbb{R}^{n}$,
which we label as 
$\varphi(\tau_1^{\mu_1},\tau_2^{\mu_2},...,\tau_r^{\mu_r})$, i.e.,
$$
\varphi(\tau_1^{\mu_1},\tau_2^{\mu_2},...,\tau_r^{\mu_r}) = 
\cap_{k = 1}^r Osc_{n - \mu_k} \varphi ( \tau_k ).
$$
The previous construction provides us with a function 
$\varphi = (\varphi_1,\varphi_2, ..., \varphi_n)^{T}$ 
from $I^n$ into $\mathbb{R}^{n}$ with the following 
straightforward properties: The function $\varphi$ is symmetric 
in its arguments and its restriction to the diagonal of $I^n$ 
is equal to $\phi$ i.e., $\varphi(t,t,...,t) = \phi(t)$. 
The function $\varphi$ is called the {\it{Chebyshev blossom}}
of the function $\phi$. Note that the definition 
of the Chebyhev blossom imply in particular that if we are given 
$n$ pairwise distinct real numbers $u_i, i=1,...,n$ in the interval $I$, 
then the Chebyshev blossom value $X = \varphi(u_1,...,u_n)$
is given by the solution of the linear system 
\begin{equation} \label{linearsystem}
\det\left(X-\phi(u_{i}),\phi'(u_{i}),...,\phi^{(n-1)}(u_i)\right) = 0,
\quad  i=1,...,n.
\end{equation}      

%%%%%%%%%%%%%%%%%%%%%%
\vskip 0.2 cm 
\noindent{{\bf {The pseudo-affinity property: }}}
Another fundamental property of Chebyshev blossom is the notion 
of pseudo-affinity. Let assume given $(n-1)$ real numbers 
$T = (u_1,u_2,...,u_{n-1}) 
= (\tau_1^{\mu_1},\tau_2^{\mu_2},...,\tau_r^{\mu_r})$ ($\mu_1 + ...+\mu_r = n-1$)
in the interval $I$. According to equation (\ref{intersection}), 
the affine space 
$$
L = \cap_{k = 1}^r Osc_{n-\mu_i}\phi(\tau_i)
$$
is an affine line. Therefore, for any $t$ in the interval $I$, the point 
$$
\varphi(u_1,...,u_{n-1},t)
$$
belongs to the line $L$. In other word, there exists a function 
$\alpha$ such that for any distinct numbers $a$ and $b$ in the interval $I$,
and for any $t \in I$, we have      
\begin{equation}\label{pseudoaffinity}
\varphi (u_1,...,u_{n-1},t)=\left(1-\alpha(t)\right)\varphi(u_1,...,u_{n-1},a) +
\alpha(t) \varphi(u_1,...,u_{n-1},b).
\end{equation}
Moreover, it is shown in \cite{Mazure1} that the function $\alpha$ is a $C^{\infty}$
strictly monotonic function from the interval $I$ to $\mathbb{R}$ satisfying
$\alpha(a) = 0$ and $\alpha(b)=0$. The function $\alpha$ will be  
called the {\it{pseudo-affinity factor}} associated with the Chebyshev 
space $\mathcal{E}(\phi)$. In general, the function $\alpha$ depends on the interval $[a,b]$,
the real numbers $u_i,i=1,...,n-1$ as well as the parameter $t$. To stress this 
dependence, we will often write the pseudo-affinity factor as 
$\alpha(u_1,...,u_{n-1};a,b,t)$. 
%%%%%%%%%%%%%%%%%%%%%%
\vskip 0.2 cm 
\noindent{{\bf {Chebyshev-Bernstein Basis: }}}
Given two real numbers $a$ and $b$ in the interval $I$ ($a < b$), and denote
by $\Pi_k$, $k=0,...,n$, the $(n+1)$ points defined as 
$$
\Pi_i = \varphi(a^{n-i},b^{i}). 
$$
The points $\Pi_i$ are affinely independent in $\mathbb{R}^{n}$ \cite{Mazure1}.
Therefore, there exist $(n+1)$ functions $B^{n}_{k}, k=0,...,n$ such that 
for any $t \in I$
$$
\phi(t) = \sum_{k=0}^n B^{n}_{k}(t)\Pi_i \quad \textnormal{and} 
\quad \sum_{k=0}^n B^{n}_{k}(t)=1.
$$
The functions $B^{n}_{0},...,B^{n}_{k},..., B^{n}_{n}$ form a
basis of the Chebyshev space $\mathcal{E}(\phi)$, called the Chebyshev-Bernstein
basis of the space $\mathcal{E}(\phi)$ with respect to the interval $[a,b]$.
In this work, we will use the following characterization of the 
Chebyshev-Bernstein basis \cite{Mazure2}:
\begin{theorem}\label{mazuretheorem}
The Chebyshev-Bernstein basis $(B^{n}_{0},...,B^{n}_{k},...,B^{n}_{n})$ 
with respect to the interval $[a,b] \subset I$, is the unique normalized basis of the
space $\mathcal{E}(\phi)$ such that for $k=0,...,n$, $B^{n}_{k}$ vanishes $k$ 
times at $a$ and $n-k$ times at $b$.    
\end{theorem}
%%%%%%%%%%%%%%%%%%%%%%
\vskip 0.2 cm 
\noindent{{\bf {$\mathcal{E}(\phi)$-functions and its blossom: }}}
A function $F$ from the interval $I$ into $\mathbb{R}^{m}$, $m\leq n$
is called a $\mathcal{E}(\phi)$-function if all its components belong 
to the space $\mathcal{E}(\phi)$ i.e., there exists an affine map $h$ 
on $\mathbb{R}^n$ such that $F = h \circ \phi$. The Chebyshev blossom 
of $F$ is then defined as the affine image of the Chebyshev blossom of 
$\phi$ under the map $h$, i.e., $f = h \circ \varphi$. We define the 
{\it{Chebyshev-B\'ezier}} points with respect to an interval 
$[a,b] \subset I$ of a $\mathcal{E}(\phi)$-function $F$ by 
$$
P_i = f(a^{n-i},b^{i}), \quad i=0,...,n.
$$
As the Chebyshev blossom $f$  of the function $F$ inherits the 
pseudo-affinity property (\ref{pseudoaffinity}), 
the value $F(t)$ can be computed as an affine combination
of the points $P_i, i=0,...,n$, leading to the so-called 
{\it de Casteljau algorithm}. Note, also that the function $F$
can be written as 
$$
F(t) = \sum_{k=0}^n B^{n}_{k}(t)P_i, \quad  t \in I,
$$   
where $(B^{n}_{0},...,B^{n}_{k},...,B^{n}_{n})$ is the Chebyshev-Bernstein 
basis of the space $\mathcal{E}(\phi)$ with respect to the interval $[a,b]$. 
%%%%%%%%%%%%%%%%%%%%%%
\vskip 0.2 cm 
\noindent{{\bf {Dimension elevation process: }}} 
Consider another Chebyshev function $\psi$ of order $n+1$ on 
the same interval $I$ and such that $\mathcal{E}(\phi) \subset \mathcal{E}(\psi)$.
Let $F$ be $\mathcal{E}(\phi)$-function and denote by $P_i, i=0,...,n$
its Chebyshev-B\'ezier points with respect to the interval $[a,b]$.
The function $F$ can also be viewed as an $\mathcal{E}(\psi)$-function 
and then having different Chebyshev-B\'ezier points $\tilde{P}_i, i=0,...,n+1$ 
with respect to the interval $[a,b]$. From the definition of the 
Chebyshev blossom we necessarily have $\tilde{P}_0 = P_{0}$ and 
$\tilde{P}_{n+1} = P_{n}$. Moreover, it can be shown \cite{Mazure1}
that there exist real numbers $\xi_i \in ]0,1[, i=1,...,n$ such that
\begin{equation}\label{generalelevation}
\tilde{P}_i = (1- \xi_{i})P_{i-1} + \xi_{i} P_{i}, \quad  
i = 1,...,n.
\end{equation}
%%%%%%%%%%%%%%%%%%%%%%%%%%%%%%%%%%%%%%%%%%%%%%%%%%%%%%%%%%%%%%%%%%%%%%%%%%%%%%%%%%%%%%%%%%%%%%%%%%
%%%%%%%%%%%%%%%%%%%%%%%%%%%%%%%%%%%%%%%%%%%%%%%%%%%%%%%%%%%%%%%%%%%%%%%%%%%%%%%%%%%%%%%%%%%%%%%%%%
\section{Chen Iterated Integrals and Chebyshev Blossom}
In this section, we give an expression of the Chebyshev blossom 
of Chebyshev functions defined in term of the so-called weight 
functions. The notion of Chen iterated integrals \cite{Chen1} and their 
properties reveal to be fundamental in deriving such expression. 
Due to the simplicity of the proofs of the properties of Chen iterated
integrals and for the sake of completeness, we will include such 
proofs in this section. 

\noindent{\bf{Chen iterated integral: }} Let $\omega_1, \omega_2,
...,\omega_n$ be $C^{\infty}$ functions on a non-empty 
real interval $I$. Let $a$ and $b$ two real numbers in $I$. 
The Chen iterated integral is defined iteratively as 
follows :
$$
L_{w_1}^{[a,b]} = \int_{a}^{b} {\omega_1(t)} dt, 
$$
and for $p=2,...,n$, we define 
$$
L_{\omega_1 \omega_2 ...\omega_p}^{[a,b]} = 
\int_{a}^{b} \omega_1(t) L_{\omega_2 ...\omega_p}^{[a,t]} dt.
$$
Therefore, the Chen iterated integral can be written as 
$$
L_{\omega_1 \omega_2 ...\omega_n}^{[a,b]} = 
\int_{a}^{b} \omega_1(t_1) \int_{a}^{t_1} \omega_2(t_2) \int_{a}^{t_2} ....
\int_{a}^{t_{n-1}} \omega_{n}(t_n) \hskip 0.1cm dt_n dt_{n-1} ...dt_{1}, 
$$
or if $a \leq b$ as  
\begin{equation}\label{simplex}
L_{\omega_1 \omega_2 ...\omega_n}^{[a,b]} = \int_{\Delta_n} 
\omega_1(t_1) \omega_2(t_2) ... \omega_n(t_n) dt_1 ...dt_n,
\end{equation}
where $\Delta_n$ is the $n$-simplex in $\mathbb{R}^{n}$ 
$$
\Delta_n = \{ (t_1,t_2,...,t_n) \in \mathbb{R}^{n} \quad | \quad 
b > t_1 > t_2 > ... > t_n > a \}.
$$
Chen iterated integrals have the following properties \cite{Bowman,Chen1} 

%%%%%%%%%%%%%%%%%%%%%
\begin{proposition}
For any real numbers $a,b$ and $c$ in the interval $I$, we have
\begin{equation}\label{reversepath}
L_{\omega_1 \omega_2 ...\omega_n}^{[a,b]} = 
(-1)^n L_{\omega_n \omega_2 ...\omega_1}^{[b,a]}
\end{equation}
and  
\begin{equation}\label{convolution}
L_{\omega_1 \omega_2 ...\omega_n}^{[a,b]} =  
\sum_{i=0}^{n} 
L_{\omega_1 \omega_2 ...\omega_i}^{[c,b]} 
L_{\omega_{i+1} \omega_{i+2} ...\omega_n}^{[a,c]}
\end{equation}
with the convention that 
$L^{[x,y]}_{\omega_1 \omega_2 ...\omega_r} = 1$ if $r=0$ or $r > n$.
\end{proposition}
%%%%%%%%%%%%%%%%%%%%%%%
\begin{proof}
To prove (\ref{reversepath}), 
we can proceed as follows. We first remark that 
\begin{equation*}
\begin{split}
L_{\omega_1 \omega_2 ...\omega_n}^{[a,b]} & =  
\int_{a}^{b} \omega_1(t_1) \int_{a}^{t_1} \omega_2(t_2) \int_{a}^{t_2} ....
\int_{a}^{t_{n-1}} \omega_{n}(t_n) \hskip 0.1cm dt_n dt_{n-1} ...dt_{1} \\
& = \int_{a}^{b} \omega_n(t_n) \int_{t_n}^{b} \omega_{n-1}(t_{n-1}) \int_{t_{n-2}}^{b} ....
\int_{t_2}^{b} \omega_{1}(t_1) \hskip 0.1cm dt_1 dt_{2} ...dt_{n}. 
\end{split}
\end{equation*}
Then, we switch the limit of integration at each level starting from $\omega_1$.
Equation (\ref{convolution}) can be proven by induction on the number of weight functions $\omega_i$.
The equality is obvious for $n=1$. Let us assume the equality to be true for the $(n-1)$ weight functions 
$\omega_2,...,\omega_{n}$. Replacing $b$ in (\ref{convolution}) by a variable $t$ and 
differentiating both sides of the equation with respect to $t$ shows, by the induction hypothesis,
that there exists a constant $K$ such that  
\begin{equation}\label{inter} 
L_{\omega_1 \omega_2 ...\omega_n}^{[a,t]} =  
\sum_{i=0}^{n} 
L_{\omega_1 \omega_2 ...\omega_i}^{[c,t]} 
L_{\omega_{i+1} \omega_{i+2} ...\omega_n}^{[a,c]} + K.
\end{equation}  
Taking $t=c$ in the last equation, shows that the constant $K =0$.
\end{proof}
%%%%%%%%%%%%%%%%%%%%%%%%
If we take $a=b$ in (\ref{convolution}), and taking into account (\ref{reversepath}),
we arrive, after renaming the variables, to the following 
%%%%%%%%%%%%%%%%%
\begin{corollary}\label{sumChen}
For any $a$ and $b$ in the interval $I$, we have 
\begin{equation}\label{corollarychen}
\sum_{i=1}^{n} (-1)^{i-1} L_{\omega_1...\omega_i}^{[a,b]} 
L_{\omega_n ... \omega_{i+1}}^{[a,b]} =L_{\omega_n \omega_{n-1} ...\omega_1}^{[a,b]}.
\end{equation}
\end{corollary}
%%%%%%%%%%%%%%
%%%%%%%%%%%%%%
\begin{remark}
Probably the most imporant property of Chen iterated integrals is the so-called shuffle 
product of two Chen iterated integrals \cite{Bowman,Chen1}. In our present context of 
Chebyshev blossom in \muntz spaces, such property is not needed. However, in a future contribution 
we will exhibit it importance for Chebyshev blossom of Chebyshev functions defined in terms of 
weight functions. 
\end{remark}
\vskip 0.2 cm
\noindent{\bf{A determinant formulas: }} Let $\omega_1, \omega_2,...,\omega_n$ 
be $C^{\infty}$ functions on a real interval $[a,b]$. Denote by 
$I^{[a,b]}(\omega_1,\omega_2,...,\omega_n) = (a_{ij})_{1\leq i,j\leq n}$ 
the square matrix of order $n$ given by 
$$
a_{ii} = L_{\omega_i}^{[a,b]}, \quad 
a_{i,i+1} = 1, \quad i=1,...,n,
$$
and
$$
a_{i,j} = 0 \quad \textnormal{if} \quad j > i+1, \quad 
a_{i,j} = L_{\omega_j \omega_{j+1} ...\omega_i}^{[a,b]} 
\quad \textnormal{if} \quad j < i.
$$
The matrix $I^{[a,b]}(\omega_1,\omega_2,...,\omega_n)$ has the form 
\begin{equation}\label{chenmatrix}
\begin{pmatrix} 
L_{\omega_1}^{[a,b]}&1 & 0 & 0 & \dots&0 \\
L_{\omega_1 \omega_2}^{[a,b]}& L_{\omega_2}^{[a,b]}&1&0& 0\dots& 0\\
L_{\omega_1 \omega_2 \omega_3}^{[a,b]}& L_{\omega_2 \omega_3}^{[a,b]}&
L_{\omega_3}^{[a,b]}&1& 0\dots&0\\
\hdotsfor[2]{5}   \\
L_{\omega_1 \omega_2 ...\omega_n}^{[a,b]}&
L_{\omega_2 ...\omega_n}^{[a,b]} & L_{\omega_3 ...\omega_n}^{[a,b]}&\dots&
L_{\omega_{n-1}...\omega_n}^{[a,b]}&L_{\omega_n}^{[a,b]}\end{pmatrix}
\end{equation}
%%%%
%%%%
\begin{proposition}\label{determinantformulas}
The determinant of the matrix $I^{[a,b]}(\omega_1,\omega_2,...,\omega_n)$
is given by  
\begin{equation}\label{determinant}
\det(I^{[a,b]}(\omega_1,\omega_2,...,\omega_n)) = 
L_{\omega_n \omega_{n-1}...\omega_1}^{[a,b]}.
\end{equation}
\end{proposition}

\begin{proof}
We proceed by induction on $n$. For $n=2$, we have 
$$
\det(I^{[a,b]}(\omega_1,\omega_2)) = L_{\omega_1}^{[a,b]} L_{\omega_2}^{[a,b]}
- L_{\omega_1 \omega_2}^{[a,b]}.   
$$
Equation (\ref{corollarychen}) for $n=2$, gives 
$$
L_{\omega_2 \omega_1}^{[a,b]} = L_{\omega_1}^{[a,b]} L_{\omega_2}^{[a,b]} -
L_{\omega_1 \omega_2}^{[a,b]},  
$$
thereby, showing (\ref{determinant}) for $n=2$. Let us assume (\ref{determinant}) 
to be true for all $m < n$. Now, by expanding the determinant
$I^{[a,b]}(\omega_1,\omega_2,...,\omega_n)$ down the first column, we shall obtain
$$
\det(I^{[a,b]}(\omega_1,\omega_2,...,\omega_n)) = 
\sum_{i=1}^{n} (-1)^{i-1}L_{\omega_1 \omega_2 ...\omega_i}^{[a,b]} 
\det(I^{[a,b]}(\omega_{i+1},\omega_{i+2},...,\omega_{n}).
$$
By the inductive hypothesis, we then have 
$$
\det(I^{[a,b]}(\omega_1,\omega_2,...,\omega_n)) = 
\sum_{i=1}^{n} (-1)^{i-1}L_{\omega_1 \omega_2 ...\omega_i}^{[a,b]} 
L_{\omega_n \omega_{n-1} ...\omega_{i+1}}^{[a,b]}.
$$ 
Applying again Corollary \ref{sumChen} leads to the desired result.
\end{proof}
%%%%%%%%%%%%%%%%%%%%%
\vskip 0.2 cm
\noindent{\bf{Chen iterated integral and Chebyshev blossom: }}
Let $\omega_1,...,\omega_n$ be $C^{\infty}$ functions non-vanishing 
on a real interval $I$ and defined on an interval $J \supset I$.
Let $a$ be a fixed real number in the interval $J$, then it is well known 
\cite{Mazure5} that the function
\begin{equation}\label{chebyshevfunction}
\phi(t) = (L_{\omega_1}^{[a,t]},L_{\omega_1 \omega_2}^{[a,t]},...,
L_{\omega_1 \omega_2 ...\omega_n}^{[a,t]})^{T}
\end{equation}
is a Chebyshev function of order $n$ on the interval $I$. 
In the following, we show that the Wronskian of $\phi$ has an 
interesting expression in terms of Chen interated integrals,
more precisely, we have
%%%%
%%%%
\begin{proposition}\label{determinantgeneral}
For any real number $t$ in the interval $I$, the Chebyshev
function $\phi$ in (\ref{chebyshevfunction}) satisfies 
$$
\det(\phi(t),\phi'(t),....,\phi^{n-1}(t)) = 
\omega_{1}^{n-1}(t) \omega_{2}^{n-2}(t)...\omega_{n-1}(t) 
L_{\omega_n \omega_{n-1} ...\omega_1}^{[a,t]}.
$$
\end{proposition}
%%%%
%%%%
\begin{proof}
We first notice that 
\begin{equation*}
\phi'(t) = \omega_1(t) \left(1,L_{\omega_2}^{[a,t]},L_{\omega_2 \omega_3}^{[a,t]},...,
L_{\omega_2 \omega_3 ...\omega_n}^{[a,t]} \right)^{T}.
\end{equation*} 
Moreover, by a simple inductive argument, it can be shown that for
$2 \leq k \leq n-1$, there exist differentiable functions 
$\rho_{i,k}$ such that 
$$
\phi^{(k)}(t) = \sum_{i=1}^{k-1} \rho_{i,k}(t) \phi^{(i)}(t) +
\omega_1(t) \omega_2(t)...\omega_{k}(t) \Psi_{k}(t),
$$
where $\Psi_{k}(t)$ is given by 
$$
\Psi_{k}(t) = (\overbrace{0,0,...,0}^{k-1},1,
L^{[a,t]}_{\omega_{k+1}},L^{[a,t]}_{\omega_{k+1} \omega_{k+2}},...
L^{[a,t]}_{\omega_{k+1} \omega_{k+2}...\omega_{n}})^{T}. 
$$
By noticing that for $k=2,...n-1$, $\Psi_{k}(t)$ is the $(k+1)$th column vector  
of the matrix  
$I^{[a,t]}(\omega_1,\omega_2,...,\omega_n)$ defined in (\ref{chenmatrix}), 
while the first and the second column of 
$I^{[a,t]}(\omega_1,\omega_2,...,\omega_n)$ are $\phi(t)$ and $\phi'(t)/\omega_1(t)$ 
respectively, we conclude that     
\begin{equation*}
\det(\phi(t),\phi'(t),....,\phi^{n-1}(t)) = 
\omega_{1}^{n-1}(t) \omega_{2}^{n-2}(t)...\omega_{n-1}(t) 
\det(I^{[a,t]}(\omega_1,\omega_2,...,\omega_n)).
\end{equation*}
The proof then results from Proposition \ref{determinantformulas}.
\end{proof}
%%%%
%%%%

\noindent Let us define the following set of functions $\Phi_i,$ $i=1,...,n$ by 
$$
\Phi_1(t) = \left(L_{\omega_1 \omega_2}^{[a,t]},L_{\omega_1 \omega_2 \omega_3}^{[a,t]},...,
L_{\omega_1 \omega_2 ...\omega_n}^{[a,t]}\right)^{T},
$$
and for $i=2,...,n-1$ 
$$
\Phi_i(t) = \left(L_{\omega_1}^{[a,t]},...,
L_{\omega_1 \omega_2... \omega_{i-1}}^{[a,t]},
L_{\omega_1 \omega_2... \omega_{i+1}}^{[a,t]},
L_{\omega_1 \omega_2 ...\omega_n}^{[a,t]}\right)^{T},
$$
and 
$$
\Phi_n(t) = \left(L_{\omega_1}^{[a,t]},L_{\omega_1 \omega_2}^{[a,t]},...,
L_{\omega_1 \omega_2 ...\omega_{n-1}}^{[a,t]}\right)^{T}.
$$
We have 
%%%%%%%%%%
%%%%%%%%%%
\begin{proposition}\label{determinantgeneral2}
For $i=1,...,n-1$, we have 
$$
\det\left(\Phi_i'(t),...,\Phi_i^{(n-1)}(t)\right) =
\omega_1^{n-1}(t)...\omega_{n-1}(t) 
L^{[a,t]}_{\omega_{n}...\omega_{i+1}}, 
$$
and
$$ 
\det\left(\Phi_n'(t),...,\Phi_n^{(n-1)}(t)\right) = 
\omega_1^{n-1}(t)...\omega_{n-1}(t).
$$
\end{proposition}
%%%%%%%%%%
%%%%%%%%%%
\begin{proof}
We will show the proposition by induction on the index $i$. 
Let us start with the determinant formula for $\Phi_1$. We have 
$\Phi_1'(t) = \omega_1(t) \Omega_1(t)$, where the function $\Omega_1$ 
is given by 
$$
\Omega_1(t) =  (L_{\omega_2}^{[a,t]},L_{\omega_2 \omega_3}^{[a,t]},...,
L_{\omega_2 \omega_3 ...\omega_n}^{[a,t]})^{T}.
$$
Therefore, we have 
$$
\det\left(\Phi_1'(t),...,\Phi_1^{(n-1)}(t)\right) =
\omega_1^{n-1}(t)\det(\Omega_1(t),\Omega_1'(t), ...,\Omega_1^{(n-2)}).
$$
Applying Proposition \ref{determinantgeneral} to the Chebyshev function $\Omega_1$ gives
$$ 
\det(\Omega_1(t),\Omega_1'(t), ...,\Omega_1^{(n-2)}) = 
\omega_{2}^{n-2}(t)...\omega_{n-1}(t) 
L_{\omega_n \omega_{n-1} ...\omega_2}^{[a,t]}.
$$
Therefore, we have shown the proposition for $\Phi_1$. Let us assume 
the proposition to be true for any $j$ such that $1 \leq j < i$. 
We have $\Phi_i'(t) = \omega_1(t) \Omega_i(t)$, where $\Omega_i$ is given by 
$$
\Omega_i(t) =  (1,L_{\omega_2}^{[a,t]},...L_{\omega_2... \omega_{i-1}}^{[a,t]},
L_{\omega_2 ....\omega_{i+1}}^{[a,t]},...,L_{\omega_2 ....\omega_{n}}^{[a,t]})^{T}.
$$
Therefore, we have 
\begin{equation}\label{local1}
\det\left(\Phi_i'(t),...,\Phi_i^{(n-1)}(t)\right) =
\omega_1^{n-1}(t)\det \left(\Omega_i(t),\Omega_i'(t), ...,\Omega_i^{(n-2)}\right).
\end{equation}
Expanding the determinant $\det(\Omega_i(t),\Omega_i'(t), ...,\Omega_i^{(n-2)})$
down the first row, shows that 
$
\det(\Omega_i(t),\Omega_i'(t), ...,\Omega_i^{(n-2)}) = \det(\rho_1'(t),...,\rho_1^{(n-2)}(t))
$
where $\rho_1$ is given by 
$$
\rho_1(t) = (L_{\omega_2}^{[a,t]},...L_{\omega_2... \omega_{i-1}}^{[a,t]},
L_{\omega_2 ....\omega_{i+1}}^{[a,t]},...,L_{\omega_2 ....\omega_{n}}^{[a,t]})^{T}.
$$
By the induction hypothesis, we have 
$$
\det(\rho_1'(t),...,\rho_1^{(n-2)}(t)) = 
\omega_2^{n-2}(t)...\omega_{n-1}(t) 
L^{[a,t]}_{\omega_{n}...\omega_{i+1}}. 
$$
Inserting the result of the last equation into (\ref{local1}) leads to the desired 
result.
\end{proof}

\noindent 
Let $\phi$ be the Chebyshev function of order $n$ on an interval $I$ 
defined in (\ref{chebyshevfunction}). Let us denote by  $\phi^{*}(t)$ the function 
$$
\phi^{*}(t) = \left(\phi^{*}_1(t), \phi^{*}_2(t),...,\phi^{*}_{n+1}(t) \right)^{T} =  \left(L^{[a,t]}_{\omega_{n}...\omega_{1}},
L^{[a,t]}_{\omega_{n}...\omega_{2}},L^{[a,t]}_{\omega_{n}},1 \right)^{T}. 
$$
Using the notation 
$$
D(f_1,...,f_n;x_1,...,x_n) = \det \left( f_{j}(x_k) \right); \quad 1 \leq j,k \leq n,
$$ 
we have the following expression of the Chebyshev blossom of the function $\phi$ 
%%%
%%%
\begin{theorem}\label{chebyshevblossomformulas}
For any pairwise distinct real numbers $u_1,...,u_n$ in the interval $I$,
the Chebyshev blossom of the function $\phi$ is given by 
$\varphi = (\varphi_1,...,\varphi_n)^{T}$, where $\varphi_i$ is given by  
\begin{equation}\label{blossomformula}
\varphi_{i}(u_1,...,u_n) = \frac{D(\phi^{*}_1,...,\phi^{*}_{i},\phi^{*}_{i+2},...,
\phi^{*}_{n+1};u_1,...,u_n)}{D(\phi^{*}_2,\phi^{*}_3,...,
\phi^{*}_{n+1};u_1,...,u_n)}.
\end{equation}  
\end{theorem}
%%%
%%%
\begin{proof}
From (\ref{linearsystem}), a point $X=(X_1,...,X_n)^{T}$ in $\mathbb{R}^{n}$
belongs to the intersection of the osculating flats of order $n-1$ at the points $\phi(u_i)$
if and only if $X$ satisfies the linear system 
\begin{equation*}
\det(X,\phi'(u_i),...,\phi^{(n-1)}(u_i) = \det(\phi(u_i),\phi'(u_i),...,\phi^{(n-1)}(u_i))
\quad 
i=1,2,...,n.
\end{equation*}
Using Proposition \ref{determinantgeneral} and Proposition \ref{determinantgeneral2},
the last linear system can be rewritten as 
$$
\sum_{j=1}^{n-1} (-1)^{j-1} L^{[a,u_i]}_{\omega_{n},...,\omega_{j+1}} X_{j}
+ (-1)^{(n-1)} X_{n}  
= L^{[a,u_i]}_{\omega_{n},...,\omega_{1}} \quad i=1,...,n.
$$
Therefore, the statement of the theorem is nothing 
but the Cramer rule for solving linear systems.
\end{proof}

If in Theorem \ref{chebyshevblossomformulas} some of the real numbers 
$u_i$ coincident, then we can compute the Chebyshev blossom from
(\ref{blossomformula}) by a straightforward iterative application of 
the l'H\^{o}pital's rule. 
%%%%%%%%%%%%%%%%%%%%%%%%%%%%%%%%%%%%%%%%%%%%%%%%%%%%%%%%%%%%%%%%%%%%%%%%%%%%%%%%%%%%%%%%%%%%%%%%%%
%%%%%%%%%%%%%%%%%%%%%%%%%%%%%%%%%%%%%%%%%%%%%%%%%%%%%%%%%%%%%%%%%%%%%%%%%%%%%%%%%%%%%%%%%%%%%%%%%%
\section{Young Diagrams and Schur Functions}
In this section, we fix notations and review some basic concepts 
in the theory of Schur functions. We will follow the standard Macdonald's
notations \cite{Macdonald}
%%%%%%%%%%%%%%%%%%%%%
\vskip 0.2 cm
\noindent{\bf{Schur functions: }}
A sequence of non-increasing non-negative integers 
\begin{equation}\label{partition}
(\lambda_1,\lambda_2,...,\lambda_{i},...), \quad 
\lambda_1 \geq \lambda_2 \geq ...\geq \lambda_i \geq ...
\end{equation}
containing only finitely many non-zero terms is called a 
{\it partition}. The total number of non-zero components, 
$l(\lambda)$, is called the length of the partition 
$\lambda$. We will always ignore the difference between two partitions 
that differ only in the number of their trailing zeros. The non-zero $\lambda_i$
of the partition in (\ref{partition}) will be called the {\it parts} of $\lambda$.
The {\it{weight}} $|\lambda|$ of a partition $\lambda$ is defined as the sum its parts 
i.e., $|\lambda| = \sum_{i=1}^{\infty} \lambda_i$. We will find it sometimes 
convenient to write a partition by the common notation that indicate 
the number of times each integer appears as a part in the partition, for example
we write the partition $\lambda = (4,4,4,3,3,1)$ as $\lambda = (4^{3},3^{2},1)$.    
Given a partition $\lambda$, the Schur symmetric function 
$S_{\lambda}(u_1,...,u_n)$, where $n \geq l(\lambda)$ is an element of 
the ring $\mathbb{Z}[u_1,...,u_n]$ defined as the ratio of two determinants 
\begin{equation}\label{schurdeterminant}
S_{\lambda}(u_1,...,u_n) = \frac{\det ( u_i^{\lambda_j + n - j} )_{
{1 \leq i, j \leq n}}}{\det ( u_i^{n - j} )_ 
{{1 \leq i, j \leq n}}}.
\end{equation}
The denominator on the right-hand side of (\ref{schurdeterminant})
is the Vandermonde determinant, equal to the product 
$$
V(u_1,...,u_n) = \prod_{1\leq i < j \leq n}(u_i - u_j).
$$
We will adopt the convention that $S_\lambda(u_1,...,u_n) \equiv 0$
if $l(\lambda) > n$. From the definition, the Schur function associated 
with the {\it empty} partition $\lambda = (0,...,0,..)$ 
is $S_{\lambda}(u_1,...,u_n) \equiv 1$. For the partition $\lambda = (r)$,
the Schur function $S_{\lambda}$ is the complete symmetric function $h_{r}$ i.e., 
$$
S_{(r)}(u_1,u_2,...,u_n) = h_{r}(u_1,...,u_n) = 
\sum_{i_1 \leq i_2 \leq ...\leq i_r} u_{i_1} u_{i_2}...u_{i_r},
$$ 
while for the partition $\lambda = (1^r)$ with $r \leq n$, 
the Schur function $S_{(1^r)}$ is given by the elementary symmetric function
$e_r$ i.e,
$$
S_{(1^{r})}(u_1,u_2,...,u_n) = e_{r}(u_1,...,u_n) = 
\sum_{i_1 < i_2 < ...< i_r} u_{i_1} u_{i_2}...u_{i_r}.
$$ 
A direct consequence of the definition is the following
\begin{equation}\label{schurplusones}
S_{(\lambda_1+1,\lambda_2+1,...,\lambda_n+1)}(u_1,...,u_n) = 
u_1 u_2 ...u_n S_{(\lambda_1,\lambda_2,...,\lambda_n)}(u_1,...,u_n).
\end{equation}
The Schur function $S_{\lambda}$ can be expressed in terms of the 
complete symmetric functions through the Jacobi-Trudi formula
\begin{equation}\label{jacobi-trudi}
S_{\lambda} = \det \left( h_{\lambda_i-i+j} \right)_{1 \leq i, j \leq n},
\end{equation} 
where we assume that $h_{m}\equiv 0$ if $m < 0$.
The {\it{conjugate}}, $\lambda'$, of a partition
$\lambda$ is the partition whose Young diagram is the transpose of the Young 
diagram of $\lambda$, equivalently $\lambda'_{i} = Card \{j | \lambda_j \geq i \}.$
Using the conjugate partition, the Schur function can be expressed in term of the 
elementary symmetric functions through the N\"agelsbach-Kostka formula
\begin{equation}\label{kostka}
S_{\lambda} = \det \left( e_{\lambda_i'-i+j} \right)_{1 \leq i, j \leq n},
\end{equation} 
where we assume that $e_{m}\equiv 0$ if $m < 0$.
Throughout this work, we will use the notation 
$$
S_{\lambda}(u_1^{m_1},u_2^{m_2},...,u_k^{m_k}),
$$ 
to mean the evaluation of the Schur function in which the argument $u_1$ 
is repeated $m_1$ times, the argument $u_2$ is repeated $m_2$ times and so on. 
%%%%%%%%%%%%%%%%%%%%%
\vskip 0.2 cm
\noindent{\bf{Combinatorial definition of Schur functions: }}The {\it Young diagram}
of a partition $\lambda$ is a sequence of $l(\lambda )$
left-justified row of boxes, with the number of boxes in the $i$th row being
$\lambda_i$ for each $i$. A box $x =(i,j)$ in the diagram of $\lambda$ 
is the box in row $i$ from the top and column $j$ from the left. For example 
the Young diagram of the partition $\lambda=(5,4,2)$ and the coordinate of its boxes  
are
$$
\lambda = (5,4,2) \qquad 
\newcommand{\ff}{\mbox{\small{(1,1)}}}
\newcommand{\fs}{\mbox{\small{(1,2)}}}
\newcommand{\ft}{\mbox{\small{(1,3)}}}
\newcommand{\ffo}{\mbox{\small{(1,4)}}}
\newcommand{\ffi}{\mbox{\small{(1,5)}}}
\newcommand{\sfs}{\mbox{\small{(2,1)}}}
\newcommand{\sss}{\mbox{\small{(2,2)}}}
\newcommand{\st}{\mbox{\small{(2,3)}}}
\newcommand{\sfo}{\mbox{\small{(2,4)}}}
\newcommand{\tf}{\mbox{\small{(3,1)}}}
\newcommand{\tss}{\mbox{\small{(3,2)}}}
{\Yvcentermath1 \Yboxdim{20pt} \young(\ff\fs\ft\ffo\ffi,\sfs\sss\st\sfo,\tf\tss)} 
$$  
A {\it semi-standard tableau} $T^{\lambda}$ with entries less or equal to $n$ 
is a filling-in the boxes of $\lambda$ with numbers from $\{ 1, 2, ..., n \}$
making the rows increasing when read from left to right and the column
strictly increasing when read from the top to bottom. We say that the shape of
$T^{\lambda}$ is $\lambda$. For each semi-standard tableau $T^{\lambda}$ of the shape
$\lambda$, we denote by $p_i$ the number of occurrence of the number $i$ in the
semi-standard tableau $T^{\lambda}$. The weight of $T^{\lambda}$ is then defined 
as the monomial
$$
u^{T^{\lambda}} = u_1^{p_1} u_2^{p_2} ...u_n^{p_n}. 
$$
For a given partition $\lambda$ of length at most $n$, 
the Schur function $S_{\lambda}(u_1,...,u_n)$
is given by
$$
S_{\lambda}(u_1,u_2,...,u_n) = \sum_{T^{\lambda}} u^{T^{\lambda}},
$$
where the sum run over all the semi-standard tableaux of shape $\lambda$ and entries 
at most $n$.

\begin{example}
Consider the partition $\lambda = (2,1)$ and $n = 3$. Then, the
Young diagram of $\lambda$ and the complete list of semi-standard tableaux 
of shape $\lambda$ are
$$
\young(11,2) \quad \young(11,3) \quad 
\young(12,3) \quad \young(13,2) \quad 
\young(12,2) \quad \young(13,3) \quad
\young(22,3) \quad \young(23,3)
$$
Therefore, the Schur function associated with the partition $\lambda$
is given by 
$$
S_{\lambda}(u_1,u_2,u_3) = u_1^2 u_2 + u_1^2 u_3 + 2 u_1 u_2 u_3 +
u_2^2 u_3 + u_2 u_3^2 + u_1 u_2^2 + u_1 u_3^2.
$$
\end{example}
%%%%%%%%%%%%%%%%%%%%%
\vskip 0.2 cm
\noindent{\bf{Giambelli formula: }}
The Young diagram of a partition $\lambda$ is said to be a \textit{hook diagram} 
if the partition $\lambda$ is of the shape $\lambda = (p+1,1^{q})$ i.e., 
%%%%%%%%%%%%%%%%%%%%%%%%%%
%%%%%%%%%%%%%%%%%%%%%%%%%%
\vskip 1cm \hskip 4.277cm 
$q \left\{\Yvcentermath1 {\yng(1,1,1,1)}\right.$
\vskip -2.55 cm \hskip 4.82cm  
$\overbrace{\Yvcentermath1 {\yng(5)}}^{p+1}$ 
\vskip 1.8 cm
%%%%%%%%%%%%%%%%%%%%%%%%%
%%%%%%%%%%%%%%%%%%%%%%%%%
In Frobenius notation, we write the partition $\lambda$ as $(p|q)$.
Expanding the Jacobi-Trudi formula (\ref{jacobi-trudi}) along the top row,
shows that the Schur function associated with the partition $(p|q)$ is given by 
\begin{equation}\label{hook}
S_{(p|q)} = h_{p+1} e_{q} - h_{p+2}e_{q-1} + .... + (-1)^{q} h_{p+q+1}.
\end{equation}
Any partition $\lambda$ can be represented in Frobenius notation as 
\begin{equation}\label{frobenius}
\lambda = ( \alpha_1, ..., \alpha_r|\beta_1,..., \beta_r),
\end{equation}
where $r$ is the number of boxes in the main diagonal of the Young 
diagram of $\lambda$ and for $i=1,...,r$, $\alpha_i$ (resp. $\beta_i$)
is the number of boxes in the $i$th row (resp. the $i$th column) of 
$\lambda$ to the right of $(i,i)$ (resp. below $(i,i)$). For example 
the partition $\lambda = (6,4,2,1^{2})$, depicted below,
can be written in Frobenius notation as $\lambda = (5,2|4,1)$ 
\newcommand{\bbbox}{\blacksquare}
\newcommand{\nth}{\mbox{}}
$$
\lambda = {\Yvcentermath1 \young(\bbbox\nth\nth\nth\nth\nth,\nth\bbbox\nth\nth,\nth\nth,\nth,\nth)} 
$$ 
With the decomposition (\ref{frobenius}) of $\lambda$ in hook diagrams,
the Giambelli formula states that 
\begin{equation}\label{giambelli}
S_{\lambda} = \det ( S_{\left( \alpha_i | \beta_j \right)} )_{1\leq
i,j \leq r}
\end{equation}
We will adopt the convention that $S_{(\alpha|\beta)} \equiv 0$ if $\alpha$ or $\beta$
are negatives.
%%%%%%%%%%%%%%%%%%%%%
\vskip 0.2 cm
\noindent{\bf{Hook length formula: }}
The {\it {hook-length}} of a partition $\lambda$ at a box $x = (i,j)$ is defined
to be $h(x) = \lambda_i + \lambda'_{i} - i - j + 1$, where $\lambda'$ is the conjugate
partition of $\lambda$. In other word 
the hook-length at the box $x$ is the number of boxes that are in the same row 
to the right of it plus those boxes in the same column below it,
plus one (for the box itself). The {\it content} of the partition $\lambda$ at the box 
$x = (i,j)$ is defined as $c(x) = j-i$. The hook-length and the content of every box 
of the partition $\lambda = (5,4,2)$ is given as  
\newcommand{\negone}{\mbox{-1}}
\newcommand{\negtwo}{\mbox{-2}}
$$ 
h(\lambda) = {\Yvcentermath1 \young(76431,5421,21)} 
\qquad
Content(\lambda) = {\Yvcentermath1 \young(01234,\negone 012,\negtwo \negone)} 
$$
With these notations, the number of semi-standard 
tableaux of shape $\lambda$ with entries at most $n$ is given by 
the so-called hook-length formula as 
\begin{equation}\label{hooklengthformulas}
f_{\lambda}(n) = S_{\lambda}(\overbrace{1,1,...,1}^{n}) = \prod_{x \in \lambda}\frac{n+c(x)}{h(x)}.     
\end{equation}
In particular, we have the following useful hook-length formulas 
\begin{equation}\label{normalization1}
f_{(1^r)}(n) = \binom{n}{r}, \quad f_{(r)}(n) = \binom{n+r-1}{r} 
\end{equation}
and 
\begin{equation}\label{normalization2}
f_{(p|q)}(n) = \frac{n}{p+q+1} \binom{n+p}{p} \binom{n-1}{q}.
\end{equation}
We will adopt the convention that for every integer $n$, the hook-length 
of the empty partition $\lambda = (0,0,...)$ is given by  
$f_{\emptyset}(n) = 1$.
%%%%%%%%%%%%%%%%%%%%%
\vskip 0.2 cm
\noindent{\bf{Skew Schur functions and Branching rule: }}
Given two partitions, $\lambda$ and $\mu$, such that $\mu \subset \lambda$ i.e.,
$\mu_i \leq \lambda_i$, $i \geq 1$, a Young diagram with skew shape $\lambda/\mu$ is the
Young diagram of $\lambda$ with the Young diagram of $\mu$ removed from its upper left-hand
corner. Note that the standard shape $\lambda$ is just the skew shape $\lambda/\mu$ 
with $\mu = \emptyset$. For example, we have 
$$
\newcommand{\nothings}{\mbox{}}
(4,3,1)/(2,1) ={\Yvcentermath1 
\young(::\nothings\nothings,:\nothings\nothings,\nothings)}
$$
The skew Schur function $S_{\lambda/\mu}$ is defined as 
$$
S_{\lambda/\mu}(u_1,u_2,...,u_n) = \sum_{T^{\lambda/\mu}} x^{T^{\lambda/\mu}}
$$
where the sum run over all the semi-standard tableaux of shape $\lambda/\mu$
and entries at most $n$. Skew Schur functions have a determinant expression as
$$
S_{\lambda/\mu} = det(h_{\lambda_i -\mu_j - i + j})_{1\leq i,j \leq n}
$$ 
Using the skew Schur functions, we have the following branching rule
\begin{equation}
S_{\lambda}(u_1,...,u_{j},u_{j+1},...,u_{n}) = 
\sum_{\mu \subset \lambda} S_{\mu}(u_1,...,u_j) 
S_{\lambda/\mu}(u_{j+1},...,u_{n}).
\end{equation} 
Particularly interesting for this work, the following two branching rules 
\begin{equation}\label{branchingrule1}
S_{\lambda}(u_1,...,u_{n-1},u_n) = 
\sum_{\mu \prec \lambda} S_{\mu}(u_1,...,u_{n-1}) u_{n}^{|\lambda|-|\mu|},
\end{equation}
where the sum is over are the interlacing partitions $\mu$ i.e., partition 
$\mu = (\mu_1,...,$ $\mu_{n-1})$ such that
\begin{equation}
\lambda_1 \geq \mu_1 \geq \lambda_2 \geq ...\mu_{n-1} \geq \lambda_n,  
\end{equation}
and 
\begin{equation}\label{branchingrule2}
S_{\lambda}(u_1,...,u_{n-1},u_n) = 
\sum_{j=0}^{\lambda_1} S_{\lambda/{(j)}}(u_1,...,u_{n-1}) u_{n}^{j}.
\end{equation}
%%%%%%%%%%%%%%%%%%%%%%%%%%%%%%%%%%%%%%%%%%%%%%%%%%%%%%%%%%%%%%%%%%%%%%%%%%%%%%%%%%%%%%%%%%%%%%%%%%
%%%%%%%%%%%%%%%%%%%%%%%%%%%%%%%%%%%%%%%%%%%%%%%%%%%%%%%%%%%%%%%%%%%%%%%%%%%%%%%%%%%%%%%%%%%%%%%%%%
\section{Blossom in \muntz space with positive integer powers}
It is well known that for any positive real numbers $0 < s_1 < s_2 < ...< s_n$, 
the function $\phi(t) = (t^{s_1},t^{s_2},...,t^{s_n})^{T}$ is a 
Chebyshev function of order $n$ on the interval $]0,\infty[$.
In this section, we give the Chebyshev blossom of the function $\phi$ in case the
parameters $s_i, i=1,...,n$ are positive integers. We will first associate the sequence 
$(s_1,...,s_n)$ with a partition $\lambda$ that will allow us to give the expression 
of the blossom in terms of Schur functions. We will first start with a definition  

\begin{definition}
Let $\lambda = (\lambda_1,\lambda_2,...,\lambda_n)$ be a partition 
of length $l(\lambda)$ at most $n$. The \muntz tableau associated to the 
partition $\lambda$ is given by a sequence of $n+1$ partitions 
$(\lambda^{(0)},\lambda^{(1)},\lambda^{(2)},...,\lambda^{(n)})$
defined as follows:   
$$
\lambda^{(0)} =(\lambda_2,\lambda_3,...,\lambda_n),
$$
for $i=1,2,...n-1$
$$
\lambda^{(i)} = (\lambda_1 + 1,\lambda_2 + 1,...,\lambda_{i}+1,
\lambda_{i+2},...,\lambda_n)
$$
and 
$$
\lambda^{(n)} = (\lambda_1+1,\lambda_2+1,...,\lambda_n+1).
$$
\end{definition}

To remember the construction of the \muntz tableau we can remark that 
the partition $\lambda^{(0)}$ is obtained form the partition $\lambda$ 
by deleting the first row. The partition $\lambda^{(0)}$ will play an important role
in this work and will be called {\it {the bottom partition of $\lambda$}}.
The partition $\lambda^{(i)}$ is obtained by adding 
a box to the first $i$ rows of the partition $\lambda$, deleting the $i+1$ row 
and keeping all the other rows the same.
\begin{example}
The \muntz tableau associated with the partition $\lambda = (4,2)$ and $n =3$ 
is depicted as
$$  
\lambda = {\Yvcentermath1 \yng(4,2)} \quad 
\lambda^{(0)} = {\Yvcentermath1 \yng(2)} \quad
\lambda^{(1)} = {\Yvcentermath1 \yng(5)} \quad
\lambda^{(2)} = {\Yvcentermath1 \yng(5,3)} \quad
$$
$$
\lambda^{(3)} = {\Yvcentermath1 \yng(5,3,1)}
$$   
\end{example}       
To a given partition $\lambda = (\lambda_1,...,\lambda_n)$ of length at most $n$,
we define the following Chebyshev function of order $n$ 
\begin{equation} \label{chebyshevconvention}
\phi(t) = \left( t^{\lambda_1 - \lambda_2 + 1}, t^{\lambda_1 - \lambda_3 + 2},...,
t^{\lambda_1 - \lambda_n + (n-1)}, t^{\lambda_1+ n} \right). 
\end{equation}
The associated Chebyshev space $\mathcal{E}(\phi)$ will be denoted by
$\mathcal{E}_{\lambda}(n)$ and will be called the \muntz space associated with the partition
$\lambda$. The function $\phi$ will be called the \muntz function associated with $\lambda$ 
and conversely, the partition $\lambda$ will be called the partition associated with the function 
$\phi$. We have the following

\begin{theorem}\label{theoremblossom}
For any sequence $( u_1, u_2, ..., u_n ) \in ] 0, + \infty
\lbrack^n$, the blossom $\varphi = (\varphi_1,\varphi_2,...,\varphi_n)^{T}$
of the Chebyshev curve $\phi$ given in (\ref{chebyshevconvention}) is given by
$$
\varphi_i ( u_1, u_2, ..., u_n ) = \frac{f_{\lambda^{(0)}} ( n ) S_{\lambda^{(i)}}(u_1,
u_2,..., u_n)}{f_{\lambda^{(i)}}(n) S_{\lambda^{(0)}}(u_1,u_2,...,u_n)},
$$
where $(\lambda^{(0)},\lambda^{(1)},...,\lambda^{(n)})$ is the \muntz tableau associated 
with the partition $\lambda$ and $f_{\mu}(n)$ refers to the number of semi-standard
tableaux of shape $\mu$ and entries at most $n$.
\end{theorem}

\begin{proof} 
We first assume that all the positive real numbers $u_i, i=1,...,n$
are pairwise distinct. Consider the functions 
$\omega_1, \omega_2,...,\omega_n$  
such that for $i=1,2,...,n-1$ 
\begin{equation}\label{omegafunctions}
\phi_i(t) = L_{\omega_1 \omega_2 ...\omega_k}^{[0,t]} = t^{\lambda_1 - \lambda_{i+1} + i}
\quad \textnormal{and} \quad 
\phi_n(t) = L_{\omega_1 \omega_2 ...\omega_n}^{[0,t]} = t^{\lambda_1+n}.
\end{equation}
Applying successive derivatives to (\ref{omegafunctions}) 
shows that there exist positive constants $K_i, i=1,...,n$ such that 
$$
\omega_i(t) = K_i t^{\lambda_i - \lambda_{i+1}}, 
\quad \textnormal{for} \quad i=1,...,n-1
\quad  \textnormal{and} \quad
\omega_n(t) = K_n t^{\lambda_n}. 
$$
Computing the Chen iterated integrals of the obtained function $\omega_i, i=1,...,n$,
shows that there exist constants $C_k, k=1,...,n$ such that 
$$
L_{\omega_n \omega_{n-1}...\omega_{k}}^{[0,t]} = C_i t^{\lambda_k + (n-k+1)}. 
$$
From Theorem \ref{chebyshevblossomformulas}, the Chebyshev blossom of the function 
$\phi$ can be expressed as 
\begin{equation}\label{localdeterminant}
\varphi_i(u_1,...,u_n) = C^{'}_{i} \frac{det(u_i^{\rho_{j}+n-j})_{1\leq i,j\leq n}}
{det(u_i^{\lambda_{j+1}+n-j})_{1\leq i,j\leq n}}, 
\end{equation}  
where
$$ 
C'_{i}, i=1,...,n \quad {\textnormal{are constants and}} \quad
\rho = (\lambda_1 + 1,...,\lambda_{i} + 1, \lambda_{i+2},...,\lambda_n) = \lambda^{(i)}.
$$
Dividing both the numerator and the denominator of the right hand side of 
(\ref{localdeterminant}) by the Vandermonde determinant 
$\prod_{1 \leq i <j \leq n}(u_i - u_j)$ leads to 
\begin{equation}\label{localschur}
\varphi_i(u_1,u_2,...,u_n) = C^{'}_{i} \frac{S_{\lambda^{(i)}}(u_1,...,u_n)}
{S_{(\lambda_2,...,\lambda_n)}(u_1,...,u_n)}.
\end{equation}
Now, as the expression (\ref{localschur}) still make sense even if some of 
the $u_i$ coincident, and since the process of intersecting osculating flat
is a smooth process, the Chebyshev blossom of the function $\phi$ evaluated at any positive
real numbers $u_i, i=1,...,n$ is still given by the expression (\ref{localschur}).
The value of the constants $C'_i$ in (\ref{localschur}) can be obtained as follows : 
From the definition of the Chebyshev blossom, we have $\varphi(t,t,...,t) = \phi(t)$,
then in particular we have $\varphi(1,1,...,1) = (1,1,...,1)$,
which gives the value of the constants as claimed by the Theorem.   
\end{proof}
 
%%%%%%%%%%%%%%%%%%%%%%%% EXAMPLES %%%%%%%%%%%%%%%%%%%%%%%%%%%%
%%%%%%%%%%%%%%%%%%%%%%%%%%%%%%%%%%%%%%%%%%%%%%%%%%%%%%%%%%%%%%
%%%%%%%%%%%%%%%%%%%%%
\vskip 0.2 cm
\noindent{\bf{Examples section:}}
Horizontal, vertical and hook Young diagrams occupy an important place
in the combinatorics of Schur functions. Therefore, it is only natural
to define the \muntz spaces associated with these particular Young diagrams
and carry throughout this work their fundamental properties. 
Some times, we will also give low order \muntz spaces  to exhibit 
the use of the combinatorics of Schur functions in solving particular problems. 
We will also define the staircase \muntz space as they have the particularity of 
being, in a sense to be precised, a ``reparametrization'' of the polynomial spaces.     
%%%%%%%%%%%%%%
\vskip 0.2 cm
%%%%%%%%%%%%%%
\noindent{\bf{Polynomial M\"untz space: }}
Consider the Chebyshev curve of order $n$ over the real line $\mathbb{R}$. 
\begin{equation}\label{polymuntz}
\phi(t)=(t,t^2,...,t^n)^{T}.
\end{equation}
The associated partition $\lambda$ is the empty partition and  
the space $\mathcal{E}_{\emptyset}(n)$ is the linear space of 
polynomials of degree $n$. The bottom partition $\lambda^{(0)}$ 
is also an empty partition, while the rest of the \muntz tableau 
is given by $\lambda^{(k)} = (1^k), k=1,...,n.$
Therefore, the Chebyshev blossom of 
$\varphi =(\varphi_1,\varphi_2,...,\varphi_n)^{T}$ of the function $\phi$
is given by  
$$
\varphi_k ( u_1, u_2, ..., u_n ) = \frac{S_{(1^k)}(u_1,...,u_n)}{f_{(1^k)}(n)} = 
\frac{e_{k}(u_1,...,u_n)}{\binom{n}{k}}.
$$

%%%%%%%%%%%%%%%%%%%%%%%%%%%%%%%%%%%%%%%%%%%%%%%%%%%%%%%%%%%%%%%%
%%%%%%%%%%%%%%
\vskip 0.2 cm
%%%%%%%%%%%%%%
\noindent{\bf{Combinatorial M\"untz space: }} 
Consider the Chebyshev function $\phi(t) = (t,t^2,t^4)^{T}$ of order
$3$ over the interval $]0,\infty[$.
The partition $\lambda$ associated with the curve $\phi$ is given 
by $\lambda = (1,1,1)$. The \muntz tableau associated with $\lambda$ is given
by $\left( \lambda^{(0)} = (1,1), \lambda^{(1)} = (2,1,0), 
\lambda^{(2)} = (2,2,0), \lambda^{(3)} = (2,2,2) \right)$  
$$
\lambda = {\Yvcentermath1 {\yng(1,1,1)}} \quad
\lambda^{(0)} = {\Yvcentermath1 {\yng(1,1)}} \quad 
\lambda^{(1)} = {\Yvcentermath1 {\yng(2,1)}} \quad 
\lambda^{(2)} = {\Yvcentermath1 {\yng(2,2)}} \quad 
\lambda^{(3)} = {\Yvcentermath1 {\yng(2,2,2)}}.
$$
Therefore, the blossom $\varphi =(\varphi_1,\varphi_2,\varphi_3)^{T}$
of the function $\phi$ is given by
$$
\varphi_k(u_1,u_2,u_3) = \frac{f_{\lambda^{(0)}}(3)}{f_{\lambda^{(k)}}(3)}
\frac{S_{\lambda^{(k)}}(u_1,u_2,u_3)}
{S_{\lambda^{(0)}}(u_1,u_2,u_3)}. 
$$
We can now proceed by computing the Schur functions associated
with the partitions in the \muntz tableau. For the
partition $\lambda^{(0)}=(1,1)$, we have 
$$
S_{\lambda^{(0)}}(u_1,u_2,u_3) = e_{2}(u_1,u_2,u_3).
$$
The Schur function associated with the partition $\lambda^{(1)}=(2,1)$
has been already computed in Example 1. The semi-standard tableaux associated with
the partition $\lambda^{(2)} = (2,2)$ and entries at most $3$ are given by 
$$
\young(11,22) \quad \young(11,23) \quad \young(11,33) \quad 
\young(12,23) \quad \young(12,33) \quad \young(22,33)
$$
Therefore, we have 
$$
S_{(2,2)}(u_1,u_2,u_3) = u_1^2 u_2^2 + u_1^2 u_2 u_3 + u_1^2 u_3^2
+ u_2^2 u_1 u_3 + u_3^2 u_2 u_1 + u_2^2 u_3^2.
$$
For the partition $\lambda^{(3)} = (2,2,2)$, we can use (\ref{schurplusones}) 
to deduce that 
$$
S_{(2^3)}(u_1,u_2,u_3) = u_1u_2u_3 S_{(1^3)}(u_1,u_2,u_3)
= u_1u_2u_3 e_3(u_1,u_2,u_3) = u_1^2u_2^2u_3^2.
$$
Therefore, the blossom $\varphi$ of the function $\phi$ is given by 

\begin{equation*}
\tiny{
\varphi(u_1,u_2,u_3) = \frac{1}{8(u_1u_2+u_1u_3+u_2u_3)} 
\begin{pmatrix}
3(u_1+u_2)(u_1+u_3)(u_2+u_3)\\
4\left((u_1^2u_2^2 + u_1^2u_3^2 + u_2^2u_3^2) + u_1u_2u_3(u_1+u_2+u_3) \right)\\
24 u_1^2 u_2^2 u_3^2
\end{pmatrix}}
\end{equation*}
%%%%%%%%%%%%%%%%%%%%%%%%%%%%%%%%%%%%%%%%%%%%%%%%%%%%%%%%%%%%%%%%%
%%%%%%%%%%%%%%
\vskip 0.2 cm
%%%%%%%%%%%%%%
\noindent{\bf {Elementary M\"untz spaces: }} 
Let $k$ and $n$ be two positive integers such that $1 \leq k \leq n$.
Consider the Chebyshev curve of order $n$ over the interval $]0,\infty[$
defined for $k \neq 1$ by
$$ 
\phi(t)=(t,t^2,...,t^{k-1},t^{k+1},...,t^{n+1})^{T},
$$
and $\phi(t)=(t^2,t^3,...,t^{n+1})^{T}$ for $k=1$. 
The partition $\lambda$ associated with the function $\phi$ is 
given by a vertical Young diagram with $k$ boxes, i.e, $\lambda = (1^{k})$. 
For this reason, we will call the curve $\phi$ the $k$th
elementary \muntz curve and the space $\mathcal{E}_{(1^{k})}(n)$ 
the $k$th elementary \muntz space. The bottom partition $\lambda^{(0)}$
is given by $\lambda^{(0)} = (1^{k-1})$, with an associated Schur function
given by $S_{\lambda^{(0)}}(u_1,...,u_n) = e_{k-1}(u_1,...,u_n)$.
The other partitions in the \muntz tableaux are given by 
$$
\lambda^{(i)} = (2^{i},1^{k-i-1}) 
\quad \textnormal{for} \quad i=1,...,k-1 \quad 
\textnormal{and} \quad \lambda^{(i)} = (2^{k},1^{i-k})
\quad \textnormal{for} \quad i=k,...,n.
$$
The Young diagram of the partitions in the \muntz tableau are of 
the form
%%%%%%%%%%%%%%%%%
\vskip 0.5 cm 
$\lambda^{(0)} = (k-1) \left\{ {\Yvcentermath1 {\yng(1,1,1,1)}} \right. \qquad$  
For $i=1,...,k-1,$ we have    
$\lambda^{(i)} = i \left\{ {\Yvcentermath1 {\yng(2,2,2)}} \right.$
\vskip -0.273cm \hskip 7.618cm 
$(k-i-1) \left\{ {\Yvcentermath1 {\yng(1,1,1)}} \right.$

\noindent and for $i=k,...,n \; $   
$\lambda^{(i)} = k \left\{ {\Yvcentermath1 {\yng(2,2,2)}} \right.$
\vskip -0.11cm \hskip 2.7cm 
$(i-k) \left\{ {\Yvcentermath1 {\yng(1,1,1)}} \right.$
\bigskip

\noindent The conjugate of the partition $\lambda^{(i)}$ for $i=1,...,n$ are then given by 
$$
\lambda^{(i)}{'} = (k-1,i)
\quad \textnormal{for} \quad i=1,...,k-1 \quad \textnormal{and} \quad 
\lambda^{(i)}{'} = (i,k)
\quad \textnormal{for} \quad i=k,...,n.
$$
Therefore, Using Theorem \ref{theoremblossom} and the N\"agelsbach-Kostka formula 
(\ref{kostka}), the Chebyshev blossom $\varphi = (\varphi_1,\varphi_2,...,\varphi_n)^{T}$ of the 
function $\phi$ is given by 
$$
\varphi_i = \frac{\binom{n}{k-1}}{\binom{n}{k-1} \binom{n}{i} - \binom{n}{i-1}\binom{n}{k}}
\frac{e_{k-1}e_{i} - e_{i-1}e_{k}}{e_{k-1}} 
\quad \textnormal{for} \quad i=1,...,k-1, 
$$
and 
$$
\varphi_i = \frac{\binom{n}{k-1}}{\binom{n}{k} \binom{n}{i} - \binom{n}{i+1}\binom{n}{k-1}}
\frac{e_{k}e_{i} - e_{i+1}e_{k-1}}{e_{k-1}} 
\quad \textnormal{for} \quad i=k,...,n. 
$$
Of a particularly interesting form is the last component of $\varphi$ as we have 
$$
\varphi_n(u_1,...,u_n) = \frac{\binom{n}{k-1}}{\binom{n}{k}}
\frac{\left(\prod_{i=1}^{n} u_i \right) e_k(u_1,...,u_n)}{e_{k-1}(u_1,...,u_n)}.
$$ 
%%%%%%%%%%%%%%
\vskip 0.2 cm
%%%%%%%%%%%%%%
%%%%%%%%%%%%%%%%%%%%%%%%%%%%%%%%%%%%%%%%%%%%%%%%%%%%%%%%%%%%%%%%%%
\noindent{\bf {Complete M\"untz spaces: }}
Let $k$ be a non-negative integer and denote by $\phi$ the Chebyshev curve
of order $n$ over the interval $]0,\infty[$
$$
\phi(t) = (t^{k+1},t^{k+2},...,t^{k+n})^{T}.
$$
The partition associated with the curve $\phi$ is given 
by a horizontal Young diagram with $k$ boxes, i.e., $\lambda = (k)$.
We will call the function $\phi$ the $k$th complete \muntz function
and the associated space $\mathcal{E}_{(k)}(n)$
the $k$th complete \muntz space. The bottom partition 
$\lambda^{(0)}$ is an empty partition, while the other partitions 
in the \muntz tableau are given by $\lambda^{(i)} = (k|i-1)$.
Therefore, the Chebyshev blossom 
$\varphi = (\varphi_1,...,\varphi_n)^{T}$ of $\phi$ is given by 
$$
\varphi_i (u_1,u_2,...,u_n ) = \frac{S_{(k|i-1)}(u_1,u_2,...,u_n)}{f_{(k|i-1)}(n)},
$$
where $S_{(k|i-1)}$ can be expressed in term of the complete and 
elementary symmetric functions according to (\ref{hook}) and the 
normalization constants can be computed using equation
(\ref{normalization2}). 
Note that $\varphi_1 = h_{k+1}/\binom{n+k}{k+1}$, 
while $\varphi_n = e_n h_{k}/\binom{n+k-1}{k}$.   

%%%%%%%%%%%%%%
\vskip 0.2 cm
%%%%%%%%%%%%%%
%%%%%%%%%%%%%%%%%%%%%%%%%%%%%%%%%%%%%%%%%%%%%%%%%%%%%%%%%%%%%%%%%%%%%
\noindent{\bf {Hook M\"untz spaces: }}
Let $l$ and $n$ be two positive integers and let $k$ be a positive integer such that $k < n$.
Consider the Chebyshev curve $\phi$ of order $n$ over the interval 
$]0,\infty[$ given by
$$ 
\phi(t) = (t^{l+1},t^{l+2},...,t^{l+k},t^{l+k+2},...,t^{l + n + 1})^{T}.
$$
The partition $\lambda$ associated with the curve $\phi$ is given 
by a $(l,k)$-hook Young diagram, i.e., $\lambda = (l|k)$. Therefore,
the function $\phi$ will be called a $(l,k)$-hook \muntz function, while the
associated space  $\mathcal{E}_{(l|k)}(n)$ will be called  
the $(l,k)$-hook \muntz space. The bottom partition $\lambda^{(0)}$
is given by $\lambda^{(0)} = (1^{k})$, while the other partitions in the 
\muntz tableau are given by  
$$
\lambda^{(i)} = (l+2,2^{i-1},1^{k-i}) 
\quad \textnormal{for} \quad i=1,...,k \quad
$$
and
$$ 
\lambda^{(i)} = (l+2,2^{k},1^{i-k-1})
\quad \textnormal{for} \quad i=k+1,...,n.
$$
Every partition in the \muntz tableau 
has at most two boxes in its main diagonal, thereby, making Giambelli formula
(\ref{giambelli}) useful for the computation of the associated Schur functions.
In Frobenius notation, the partitions in the \muntz tableau are given by 
$$
\lambda^{(i)} = (l+1,0 | k-1,i-2) 
\quad \textnormal{for} \quad i=1,...,k \quad 
$$ 
and
$$
\lambda^{(i)} = (l+1,0|i-1,k-1)
\quad \textnormal{for} \quad i=k+1,...,n.
$$
Therefore, the Chebyshev blossom $\varphi = (\varphi_1,...,\varphi_n)^{T}$ of 
$\phi$ is given by
$$
\varphi_i = \frac{\binom{n}{k}}{f_{\lambda^{(i)}}(n)} \frac{e_{i-1} S_{(l+1|k-1)}-e_{k} S_{(l+1|i-2)}}{e_k}
\quad \textnormal{for} \quad i=1,...,k
$$ 
and 
$$
\varphi_i = \frac{\binom{n}{k}}{f_{\lambda^{(i)}}(n)} \frac{e_{k} S_{(l+1|i-1)}-e_{i} S_{(l+1|k-1)}}{e_k}
\quad \textnormal{for} \quad i=k+1,...,n,
$$ 
where the normalizing factors $f_{\lambda^{(i)}}(n)$ can be computed using equations 
(\ref{normalization1}) and (\ref{normalization2}). 
In particular, we have 
$$
\varphi_1 = \frac{\binom{n}{k}}{f_{(l+1|k-1)}(n)} \frac{S_{(l+1|k-1)}}{e_k}
\quad \textnormal{and} \quad 
\varphi_n = \frac{\binom{n}{k}}{f_{(l|k)}(n)} \frac{e_n S_{(l|k)}}{e_k}.
$$ 
%%%%%%%%%%%%%%
\vskip 0.2 cm
%%%%%%%%%%%%%%
%%%%%%%%%%%%%%%%%%%%%%%%%%%%%%%%%%%%%%%%%%%%%%%%%%%%%%%%%%%%%%%%%%%%
\noindent{\bf {Staircase M\"untz spaces: }}
Let $\phi = (\phi_1,\phi_2,...\phi_n)$ be a Chebyshev function of order $n$ 
on an non-empty interval $I$ and denote by $\varphi$ its Chebyshev blossom. Let 
$\theta : \tilde{I} \longrightarrow I$ be a $C^{\infty}$ strictly monotonic 
function. Then, the function 
\begin{equation}\label{reparametrizationfunction}
\tilde{\phi} = \phi \circ (\theta,\theta,...,\theta)
\end{equation}
is a Chebyshev function of order $n$ on the interval $\tilde{I}$. 
Moreover, as the process of intersecting osculating flats is a geometrical
concept depending only on the curve itself and not on the chosen parametrization,
the Chebyshev blossom $\tilde{\varphi}$ of the function 
$\tilde{\phi}$ is given by
\begin{equation}\label{reparametrization}
\tilde{\varphi} = \varphi \circ (\theta,\theta,....,\theta).
\end{equation}
For similar reasons, the pseudo-affinity factor $\tilde{\alpha}$ of the
space $\mathcal{E}(\tilde{\phi})$ is related to the pseudo-affinity factor $\alpha$
of the space $\mathcal{E}(\phi)$ by 
\begin{equation}\label{reparametrizationalpha}
\tilde{\alpha}(u_1,...,u_{n-1};a,b,t) =  
\alpha(\theta(u_1),...,\theta(u_{n-1});\theta(a),\theta(b),\theta(t)).
\end{equation}
Finally, if we denote by $B^{n}_{k}$ and $\tilde{B}^{n}_{k}, k=0,...,n$
the Chebyshev-Bernstein basis of the spaces $\mathcal{E}(\phi)$ and 
$\mathcal{E}(\tilde{\phi})$ respectively, then we have
\begin{equation}\label{reparametrizationbernstein}
\tilde{B}^{n}_{k}(t) = B^{n}_{k}(\theta(t)), \quad k=0,...,n.
\end{equation}
Now, we will deal with the simplest case of a situation such 
(\ref{reparametrizationfunction}), namely, a reparametrization of 
the space of polynomials. Let $l$ be a non-negative
integer and consider the Chebyshev curve $\phi$ of order $n$ over 
the interval $]0,\infty[$ given by 
$$
\phi(t) = (t^{l+1},t^{2(l+1)},...,t^{k(l+1)},...,t^{n(l+1)})^{T}.
$$
The partition $\lambda$ associated with the function $\phi$ is given
by the so-called $l$-staircase partition 
\begin{equation}\label{staircase}
\lambda = \left(nl, (n-1)l, (n-2)l,...,l\right).
\end{equation}
The function $\phi$ will be called a $l$-staircase \muntz function,
while the associated Chebyshev space will be called a $l$-staircase 
\muntz space. The function $\phi$ can be rewritten as 
$$
\phi(t) = ((t^{(l+1)})^1,(t^{(l+1)})^2,...,(t^{(l+1)})^k,...,(t^{(l+1)})^n)^{T}.
$$
Therefore, the function $\phi$ is a reparametrization of the \muntz polynomial 
function (\ref{polymuntz}). Taking the Chebyshev blossom of $\phi$ using 
Theorem \ref{theoremblossom} in one hand and equation (\ref{reparametrization})
in the another hand, in which $\theta(t) = t^{l+1}$ in (\ref{reparametrizationfunction}),
lead to a set of power plethysms
\begin{equation}\label{power}
\frac{e_{k}(u_1^{l+1},u_2^{l+1},...,u_n^{l+1})}{\binom{n}{k}} = 
\frac{f_{\lambda^{(0)}}}{\lambda^{(k)}} \frac{S_{\lambda^{(k)}}(u_1,u_2,...,u_n)}
{S_{\lambda^{(0)}}(u_1,u_2,...,u_n)}, 
\end{equation}
where $(\lambda^{(0)},...,\lambda^{(n)})$ is the \muntz tableau associated with
the partition $\lambda$ in (\ref{staircase}). Note that (\ref{power}) is not 
a genuine power plethysm as we do not expand the quantity in the left hand 
of (\ref{power}) in the Schur basis. In this work, our interest in the 
staircase \muntz spaces is motivated by two facts. The first, is that 
as their pseudo-affinity factors as well as their Chebyshev-Bernstein bases 
are well known, they will play a role of reconfirming our theoretical results.
The second fact is that, in practice, these spaces will play a sort of short-cut
in finding explicit expressions of the Chebyshev-Bernstein basis for a generic \muntz space.
A property of $l$-staircase Young diagram that will be needed later is the following expression of their 
associate Schur functions, namely for the partition $\lambda$ given in (\ref{staircase}), we 
have
\begin{equation}\label{staircaseschur}
S_{\lambda}(u_1,u_2,...,u_{n},u_{n+1}) = 
\prod_{1 \leq i < j \leq n+1} h_{l}(u_i,u_j) =
\prod_{1 \leq i < j \leq n+1} \frac{u_i^{l+1}-u_j^{l+1}}{u_i - u_j}.
\end{equation}
%%%%%%%%%%%%%%
\vskip 0.2 cm
%%%%%%%%%%%%%%
\begin{remark}
The definitions of elementary, complete and hook \muntz spaces in our previous examples 
depend primarily on the convention that we have adopted in associating a \muntz space 
to a partition in (\ref{chebyshevconvention}). 
However, as it will be clear, once we give the expressions of the pseudo-affinity factors
and the Chebyshev-Bernstein bases of these spaces, that the adopted convention 
is the most natural one.     
\end{remark}

\begin{remark}
Theorem \ref{theoremblossom} it true even if $\lambda = (\lambda_1,\lambda_2,...,\lambda_n)$ 
is such that $\lambda_1 \geq \lambda_2 \geq ...\geq \lambda_n  \geq 0$ and $\lambda_i$ are real numbers.
In this case, the Schur function should be defined only as the ratio of determinants as in 
(\ref{schurdeterminant}) and in which we make use of the l'H\^{o}pital's rule when some or all of the arguments
coincident. In the case the $\lambda_i$ are positive rational numbers, we can, in principle,
write the associated Chebyshev function as a composition of the form 
(\ref{reparametrizationfunction}) and in which the Chebyshev function $\phi$ is 
associated with a true partition. For example, the Chebyshev 
function $\phi(t) = (t^{\frac{1}{6}}, t^{\frac{1}{2}},t^{\frac{2}{3}})^{T}$ 
on the interval $]0,\infty[$ can be written as 
$\phi(t) = (t^{\frac{1}{6}},(t^{\frac{1}{6}})^{3} (t^{\frac{1}{6}})^{4})^{T}$. Therefore, we can 
use the remarks in examples section related to the staircase \muntz spaces 
to compute the blossom of the function $\phi$.       
\end{remark}

\begin{remark}
In the proof of Theorem \ref{theoremblossom}, we have decided to not to keep track of 
the exact value of the constants that naturally appears within the proof.
The main reason for this decision is the fact that we can always use the
diagonal coincidence property of the Chebyshev blossom to compute the final normalizing factors.
However, if we had kept track of the constants, we would have proven a formula for the ratio
of the hook-lengths. The fact that complete \muntz spaces have polynomials blossom would have then allow 
us to find a new proof for the hook-length formula (\ref{hooklengthformulas}) for the hook Young 
diagrams and then using the Giambelli formula, we would have proven a determinant
expression for the hook-length formula.  
\end{remark}

In the following, we would like to draw attention that the \muntz tableau 
associated with a partition $\lambda$ appears naturally in the expansion of the Jacobi-Trudi 
determinant (\ref{jacobi-trudi}). More precisely, let 
$\left( \lambda^{(0)},...,\lambda^{(n)} \right)$ be the \muntz tableau 
associated with a partition $\lambda$ of length at most $n$.
Computing the Schur function 
$S_{\lambda^{(n)}}$ using the Jacobi-trudi determinant (\ref{jacobi-trudi}) by
expanding the determinant up the last column \cite{Macdonald}, we find 
\begin{equation}\label{localexpansion}
S_{\lambda^{(n)}} = (-1)^{n-1} S_{\lambda^{(0)}} h_{\lambda_1+n} + 
\sum_{i=2}^{n} (-1)^{n-i} h_{\lambda_{i} + (n+1-i)} S_{\lambda^{(i-1)}}.
\end{equation}
Dividing (\ref{localexpansion}) by $S_{\lambda^{(0)}}$ and normalizing using the hook 
length factors $f_{\lambda^{(i)}}(n)$, we arrive at
%%%
%%%
\begin{proposition}
Let $\lambda = (\lambda_1,...,\lambda_n)$ be a partition of length 
at most $n$. Let $\varphi = (\varphi_1,...,\varphi_n)^{T}$ be 
the blossom of the Chebyshev function associated with the 
partition $\lambda$. Then we have 
$$
h_{\lambda_1+n} = \sum_{j=1}^{n} (-1)^{(j+1)} 
\frac{f_{\lambda^{(j)}}(n)}{f_{\lambda^{(0)}}(n)} 
h_{\lambda_{j+1}+n-j} \varphi_j,
$$
where $(\lambda^{(0)},...,\lambda^{(n)})$ is the \muntz tableau 
associated with the partition $\lambda$ and $\lambda_{n+1} = 0$.  
\end{proposition}
%%%%%%%%%%%%%%%%%%%%%%%%%%%%%%%%%%%%%%%%%%%%%%%%%%%%%%%%%%%%%%%%%%%%%%%%%%%%%%%%%%%%%%%%%%%%%%%%%%
%%%%%%%%%%%%%%%%%%%%%%%%%%%%%%%%%%%%%%%%%%%%%%%%%%%%%%%%%%%%%%%%%%%%%%%%%%%%%%%%%%%%%%%%%%%%%%%%%%
\section{The pseudo-affinity factor}
For a given partition $\lambda$ of length at most $n$,
we give an expression of the pseudo-affinity factor associated
with the \muntz space $\mathcal{E}_{\lambda}(n)$ in terms of 
Schur functions. The following so-called Dodgson condensation 
formula \cite{Kra} will be crucial to this end.  
%%%%%%%%%%%%%%%%%%%%%
\begin{proposition}\label{condensation}
Let $A$ be an $(n,n)$ matrix. Denote the submatrix of $A$ in which rows
$i_1, i_2,...,i_k$ and columns $j_1, j_2,...,j_k$ are omitted by 
$A_{i_1,i_2,...,i_k}^{j_1,j_2,...,j_k}$. Then we have
\begin{equation}\label{condensationformulas}
det(A) det(A_{1,n}^{1,n}) = det(A_{1}^{1}) det(A_{n}^{n}) - det(A_{1}^{n}) det(A_{n}^{1}).
\end{equation}
\end{proposition}
%%%%%%%%%%%%%%%%%%%%
From the last proposition, we can prove the following 
%%%%%%%%%%%%%%%%%%%%
\begin{proposition}\label{condensationproposition}
Let $\lambda = (\lambda_1,\lambda_2,...,\lambda_n)$ be a partition of length at 
most $n$. Then, for any sequence of real numbers $U = (u_1,u_2,...,u_{n-1})$ 
and real numbers $x,y$, we have 
$$
(x-y) S_{\lambda}(U,x,y) S_{\lambda^{(0)}}(U)  =  
x  S_{\lambda}(U,x) S_{\lambda^{(0)}}(U,y) - y  S_{\lambda}(U,y) S_{\lambda^{(0)}}(U,x),
$$
where $\lambda^{(0)}$ is the bottom partition of $\lambda$.
\end{proposition}
%%%%%%%%%%%%%%%%%%%
\begin{proof}
Without loss of generality, we can assume that all variables 
$u_i, i=1,...,n-1; x$ and $y$ are pairewise distinct.
Let us denote by $V_U$ the Vandermonde factor 
$$
V_U = \prod_{1\leq i<j \leq n-1} (u_i - u_j).
$$ 
Now, let us apply Proposition \ref{condensation} to the $(n+1,n+1)$ matrix $A$ defined as     
\begin{equation}\label{matrixA}
A = (a_{ij})_{1\leq i,j\leq n+1} = x_{i}^{\lambda_j+(n+1)-j}
\end{equation}
where $x_1 = x$, $x_i = u_{i-1}$ for $i=2,...,n, x_{n+1} = y$ 
and $\lambda_{n+1}=0$.
The following determinant formulas can be readily checked 
$$
det(A) = (x-y) V_U S_{\lambda}(U,x,y)  \prod_{i=1}^{n-1}(x-u_i)(u_i-y); \quad 
det(A_{1,n}^{1,n}) = V_U S_{\mu}(U)  \prod_{i=1}^{n-1}u_i 
$$  
$$
det(A_{1}^{1}) = V_U S_{\mu}(U,y) \prod_{i=1}^{n-1} (u_i-y)   ; \quad  
det(A_{n}^{n}) = x V_U S_{\lambda}(U,x) \prod_{i=1}^{n-1}u_i (x-u_i)   
$$
$$
det(A_{1}^{n}) =  y V_U S_{\lambda}(U,y) \prod_{i=1}^{n-1}u_i (u_i-y); \quad  
det(A_{n}^{1}) = V_U S_{\mu}(U,x) \prod_{i=1}^{n-1} (x-u_i) 
$$
Upon applying (\ref{condensationformulas}), the claim of the proposition 
follows.
\end{proof}

At this point, we can give a Schur function representation 
of the pseudo-affinity factor as follows 

\begin{theorem}\label{pseudotheorem}
The pseudo-affinity factor of the Chebyshev space $\mathcal{E}_{\lambda}(n)$
associated with a partition $\lambda$ of length at most $n$ is given by 
$$
\alpha(U;a,b,t) = ( \frac{t-a}{b-a} ) \frac{S_{\lambda}(U,a,t) S_{\lambda^{(0)}}(U,b)}
{S_{\lambda}(U,a,b) S_{\lambda^{(0)}}(U,t)}, 
$$
and 
$$
\beta(U;a,b,t) = 1 - \alpha(U,a,b,t) = ( \frac{b-t}{b-a} ) 
\frac{S_{\lambda}(U,b,t) S_{\lambda^{(0)}}(U,a)}
{S_{\lambda}(U,a,b) S_{\lambda^{(0)}}(U,t)}, 
$$
where $U$ is a sequence of positive real numbers $U = (u_1,...,u_{n-1})$ and 
$\lambda^{(0)}$ is the bottom partition of $\lambda$.
\end{theorem}

\begin{proof}
Let $\phi$ the \muntz function associated with the partition $\lambda$,
and $\varphi = (\varphi_1,...,\varphi_n)^{T}$ its Chebyshev blossom.
As the pseudo-affinity factor is independent of which $\mathcal{E}_{\lambda}(n)$ function
we choose, we can work with the last component $\varphi_n$ of the blossom 
\begin{equation}\label{phin}
\varphi_n(u_1,...,u_n) =  \frac{f_{\lambda^{(0)}}(n)}{f_{\lambda^{(n)}}(n)} 
\frac{ S_{\lambda}(u_1,...,u_n) \prod_{i=1}^n u_i}{S_{\lambda^{(0)}}(u_1,...,u_n)}.
\end{equation}
By equation (\ref{pseudoaffinity}) and if we denote $U = (u_1,u_2,...,u_{n-1})$, we have 
$$
\alpha(U;a,b,t) = \frac{\varphi_{n}(U,t)-\varphi_{n}(U,a)} 
{\varphi_{n}(U,b)-\varphi_{n}(U,a)}.
$$
Inserting (\ref{phin}) into  the last equation, leads to 
\begin{equation}\label{localpseudo}
\alpha(U;a,b,t) = \frac{t S_{\lambda}(U,t) S_{\lambda^{(0)}}(U,a) - 
a S_{\lambda}(U,a) S_{\lambda^{(0)}}(U,t)}{b S_{\lambda}(U,b) S_{\lambda^{(0)}}(U,a) - 
a S_{\lambda}(U,a) S_{\lambda^{(0)}}(U,b)} 
\frac{S_{\lambda^{(0)}}(U,b)}{S_{\lambda^{(0)}}(U,t)}.
\end{equation}
Applying Proposition \ref{condensationproposition} with $x=a$ and $y=t$ to 
(\ref{localpseudo}) leads to the desired expression for the 
pseudo-affinity factor. Similar treatment with $1-\alpha$ leads to the second 
equation of the proposition.
\end{proof}

The pseudo-affinity factor of the \muntz spaces defined in the examples section 
can be derived for the last proposition. For the \muntz polynomial space 
$\mathcal{E}_{\emptyset}(n)$, the partition $\lambda$ and it 
bottom partition $\lambda^{(0)}$ are empty and therefore by Theorem \ref{pseudotheorem},  
the pseudo-affinity factor is given by 
$$
\alpha(u_1,...,u_{n-1};a,b,t) = \frac{t-a}{b-a}.
$$
Similarly, the pseudo-affinity factor of the $k$th elementary
\muntz space $\mathcal{E}_{(1^k)}(n)$ is given by 
$$
\alpha(u_1,...,u_{n-1};a,b,t) = \frac{t-a}{b-a} 
\frac{e_{k}(u_1,...,u_{n-1},a,t) e_{k-1}(u_1,...,u_{n-1},b)}
{e_{k}(u_1,...,u_{n-1},a,b) e_{k-1}(u_1,...,u_{n-1},t)}.
$$
For the $k$th complete \muntz space $\mathcal{E}_{(k)}(n)$, we have 
$$
\alpha(u_1,...,u_{n-1};a,b,t) = \frac{t-a}{b-a} 
\frac{h_{k}(u_1,...,u_{n-1},a,t)}
{h_{k}(u_1,...,u_{n-1},a,b)}.
$$
For the $(l,k)$-hook \muntz space $\mathcal{E}_{(l|k)}(n)$, we have
$$
\alpha(u_1,...,u_{n-1};a,b,t) = \frac{t-a}{b-a} 
\frac{S_{(l|k)}(u_1,...,u_{n-1},a,t) e_{k}(u_1,...,u_{n-1},b)}
{S_{(l|k)}(u_1,...,u_{n-1},a,b) e_{k}(u_1,...,u_{n-1},t)}.
$$
Consider now the pseudo-affinity factor of the $l$-staircase \muntz space
associated with the partition $\lambda$ in (\ref{staircase}). 
Using the fact that 
\begin{equation}\label{staircaseformulas}
\begin{split}
S_{\lambda}(u_1,...,u_{n-1},u_{n},u_{n+1}) & = \hskip -0.6 cm
\prod_{1 \leq i < j \leq n+1} \hskip -0.6 cm h_{l}(u_i,u_j), \\
 S_{\lambda^{(0)}}(u_1,...,u_{n-1},u_{n}) & =  \hskip -0.4 cm
\prod_{1 \leq i < j \leq n}  \hskip -0.3 cm
h_{l}(u_i,u_j), 
\end{split}
\end{equation}
and carrying out all the simplifications that appear in the computation
of the pseudo-affinity factor, we find  
$$
\alpha(u_1,...,u_{n-1};a,b,t) = \frac{t-a}{b-a} 
\frac{h_{l}(a,t)}{h_{l}(a,b)} = \frac{t^{l+1} - a^{l+1}}{b^{l+1} - a^{l+1}},
$$
which is what is expected from the relation (\ref{reparametrizationalpha}).
%%%%%%%%
\vskip 0.2 cm
%%%%%%%%
For later use, we will need the equivalent of Proposition \ref{condensationproposition}, 
for every partition $\lambda^{(k)}$ in the \muntz tableau of the partition $\lambda$.

\begin{proposition}\label{condensationmuntztableau}
Let $\lambda$ be a partition of length at most $n$ and let
$(\lambda^{(0)},\lambda^{(1)},...,\lambda^{(n)})$ its \muntz tableau.
Then, for any real numbers $U=(u_1,...,u_{n-1})$, real numbers $x$ and $y$,
and $k=1,...,n-1$, we have
$$
S_{\lambda^{(0)}}(U,x) S_{{\lambda}^{(k)}}(U,y) - 
S_{\lambda^{(0)}}(U,y) S_{{\lambda}^{(k)}}(U,x) =
(y-x) S_{\eta}(U) S_{\lambda}(U,x,y),
$$
where $\eta$ is the bottom partition of $\lambda^{(k)}$ i.e., $\eta$ is the partition 
$\eta = (\lambda_{2}+1,\lambda_{3}+1,...,\lambda_{k}+1,\lambda_{k+2},...,\lambda_{n}).$
\end{proposition} 
\begin{proof} 
We can, without loss of generality, assume that all the variables $u_i, i=1,...,n-1$, 
$x$ and $y$ are pairewise distinct. Consider the $(n+1,n+1)$ matrix $A$
defined in (\ref{matrixA}). Now, construct a matrix $B_k$ by putting the first column 
of the matrix $A$ as the last column and putting the $(k+1)$th column of the matrix $A$ 
as the first column. The proof of the proposition is then derived by applying
the condenstation formula (\ref{condensationformulas}) to the matrix $B$.
\end{proof}

\begin{remark}
Note that Proposition \ref{condensationmuntztableau} can be used to reconfirm
the fact that the pseudo-affinity factor associated with a \muntz space 
$\mathcal{E}_{\lambda}(n)$ can be computed from (\ref{pseudoaffinity}) using any component 
of the Chebyshev blossom.
To show this fact, we can choose to work with the component $\varphi_{k}$ of the Chebyshev 
blossom and give the expression of the pseudo-affinity factor in a similar fashion as in 
the proof of Theorem \ref{pseudoaffinity} and in which this time we use 
the last Proposition instead of Proposition \ref{condensationproposition}.  
\end{remark} 
%%%%%%%%%%%%%%%%%%%%%%%%%%%%%%%%%%%%%%%%%%%%%%%%%%%%%%%%%%%%%%%%%%%%%%%%%%%%%%%%%%%%%%%%%%%%%%%%%%
%%%%%%%%%%%%%%%%%%%%%%%%%%%%%%%%%%%%%%%%%%%%%%%%%%%%%%%%%%%%%%%%%%%%%%%%%%%%%%%%%%%%%%%%%%%%%%%%%%
\section{The Chebyshev-Bernstein Basis}
The main objective of this section is to give an explicit 
expression in terms of Schur functions of the Chebyshev-Bernstein
basis of the space $\mathcal{E}_{\lambda}(n)$ associated with a partition 
$\lambda$ of length at most $n$. As the proof involve several 
technical steps, we will first give the main result and some of it
consequences. We will explain the methodology of the proof along 
the coming subsections.

\begin{theorem}\label{bernsteintheorem}
The Chebychev-Bernstein basis
$(B^{n}_{0,\lambda},B^{n}_{1,\lambda},...,B^{n}_{n,\lambda})$
of the \muntz space associated with a partition 
$\lambda = (\lambda_1,\lambda_2,...,\lambda_n)$ 
of length at most $n$ over an interval $[a,b]$ is given by 
\begin{equation}\label{bernstein}
B^{n}_{k,\lambda}(t) = \frac{f_{\lambda}(n+1)}{f_{\lambda^{(0)}}(n)}
B^{n}_{k}(t) 
\frac{S_{\lambda^{(0)}}(a^{n-k},b^{k}) t^{\lambda_1}S_{\lambda}(a^{n-k},b^{k},\frac{ab}{t})} 
{S_{\lambda}(a^{n+1-k},b^{k}) S_{\lambda}(a^{n-k}, b^{k+1})},
\end{equation}
where $B^{n}_{k}$ is the classical Bernstein basis of the polynomial space
over the interval $[a,b]$ and $\lambda^{(0)}$ is the bottom partition of $\lambda$.
\end{theorem}

To exhibit the fact that the Chebyshev-Bernstein basis $B^{n}_{k,\lambda}$ in (\ref{bernstein})  
is indeed a polynomial function in $t$, we could use the Branching rule
(\ref{branchingrule1}) as 
$$
t^{\lambda_1}S_{\lambda}\left( a^{n-k},b^{k},\frac{ab}{t} \right) = 
t^{\lambda_1} \sum_{\eta \prec \lambda} S_{\eta}(a^{n-k},b^{k})
\left( \frac{ab}{t} \right)^{|\lambda|-|\eta|},
$$
the sum is over are the interlacing partitions $\eta$ i.e., partition 
$\eta = (\eta_1,...,\eta_{n-1})$ such that
\begin{equation}\label{localeta}
\lambda_1 \geq \eta_1 \geq \lambda_2 \geq ...\eta_{n-1} \geq \lambda_n.
\end{equation}
Therefore, 
\begin{equation}
t^{\lambda_1}S_{\lambda}\left( a^{n-k},b^{k},\frac{ab}{t} \right) = 
\sum_{\eta \prec \lambda} S_{\eta}(a^{n-k},b^{k}) 
(ab)^{|\lambda|-|\eta|} t^{|\eta|-|\lambda^{(0)}|}, 
\end{equation}
where $\lambda^{(0)}$ is the bottom partition of $\lambda$.  
Any partition $\eta$ that satisfies (\ref{localeta}) also satisfies 
$|\eta|-|\lambda^{(0)}| \geq 0$. Therefore, for any $k=0,..,n$, 
the Chebyshev-Bernstein function $B^{n}_{k,\lambda}$ is a polynomial
function in $t$. We can also use the branching rule (\ref{branchingrule2})
as  
\begin{equation}\label{branche3}
\begin{split}
t^{\lambda_1}S_{\lambda}\left( a^{n-k},b^{k},\frac{ab}{t} \right) = &  \quad 
t^{\lambda_1} \sum_{j=0}^{\lambda1} S_{\lambda/(j)}(a^{n-k},b^{k})(\frac{ab}{t})^{j} \\
= &  \quad \sum_{j=0}^{\lambda1}(ab)^{j} S_{\lambda/(j)}(a^{n-k},b^{k})t^{\lambda_1-j}.
\end{split}
\end{equation}
The term $f_{\lambda}(n+1)/f_{\lambda^{(0)}}(n)$ in (\ref{bernstein}) can
be computed using the hook-length formula (\ref{hooklengthformulas}).
However, since we have a ratio of hook lengths of two related 
partitions, several simplifications will appear. In fact as the following
lemma shows, to compute this term, we need only 
to form the hook length and the content of the first row of 
the partition $\lambda$.
%%%%%%%%%%%%%
%%%%%%%%%%%%%
\begin{lemma}\label{hooklemma}
Let $\lambda = (\lambda_1,\lambda_2,...,\lambda_n)$ be a non-empty
partition of length at most $n$ and let $\lambda^{(0)}$ be its bottom partition.
Then, we have    
$$
\frac{f_{\lambda}(n+1)}{f_{\lambda^{(0)}}(n)} = 
\prod_{j=1}^{\lambda_1} \frac{(n+1) + c_{\lambda}(1,j)}{h_{\lambda}(1,j)}.
$$
\end{lemma}
%%%%%%%%%%%%%
%%%%%%%%%%%%%
\begin{proof}
If the partition $\lambda$ consist of a single part 
$\lambda = (\lambda_1,0,0,...,0)$, then the partition $\lambda^{(0)}$
is empty and the lemma is the statement 
of the hook length formulas (\ref{hooklengthformulas}). 
Let us assume that $\lambda = (\lambda_1,\lambda_2,...,\lambda_s,0,...,0)$
consists of more than a single part, i.e., $\lambda_1 \geq \lambda_2 \geq 1$.
The definition of a bottom partition imply that
for every non-empy box in the partition $\lambda^{(0)}$,
and $i \neq 1$, we have 
$$
h_{\lambda}\left(i,j\right) = h_{\lambda^{(0)}}\left(i-1,j\right) 
\quad \textnormal{and} \quad 
c_{\lambda} \left(i,j\right) = c_{\lambda^{(0)}}\left(i-1,j\right)- 1.
$$
Therefore, for any $i \neq 1$, we have 
$$
\frac{(n+1) + c_{\lambda}(i,j)}{h_{\lambda}(i,j)}  = 
\frac{n + c_{\lambda^{(0)}}(i-1,j)}{h_{\lambda^{(0)}}(i-1,j)}.
$$
Thereby, we have 
$$
\frac{f_{\lambda}(n+1)}{f_{\lambda^{(0)}}(n)} = 
\prod_{j=1}^{\lambda_1} \frac{(n+1) + c_{\lambda}(1,j)}{h_{\lambda}(1,j)}. 
$$ 
\end{proof}

\begin{example} 
Consider the partition $\lambda = (4,2,1)$. 
Then the content and the hook length of the first row are
given by 
$$
\newcommand{\nothing}{\mbox{}}
Content(\lambda) = {\Yvcentermath1 \young(0123,\nothing \nothing,\nothing)} 
\qquad  
h(\lambda) = {\Yvcentermath1 \young(6421,\nothing \nothing,\nothing)} 
$$
Therefore, 
$$
\frac{f_{\lambda}(n+1)}{f_{\lambda^{(0)}}(n)} = 
\frac{(n+1)}{6} \frac{(n+2)}{4} \frac{(n+3)}{2} \frac{(n+4)}{1}.
$$
\end{example}
Using Theorem \ref{bernsteintheorem}, we can give the explicit expression 
of the Chebyshev-Bernstein basis associated with the \muntz spaces defined
in the examples section 
%%%%%%%%%%%%%%%%%
\vskip 0.2 cm
\noindent{\bf{Combinatorial \muntz space:}}
Let $(B^{3}_{0,(2,2)}, B^{3}_{1,(2,2)},(B^{3}_{2,(2,2)}, B^{3}_{3,(2,2)})$
be the Chebyshev-Bernstein basis of the \muntz space $\mathcal{E}_{(2,2)}(3)$
associated with the partition $\lambda = (2,2)$ over an interval $[a,b]$, i.e.,
the space $\mathcal{E}_{(2,2)}(3) = span(1,t,t^4,t^5)$. From Theorem \ref{bernsteintheorem},
Lemma \ref{hooklemma} and the branching rule (\ref{branche3}), we have
\begin{equation*}
\newcommand{\nothings}{\mbox{}}
B^{3}_{k,(2,2)}(t) = \frac{10}{3} \frac{\tiny S_{\yng(2)}(U)
\left(S_{{\tiny \yng(2,2)}}(U) t^2 + ab S_{{{ \tiny \young(:\nothings,\nothings\nothings)}}}(U) t +
a^2b^2 \tiny S_{\yng(2)}(U)\right)}{S_{{\tiny \yng(2,2)}}(U,a) S_{{\tiny \yng(2,2)}}(U,b)}
B_{k}^{3}(t),
\end{equation*}
where $U = (a^{3-k},b^{k})$ and the Schur and the skew Schur functions can be computed, for instance,
using the combinatorial definitions.
%%%%%%%%%%%%%%%%%
\vskip 0.2 cm
\noindent{\bf{Elementary \muntz spaces:}}
Let $(B^{n}_{0,(1^r)}, B^{n}_{1,(1^r)},...,B^{n}_{n,(1^r)})$
be the Chebyshev-Bernstein basis of the $r$th elementary \muntz space 
$\mathcal{E}_{(1^r)}(n)$ over an interval $[a,b]$. We have 
$$
t S_{(1^r)}(a^{n-k},b^{k},\frac{ab}{t}) = 
t e_{r}(a^{n-k},b^{k}) + e_{r-1}(a^{n-k},b^{k}).
$$
Moreover, by Lemma \ref{hooklemma}, we have 
$$
\frac{f_{\lambda}(n+1)}{f_{\lambda^{(0)}}(n)} = 
\frac{f_{(1^r)}(n+1)}{f_{(1^{r-1})}(n)} = \frac{n+1}{r}.
$$
Therefore, a direct application of Theorem \ref{bernsteintheorem}, 
shows 
\begin{corollary}
The Chebyshev-Bernstein basis  
$(B^{n}_{0,(1^r)}, B^{n}_{1,(1^r)},...,B^{n}_{n,(1^r)})$ of the $r$th 
elementary \muntz space over an interval $[a,b]$ is given by  
$$ 
B^{n}_{k,(1^{r})}(t) = \frac{(n+1)}{r} B^{n}_{k}(t) 
\frac{e_{r-1}(a^{n-k},b^{k}) \left(t e_{r}(a^{n-k},b^{k})+ a b e_{r-1}(a^{n-k},b^{k}) \right)}
{e_{r}(a^{n+1-k},b^{k}) e_{r}(a^{n-k},b^{k+1})}.
$$ 
\end{corollary}
%%%%%%%%%%%%%%%%%
\vskip 0.2 cm
\noindent{\bf{Complete \muntz spaces: }}
Let $(B^{n}_{0,(r)}, B^{n}_{1,(r)},...,B^{n}_{n,(r)})$
be the Chebyshev-Bernstein basis of the $r$th complete \muntz space 
$\mathcal{E}_{(r)}(n)$ over an interval $[a,b]$. 
The branching rule (\ref{branchingrule1}) leads to  
$$
t^{r} S_{(r)}(a^{n-k},b^{k},\frac{ab}{t}) = 
\sum_{j=0}^{r} (ab)^{r-j} t^{j} h_{j}(a^{n-k},b^{k}).
$$
Therefore, applying Theorem \ref{bernsteintheorem} lead to the same result as in
\cite{Mazure4}, namely 
 
\begin{corollary}
The Chebyshev-Bernstein basis 
$(B^{n}_{0,(r)}, B^{n}_{1,(r)},...,B^{n}_{n,(r)})$ 
of the $r$th complete \muntz  space over an interval $[a,b]$ is given by  
$$ 
B^{n}_{k,(r)}(t) =\binom{n+r}{n} B^{n}_{k}(t)  
\frac{\sum_{j=0}^{r} (ab)^{r-j} h_{j}(a^{n-k},b^{k}) t^{j}}
{h_{r}(a^{n+1-k},b^{k}) h_{r}(a^{n-k},b^{k+1})}.
$$ 
\end{corollary}
%%%%%%%%%%%%%%%%%
%%%%%%%%%%%%%%%%%
\vskip 0.2 cm
\noindent{\bf{Hook \muntz spaces:}}
Let $(B^{n}_{0,(l|r)}, B^{n}_{1,(l|r)},...,B^{n}_{n,(l|r)})$ be the
the Chebyshev-Bernstein basis of the $(l|r)$ hook \muntz space 
$\mathcal{E}_{(l|r)}(n)$ over an interval $[a,b]$. 
Noticing that for $j=1,..,l+1$, $(l|r)/(j)$ consist of two 
connected components. Therefore, for $j=1,...,n$,  
we have $S_{(l|r)/(j)} = h_{l+1-j} e_{r}, j=1,...,\lambda_1$.
Therefore, using the branching rule 
(\ref{branchingrule2}), we have 
\begin{equation*}
\begin{split}
t^{l+1}S_{(l|r)} \left( a^{n-k},b^{k},\frac{ab}{t} \right) & = S_{(l|r)}(a^{n-k},b^{k}) t^{l+1} + \\
& \quad \sum_{j=1}^{l+1} h_{l+1-j}(a^{n-k},b^{k})
e_{r}(a^{n-k},b^{k}) (ab)^{j} t^{l+1-j}. \\
\end{split}
\end{equation*}
Therefore, Theorem \ref{bernsteintheorem} gives  
\begin{corollary}
The Chebyshev-Bernstein basis of the $(l|r)$-hook \muntz  
space over an interval $[a,b]$ is given by  
\begin{equation*}
\begin{split}
& B^{n}_{k,(l|r)}(t) = B^{n}_{k}(t) \frac{n+1}{r+l+1} \binom{n+l+1}{n+1} \\
& \frac{\left( e_{r}(a^{n-k},b^{k}) \right)^2 
\sum_{j=1}^{l+1} (ab)^{j} t^{l+1-j} h_{l+1-j}(a^{n-k},b^{k}) +  
e_{r}(a^{n-k},b^{k}) S_{(l|r)}(a^{n-k},b^k) t^{l+1}} 
{S_{(l|r)}(a^{n+1-k},b^{k}) S_{(l|r)}(a^{n-k},b^{k+1})}\quad \\
\end{split}
\end{equation*}
\end{corollary}
%%%%%%%%%%%%%%
\vskip 0.2 cm
\noindent{\bf{Staircase \muntz spaces: }}
If $(B^{n}_{0,\lambda}, B^{n}_{1,\lambda},...,B^{n}_{n,\lambda})$
is the Chebyshev-Bernstein basis over an interval $[a,b]$ 
of the $l$-staircase \muntz space $\mathcal{E}_{\lambda}(n)$,
where the partition $\lambda$ is given in (\ref{staircase}),
then from (\ref{reparametrizationbernstein}), we have  
$$
B_{k,\lambda}^{n}(t) = B_{k}^{n}(t^{l+1}), 
\quad \textnormal{for} \quad k=0,...,n,
$$
where $B_{k}^n$ is the classical Bernstein basis over the interval 
$[a^{l+1},b^{l+1}]$. Our objective here is then to show that Theorem \ref{bernsteintheorem}
reconfirm this fact. The method of computation consists in using
equations (\ref{staircaseformulas}) for the partitions $\lambda$ 
and $\lambda^{(0)}$ and inserting these equations into 
Theorem \ref{bernsteintheorem}. We will omit all the details of the computation
but only mention that it is helpful to rename $(a^{n-k},b^{k}) = (u_1,...,u_n)$
in order to detect easily the simplifications and that by equations 
(\ref{staircaseformulas}), the term
$$
\frac{f_{\lambda}(n+1)}{f_{\lambda^{(0)}}(n)} = 
\frac{S_{\lambda}(1^{n+1})}{S_{\lambda^{(0)}}(1^n)}
= (l+1)^{n}
$$
appears naturally within the computations. We find 
$$
B_{k,\lambda}^{n}(t) = \binom{n}{k}\frac{(t-a)^{k}(b-t)^{n-k}}{(b-a)^n}
\frac{h_{l}(t,a)^{k} h_{l}(t,b)^{n-k}}{h_{l}(a,b)^{n}}.   
$$
Therefore, 
$$
B_{k,\lambda}^{n}(t) =  \binom{n}{k} \frac{(t^{l+1}-a^{l+1})^{k}(b^{l+1}-t^{l+1})^{n-k}}{(b^{l+1}-a^{l+1})^n}
$$
as expected.

\subsection{Weighted de Casteljau paths}
Our strategy for computing the Chebyshev-Bernstein basis associated 
with a \muntz space $\mathcal{E}_{\lambda}(n)$ consists of computing the 
product of weights associated with specific paths in the de Casteljau 
algorithm. Working directly with the Pascal-like graph of the de Casteljau algorithm reveal 
to be challenging and induction arguments does not seem to work. 
For these reasons, we will define a combinatorial object that, in some sense,
can be viewed as one-dimensional projection of the two-dimensional paths 
in the de Casteljau graph.
%%%%%%%%   
\begin{definition}
A set $\mathbb{A}_n = (A_0,A_{1},A_{2},....,A_{n})$ is said to be a de Casteljau path 
in $\{0,1,...,n\}$ if and only if,
\vskip 0.2 cm
\noindent i) Every set $A_l$ is a subset of $\{0,1,...,n\}$ such that $|A_l|=l+1$.
\vskip 0.2cm
\noindent ii) The set $A_{l+1}$ is obtained from $A_{l}$ by adding to $A_{l}$ an element
of the form $max(A_l)+1$ or $min(A_l)-1$ under the condition that the set 
remains a subset of  $\{0,1,...,n\}$.
\end{definition}  
\noindent We will often represent a de Casteljau path as 
$$
A_0 \longrightarrow A_1 \longrightarrow.....A_{n-1} \longrightarrow A_{n}
$$
in which the subsets $A_{i}, i=0,...,n$ are viewed as vertices and the 
arrows are viewed as edges. We will also adopt the convention of writing 
the elements of the set $A_{i}, i=0,...,n$ in increasing order. Note that for 
any de Casteljau path $\mathbb{A}_n = (A_0,...,A_{n})$ in $\{0,1,...,n\}$,
we necessarily have $A_{n} = \{0,1,...,n\}$.  
\begin{example}
Examples of de Casteljau paths in $\{0,1,2,3\}$ and 
$\{0,1,2,3,4\}$ respectively can be given as    
$$
\{1\} \longrightarrow  \{0,1\} \longrightarrow  \{0,1,2\} 
\longrightarrow  \{0,1,2,3\}  
$$
or
$$ 
\{2\} \longrightarrow  \{1,2\} \longrightarrow  \{1,2,3\} 
\longrightarrow  \{1,2,3,4\} \longrightarrow  \{0,1,2,3,4\} 
$$
\end{example}

\begin{definition}\label{weightdefinition}
Let $a, b$ and $t$ be real parameters and let $\psi = (\psi^{+},\psi^{-})$ 
be a non-zero two components real function 
$$
\psi(u_1,...,u_{n-1};a,b,t) = (\psi^{+}(u_1,...,u_{n-1};a,b,t),\psi^{-}(u_1,...,u_{n-1};a,b,t))
$$ 
where $\psi^{+}$ and $\psi^{-}$ are symmetric in the variables $u_i, i=1,...,n-1$.
A $\psi$-weighted de Casteljau path $\mathbb{A}_n = (A_{0},A_{1},...,A_{n})$ is defined as 
associating a weight to every edge $A_l \longrightarrow A_{l+1}$
of the path $\mathbb{A}_{n}$ according to the following rules :

\noindent 1) If $A_{l+1}$ is obtained from $A_{l}$ by adding the element
$max(A_{l})+1$, then the weight of the edge is given by 
\begin{equation}\label{rule1}
\psi^{+}(t^{|A_{l}|-1},b^{min(A_{l+1})},a^{n-|A_{l}|-min(A_{l+1})};a,b,t). 
\end{equation}
\noindent 2) If $A_{l+1}$ is obtained from $A_{l}$ by adding the element
$min(A_{l})-1$, then the weight of the edge is given by  
\begin{equation}\label{rule2}
\psi^{-}(t^{|A_{l}|-1},b^{min(A_{l+1})},a^{n-|A_{l}|-min(A_{l+1})};a,b,t). 
\end{equation}
We will represent the weight on the edges as 
$$
A_{l} \xrightarrow{\psi_1^{+}}   A_{l+1} 
\quad \textnormal{or} \quad 
A_{l} \xrightarrow{\psi_1^{-}}   A_{l+1}
$$
in which there is no need to write the arguments of the functions 
$\psi^{+}$ and $\psi^{-}$ as they are uniquely defined from the sets 
$A_{l}$ and $A_{l+1}$ according to the rules (\ref{rule1}) and (\ref{rule2}). 

We define the weight $W_{\psi,\mathbb{A}_n}(a,b,t)$ of a
$\psi$-weighted de Casteljau path $\mathbb{A}_n$ as the product
of the weights of the edges. 
\end{definition}

Note that by a simple application of the pigeonhole principle, for every de Casteljau 
path $\mathbb{A}_n = (A_0,A_1,...,A_l)$, we have 
$|A_{l}| + min(A_{l+1}) \leq n$. Therefore, all the exponents (referring to the number 
of occurrence of the arguments) in (\ref{rule1}) and (\ref{rule2}) are non-negatives.

\bigskip
Let $\lambda$ be a partition of length at most $n$ and $\lambda^{(0)}$ 
its bottom partition. Let us define a function $\psi_1 = (\psi_1^{+},\psi_1^{-})$ by   
\begin{equation}\label{psi1}
\begin{split}
\psi_1^{+}(u_1,...,u_{n-1};a,b,t) = & 
\frac{S_{\lambda^{(0)}}(u_1,...,u_{n-1},a)}{S_{\lambda^{(0)}}(u_1,...,u_{n-1},t)}, \\ 
\psi_1^{-}(u_1,...,u_{n-1};a,b,t) = & 
\frac{S_{\lambda^{(0)}}(u_1,...,u_{n-1},b)}{S_{\lambda^{(0)}}(u_1,...,u_{n-1},t)}.
\end{split}
\end{equation}
We have 
%%%%%%%%%%%%
%%%%%%%%%%%%
\begin{proposition}\label{lambda0}
The weight $W_{\mathbb{A},\psi_1}(a,b,t)$ of any de Casteljau path
$\mathbb{A}=(A_0,...,A_n)$ with $A_{0} = \{k\}$ and $\psi_1$ defined in 
(\ref{psi1}) is given by
$$
W_{\mathbb{A},\psi_1}(a,b,t) = 
\frac{S_{\lambda^{(0)}}(a^{n-k},b^{k})}{S_{\lambda^{(0)}}(t^{n})}.
$$
\end{proposition}

\begin{proof}
Let $\mathbb{A}_n = (A_0,A_1,...,A_n)$ be a de Casteljau path and  
consider the generic product of weights over two arbitrary adjacent
edges 
$$
A_{l} \longrightarrow  A_{l+1} \longrightarrow  A_{l+1}.
$$
In general, we would have four situations with regards to the weights,
namely
\begin{equation}\label{edges1}
A_{l} \xrightarrow{\psi_1^{+}} A_{l+1} \xrightarrow{\psi_1^{+}} A_{l+2}, \quad
A_{l} \xrightarrow{\psi_1^{+}} A_{l+1} \xrightarrow{\psi_1^{-}} A_{l+2},
\end{equation}
\begin{equation}\label{edges2}
A_{l} \xrightarrow{\psi_1^{-}} A_{l+1} \xrightarrow{\psi_1^{-}} A_{l+2}, \quad
A_{l} \xrightarrow{\psi_1^{-}} A_{l+1}  \xrightarrow{\psi_1^{+}} A_{l+2},
\end{equation}
Let us consider the first case 
$A_{l} \xrightarrow{\psi_1^{+}} A_{l+1} \xrightarrow{\psi_1^{+}} A_{l+2}$ and 
let $D$ be the denominator of the first edge and $N$ the numerator of the second edge.
By definition \ref{weightdefinition}, we have 
$$
D = S_{\lambda^{(0)}}(t^{|A_l|-1},b^{min(A_{l+1})},a^{n-|A_l|-min(A_{l+1})},t)
$$
and 
$$
N = S_{\lambda^{(0)}}(t^{|A_{l+1}|-1},b^{min(A_{l+2})},a^{n-|A_{l+1}|-min(A_{l+2})},a).
$$
In this case, we have $min(A_{l}) = min(A_{l+1}) = min(A_{l+2})$.
Counting the number of occurrence of the variables $t, a$ and $b$
in the expression of $D$ and $N$, shows that $D = N$. 
Similar arguments show that also for the other situations in (\ref{edges1}) 
and (\ref{edges2}) the denominator of the first edge is equal 
to the numerator of the second edge. Therefore, upon taking the product of 
the weights of a de Casteljau path, the denominator of the weight of an edge
will be simplified with the numerator of the weight of the adjacent edge.
The two factors that survive the simplifications are: the numerator of the 
weight of the first edge and the denominator of the weight of the last edge
of the de Casteljau path.
Let us compute the numerator of the weight of the first edge:
Since, by assumption, we have $A_{0} = \{k \}$, we will have 
two situations 
$$
A_{0} = \{ k \}  \xrightarrow{\psi_1^{+}} \{ k,k+1 \} 
\quad or \quad
A_{0} = \{ k \}  \xrightarrow{\psi_1^{-}} \{ k,k-1 \}. 
$$
In the first case, the numerator is $S_{\lambda^{(0)}}(t^0,b^k,a^{n-k-1},a)$, 
while in the second case, the numerator is $S_{\lambda^{(0)}}(t^0,b^{k-1},a^{n-k},b)$.
Therefore, in both cases, the numerator is $S_{\lambda^{(0)}}(a^{n-k},b^{k})$.  
For the denominator of the weight of the last edge,
we can only have a single situation, which is   
$$
A_{n-1} \xrightarrow{\psi_1^{+}} \{0,1,...n\}, 
$$
and in which the denominator is given by $S_{\lambda^{(0)}}(t^n)$.
\end{proof}

\begin{remark}
Note that we did not use the fact that $S_{\lambda^{(0)}}$ is a Schur function 
and all what was needed is for the function $S_{\lambda^{(0)}}$ to be a symmetric function.
A similar remark can be applied to the next proposition. The reason for us not 
to state the propositions in their full generality is to make the exposition for later use 
more transparent.
\end{remark}

Let $\lambda$ be a partition of length at most $n$ and consider the 
function $\psi_2 = (\psi_2^{+},\psi_2^{-})$ defined by  
\begin{equation}\label{psi2}
\begin{split}
\psi_2^{+}(u_1,...,u_{n-1};a,b,t) = &
\frac{S_{\lambda}(u_1,...,u_{n-1},b,t)}{S_{\lambda}(u_1,...,u_{n-1},a,b)}, \\
\psi_2^{-}(u_1,...,u_{n-1};a,b,t) = &
\frac{S_{\lambda}(u_1,...,u_{n-1},a,t)}{S_{\lambda}(u_1,...,u_{n-1},a,b)}.
\end{split}
\end{equation} 
In this case, there would be simplifications in the weight associated to every 
de Casteljau path, however there is no simple close explicit formulas that embodies 
all the simplifications. In the following, we will show that the 
specialization $t=a$ or $t=b$ on the weights of every de Casteljau path is given by 
a simple closed formulas. More precisely, we have 

\begin{proposition}\label{psi2path}
For any de Casteljau path $\mathbb{A}_n=(A_0,A_1,...,A_n)$ such that 
$A_{0} = \{  k \}$, we have 
\begin{equation}\label{first}
W_{(\psi_2,\mathbb{A}_n)}(a,b,t)|_{t=a} = \frac{S_{\lambda}(a^{n+1})}
{S_{\lambda}(a^{n+1-k},b^{k})}
\end{equation}
and 
\begin{equation}\label{second}
W_{(\psi_2,\mathbb{A}_n)}(a,b,t)|_{t=b} = \frac{S_{\lambda}(b^{n+1})}
{S_{\lambda}(a^{n-k},b^{k+1})},
\end{equation}
where $\psi_2$ is defined in (\ref{psi2}).
\end{proposition}

\begin{proof}
Let us start with the first equation (\ref{first}) of the Proposition. Noticing
that for $t=a$, we have $\psi_{2}^{+} \equiv 1$ shows that the edges with weight $1$ 
does not contribute to the total weight of a de Casteljau path.
Let $\mathbb{A}_n = (A_0,A_1,...,A_n)$ be a de Casteljau path and consider a situation
in which a part of the path has the following weights  
\begin{equation}\label{localpath}
A_{l} \xrightarrow{\psi_2^{-}} A_{l+1} \xrightarrow{1} A_{l+2}\xrightarrow{1} ....\xrightarrow{1}A_{l+h-1}
\xrightarrow{\psi_2^{-}} A_{l+h}
\end{equation}
with $h \geq 1$. Consider the numerator $N$ of the weight of the first
edge of (\ref{localpath}) and the denominator $D$ of the weight of 
last edge of (\ref{localpath}). We have 
$$
N = S_{\lambda}(t^{|A_l|-1},b^{min(A_{l+1})},a^{n-|A_{l}|-min(A_{l+1})},a,t) 
$$
and 
$$
D = S_{\lambda}(t^{|A_{l+h-1}|-1},b^{min(A_{l+h})},a^{n-|A_{l+h-1}|-min(A_{l+h})},a,b).
$$
From the structure of the path (\ref{localpath}), we can see that
$min(A_{l+h}) = min(A_{l+1})-1$. Therefore, the number of occurence of 
$b$ in $N$ is equal to the number of occurence of $b$ in $D$.
This shows that when $t=a$, we have $N=D$.
Thus, for any de Casteljau path,
the numerator of an edge with weight $\psi_2^{-}$ will be simplified with the denominator
of the weight of the closest edge with weight $\psi_2^{-}$. There is two special cases that 
will need a separate treatment. The case in which there is no edge with weight $\psi_2^{-}$ 
and the case where there is a single edge with weight  $\psi_2^{-}$. In the former case,
we only have a single de Casteljau path, namely  
$$
\{0\}  \xrightarrow{1}  \{0,1\}  \xrightarrow{1} \{0,1,2\} \xrightarrow{1} 
 ....  \xrightarrow{1}  \{0,1,2,...,n\}
$$
with weight equal to $1$. This case corresponds to a de Casteljau path
$\mathbb{A}_n = (A_0,A_1,...,A_n)$ with $A_{0} = \{0\}$ and the weight of the path is
consistent with equation (\ref{first}) of the proposition for $k=0$.  
In the latter case, there exists an $r \leq n$ such that the de Casteljau path will look like 
$$
\{1\}  \xrightarrow{1}  \{1,2\} .... \{1,2,...,r\}  \xrightarrow{\psi_2^{-}} 
\{0,1,...,r\} \xrightarrow{1} ....   \xrightarrow{1} \{0,1,2,...,n\}.
$$
In this case the weight of the path, at the specialization $t=a$, is given by 
$$
\frac{S_{\lambda}(t^{r},a^{n-r+1})}{S_{\lambda}(t^{r-1},a^{n-r+1},b)}|_{t=a} 
= \frac{S_{\lambda}(a^{n+1})}{S_{\lambda}(a^{n},b)}.
$$
As this case corresponds to a de Casteljau path
$\mathbb{A}_n = (A_0,A_1,...,A_n)$ with $A_{0} = \{1\}$, the result is again  
consistent with the first equation of the proposition when $k =1$.
For all the other cases, the weight of the de Casteljau path is then a ratio in which 
the denominator is given by the denominator of the first edge with weight $\psi_2^{-}$,
and the numerator is the numerator of the last edge with weight $\psi_2^{-}$. 
For the first edge we can only have the following two situations 
$$
\{k\}  \xrightarrow{\psi_2^{-}} \{k,k-1\} 
$$
and
$$
\{k\}  \xrightarrow{1} \{k,k+1\}  \xrightarrow{1} \{k,k+1,k+2\}.... \xrightarrow{1} \{k,k+1,k+2,...\}
\xrightarrow{\psi_2^{-}} \{k-1,k,k+1,...\}.
$$
In both cases, the denominator, when $t=a$, is given by 
$$
S_{\lambda}(a^{n+1-k},b^{k}).
$$ 
For the last edge, we would have the following situation 
$$
A_{l}\xrightarrow{\psi_2^{-}} A_{l+1} \xrightarrow{1} A_{l+2}\xrightarrow{1}....\xrightarrow{1} \{0,1,...,n\}.
$$
In this case, we have $min(A_{l+1})=0$ and the numerator of the edge 
$A_{l}\xrightarrow{\psi_2^{-}} A_{l+1}$, when $t=a$, is given by 
$$
S_{\lambda}(a^{n+1}).
$$
This prove the statement of (\ref{first}). Similar arguments when we take the 
specialization $t=b$ lead to the second equation (\ref{second}) of the proposition.
\end{proof}

Consider, now, the triangular Pascal-like graph of the de Casteljau
algorithm. We encode every node of the graph as follows : 
the node that lies in the $r$th horizontal level and the $k$th
position going from left to right is encoded as 
$\{k-1,k,k+1,...,k+r-2\}$. For example, the code associated with
the de Casteljau algorithm based on $4$ control points is given by

\setlength{\unitlength}{1mm}
\begin{picture}(100,40)
\put(24, 33){\{0\}} \put(44,33){\{1\}}
\put(64,33){\{2\}}   \put(84,33){\{3\}} 
\put(26,31){\vector(1,-1){5}} \put(46,31){\vector(-1,-1){5}}
\put(47,31){\vector(1,-1){5}} \put(66,31){\vector(-1,-1){5}}
\put(67,31){\vector(1,-1){5}} \put(86,31){\vector(-1,-1){5}}
\put(32, 22){\{0,1\}} \put(52, 22){\{1,2\}} \put(72, 22){\{2,3\}}
\put(100, 22){\mbox{($\Gamma$)}} 
\put(36,20){\vector(1,-1){5}} \put(56,20){\vector(-1,-1){5}} 
\put(57,20){\vector(1,-1){5}}  \put(76,20){\vector(-1,-1){5}} 
\put(41, 12){\{0,1,2\}} \put(61, 12){\{1,2,3\}} 
\put(46,10){\vector(1,-1){5}} \put(66,10){\vector(-1,-1){5}} 
\put(50, 2){\{0,1,2,3\}}
\end{picture}

\noindent Let $\lambda$ be a partition of length at most $n$ and 
denote by $\psi = (\psi^{+},\psi^{-})$ the function such that 
\begin{equation}
\begin{split}
\psi^{+}(u_1,....,u_{n-1};a,b,t) = & \beta(u_1,....,u_{n-1};a,b,t) \\
\psi^{-}(u_1,....,u_{n-1};a,b,t) = & \alpha(u_1,....,u_{n-1};a,b,t),
\end{split}
\end{equation}
where $\alpha$ and $\beta$ are the pseudo-affinity factors of the space 
$\mathcal{E}_{\lambda}(n)$ as defined in Theorem \ref{pseudotheorem}. 
We have
\begin{equation}\label{totalpsi}
\psi = (\psi^{+},\psi^{-}) = 
(\frac{b-t}{b-a}\psi_1^{+} \psi_2^{+}, \frac{t-a}{b-a}\psi_1^{-} \psi_2^{-}),
\end{equation}
where $\psi_1$ and $\psi_2$ are defined in (\ref{psi1}) and (\ref{psi2}) 
respectively. Consider the de Casteljau algorithm based on the 
control points $(\varphi(a^{n}),\varphi(a^{n-1},b),....,\varphi(b^{n}))$ 
where $\varphi$ is the Chebyshev blossom of a Chebyshev function $\phi$ over an interval 
$[a,b]$. Let us write a generic triangle in the de Casteljau algorithm with the value 
of the pseudo-affinity on the edges of the triangle and write the same triangle
with our coding of the de Casteljau algorithm and in which the weight on the edges follow
the rules (\ref{rule1}) and (\ref{rule2}) of Definition \ref{weightdefinition}.
We have 
\setlength{\unitlength}{1mm}
\begin{picture}(100,40)
\put(20,33){\mbox{$\small{\varphi(a^{n-r-k+2},b^{k-1},t^{r-1})}$}} 
\put(60,33){\mbox{$\varphi(a^{n-r-k+1},b^{k},t^{r-1})$}} 
\put(40,31){\vector(1,-2){7}} \put(80,31){\vector(-1,-2){7}}
\put(43,13){\mbox{$\varphi(a^{n-r-k+1},b^{k-1},t^{r})$}} 
%%the alpha factors
\put(7,25){\mbox{$\small{\beta(a^{n-r-k+1},b^{k-1},t^{r-1};a,b,t)}$}} 
\put(62,25){\mbox{$\small{\alpha(a^{n-r-k+1},b^{k-1},t^{r-1};a,b,t)}$}} 
\end{picture}

\noindent and 

\setlength{\unitlength}{1mm}
\begin{picture}(100,40)
\put(20,33){\{$k-1,k,...,k+r-2$\}} 
\put(60,33){\{$k,k+1,...,k+r-1$\}} 
\put(40,31){\vector(1,-2){7}} \put(80,31){\vector(-1,-2){7}}
\put(43,13){\{$k-1,k,...,k+r-1$\}} 
%%the alpha factors
\put(7,25){\mbox{$\small{\psi^{+}(a^{n-r-k+1},b^{k-1},t^{r-1};a,b,t)}$}} 
\put(62,25){\mbox{$\small{\psi^{-}(a^{n-r-k+1},b^{k-1},t^{r-1};a,b,t)}$}} 
\end{picture}

\noindent Realizing that the weights in both of the edges of the triangles are the same shows that 
if we denote by $(B_{0,\lambda}^{n},B_{1,\lambda}^{n},...,B_{n,\lambda}^{n})$
the Chebyshev-Bernstein basis of the \muntz space $\mathcal{E}_{\lambda}(n)$ over 
the interval $[a,b]$, then the de Casteljau algorithm is the claim that for $k=0,...,n$
$$
B_{k,\lambda}^{n}(t) = \sum_{\mathbb{A}_n} W_{\psi,\mathbb{A}_n}(a,b,t),
$$
where the sum is over all the de Casteljau paths $\mathbb{A}_n = (A_0,...,A_n)$
such that $A_0 = \{k\}$. It is simple to see from (\ref{totalpsi}) 
that for any de Casteljau path $\mathbb{A}_n =(A_0,...,A_n)$ such that $A_0 = \{k\}$, 
we have 
$$
W_{\psi,\mathbb{A}_n}(a,b,t) = \frac{(b-t)^{n-k} (t-a)^k}{(b-a)^n} 
W_{\psi_1,\mathbb{A}_n}(a,b,t) W_{\psi_2,\mathbb{A}_n}(a,b,t)
$$
Using Proposition \ref{lambda0} for $W_{\psi_1,\mathbb{A}_n}(a,b,t)$, we obtain

\begin{proposition}\label{pathproposition}
Let $\lambda$ be a partition of length at most $n$ and let 
$(B_{0,\lambda}^{n},B_{1,\lambda}^{n},...,$ $B_{n,\lambda}^{n})$ be
the Chebyshev-Bernstein basis of the \muntz space $\mathcal{E}_{\lambda}(n)$
over an interval $[a,b]$.Then, we have
$$
B_{k,\lambda}^{n}(t) = \frac{(b-t)^{n-k} (t-a)^k}{(b-a)^n} 
\frac{S_{\lambda^{(0)}}(a^{n-k},b^{k})}{S_{\lambda^{(0)}}(t^{n})} \sum_{\mathbb{A}_n} 
W_{\psi_2,\mathbb{A}_n}(a,b,t),
$$  
where the sum is over all the de Castelaju paths 
$\mathbb{A}_n= (A_0,...,A_n)$ such that $A_0 = \{ k \}$. $\lambda^{(0)}$ is the bottom
partition of the partition $\lambda$ and the function $\psi_2$ is defined in (\ref{psi2}).
\end{proposition}   

There is two special cases in which the quantity $W_{\psi_2,\mathbb{A}}(a,b,t)$ can be given 
an explicit expression. Let us consider the set of the de Casteljau paths 
$\mathbb{A}_n = (A_0,...,A_n)$ such that $A_0 = \{ 0 \}$. In fact, there is a
single path which is 
$$
\{0\}  \xrightarrow{\psi_2^{+}}  \{0,1\}  \xrightarrow{\psi_2^{+}} \{0,1,2\} \xrightarrow{\psi_2^{+}} 
 ....  \xrightarrow{\psi_2^{+}}  \{0,1,2,...,n\}.
$$
As the reader can readily check, the product of weights along this single path
is given by 
$$
W_{\psi_2,\mathbb{A}_n}(a,b,t) =\frac{S_{\lambda}(b,t^{n})}{S_{\lambda}(b,a^{n})}.
$$ 
For similar reasons, there is a single path  $\mathbb{A}_n = (A_0,...,A_n)$ 
such that $A_0 = \{n\}$, namely 
$$
\{n\}  \xrightarrow{\psi_2^{-}}  \{n-1,n\}  \xrightarrow{\psi_2^{-}} \{n-2,n-1,n\} \xrightarrow{\psi_2^{-}} 
 ....  \xrightarrow{\psi_2^{-}}  \{0,1,2,...,n\}.
$$
Along this path we have    
$$
W_{\psi_2,\mathbb{A}_n}(a,b,t) =\frac{S_{\lambda}(a,t^{n})}{S_{\lambda}(a,b^{n})}.
$$ 
Therefore, according to Proposition \ref{pathproposition}, we have 
%%%%%%
%%%%%%
\begin{proposition}\label{firstlast}
Let $\lambda$ be a partition of length at most $n$ and let 
$(B_{0,\lambda}^{n},B_{1,\lambda}^{n},...,$ $B_{n,\lambda}^{n})$ be
the Chebyshev-Bernstein basis of the \muntz space $\mathcal{E}_{\lambda}(n)$
over an interval $[a,b]$. Then, we have   
$$
B_{0,\lambda}^{n}(t) = \frac{(b-t)^n}{(b-a)^n} 
\frac{S_{\lambda}(b,t^{n})}{S_{\lambda}(b,a^{n})}
\frac{S_{\lambda^{(0)} }(a^{n})}{S_{\lambda^{(0)}}(t^{n})}
$$
and
$$
B_{n,\lambda}^{n}(t) = \frac{(t-a)^n}{(b-a)^n} 
\frac{S_{\lambda}(a,t^{n})}{S_{\lambda}(a,b^{n})}
\frac{S_{\lambda^{(0)}}(b^{n})}{S_{\lambda^{(0)}}(t^{n})}.
$$
\end{proposition}
%%%%%
\vskip 0.2 cm
%%%%%
In general, the expression of the Chebyshev-Bernstein basis
obtained by computing the weight on the de Casteljau paths
has complicated expressions. Let us for example consider a partition
$\lambda$ of length at most $2$ with it associated Chebyshev space 
$\mathcal{E}_{\lambda}(2)$.
Proposition \ref{firstlast} provides us with the Chebyshev-Bernstein functions
$B_{0,\lambda}^2$ and $B_{2,\lambda}^2$. In order to compute
$B_{1,\lambda}^2$, we should compute 
$\sum_{\mathbb{A}_2} W_{\psi_2,\mathbb{A}_2}(a,b,t)$
where the sum is over all the de Casteljau paths 
$\mathbb{A}_n = (A_0,A_1,A_2)$ with $A_0 = \{ 1 \}$.
In this case, we have two de Casteljau paths namely, 
$$
\{1\}  \xrightarrow{\psi_2^{-}} \{0,1\}  \xrightarrow{\psi_2^{+}} \{0,1,2\}, 
\qquad 
\{1\}  \xrightarrow{\psi_2^{+}} \{1,2\}  \xrightarrow{\psi_2^{-}}  \{0,1,2\}.
$$ 
Computing the weights along these two paths lead to 
$$
B_{1,\lambda}^2(t) = \frac{(b-t)(t-a)}{(b-a)^2}
\frac{S_{\lambda^{(0)}}(a,b)}{S_{\lambda^{(0)}}(t,t)}
\left( \frac{S_{\lambda}(a,a,t) S_{\lambda}(t,t,b)}
{S_{\lambda}(a,b,t) S_{\lambda}(a,a,b)}  + 
\frac{S_{\lambda}(b,b,t) S_{\lambda}(t,t,a)}
{S_{\lambda}(a,b,t) S_{\lambda}(b,b,a)} \right).
$$
It is rather surprising that with this expression in hand 
the function $B_{1,\lambda}^2$ is a polynomial function in $t$.
%%%%%
\vskip 0.2 cm
%%%%%
Although summing the weights over the de Castelaju paths does not lead
to a practical method of computing the Chebyshev-Bernstein basis,
the concept leads to the following important information about 
the derivatives of the Chebyshev-Bernstein basis, which by some 
inductive argument on nested \muntz spaces will provide us with
the desired explicit expression.   
%%%%%%
%%%%%%
\begin{theorem}\label{bernsteinderivative}
Let $\lambda = (\lambda_1,\lambda_2,...,\lambda_n)$ be a partition of 
length at most $n$ and let 
$(B_{0,\lambda}^n(t),B_{1,\lambda}^n(t),...,B_{n,\lambda}^n(t))$ be 
the Chebyshev-Bernstein basis of the \muntz space $\mathcal{E}_{\lambda}(n)$
over an interval $[a,b]$. Then, we have 
$$
{B_{k,\lambda}^n}^{(k)}(a) = \frac{n! \quad a^{\lambda_1}}{(n-k)!(b-a)^k} 
\frac{f_{\lambda}(n+1)}{f_{\lambda^{(0)}}(n)} 
\frac{S_{\lambda^{(0)}}(a^{n-k},b^{k})} 
{S_{\lambda}(a^{n+1-k},b^{k})} 
$$
and 
$$
{B_{k,\lambda}^n}^{(n-k)}(b) =  (-1)^{n-k}\frac{n!\quad b^{\lambda_1}}{k!(b-a)^{n-k}} 
\frac{f_{\lambda}(n+1)}{f_{\lambda^{(0)}}(n)} 
\frac{S_{\lambda^{(0)}}(a^{n-k},b^{k})}
{S_{\lambda}(a^{n-k},b^{k+1})},
$$
where $\lambda^{(0)}$ is the bottom partition of $\lambda$.
\end{theorem}
%%%%%%%%%%
%%%%%%%%%%
\begin{proof}
Let us define the function $H(a,b,t)$ by 
\begin{equation*}
H(a,b,t) = \frac{S_{\lambda^{(0)}}(a^{n-k},b^{k})}
{S_{\lambda^{(0)}}(t^{n})} \sum_{\mathbb{A}_n} W_{\psi_2,\mathbb{A}_n}(a,b,t),
\end{equation*}
where the sum is over all the de Casteljau paths $\mathbb{A}_n = (A_0,A_1,...,A_n)$
such that 
$A_0 = \{ k \}$. From Proposition \ref{pathproposition}, we have 
$$
{B_{k,\lambda}^n}(t) = \frac{(b-t)^{n-k} (t-a)^k}{(b-a)^n} H(a,b,t).
$$
By the Leibniz derivative formula, we then have 
$$
{B_{k,\lambda}^n}^{(k)}(a) = \frac{n!}{\binom{n}{k}(n-k)! (b-a)^k} H(a,b,t)_{|t=a}.
$$
Using Proposition \ref{psi2path}, the fact that there is $C_{n}^{k}$
de Casteljau paths $\mathbb{A}_n= (A_0,...,A_n)$ such that  $A_0 = \{ k \}$
and the fact that for any partition $\mu$, we have  
$S_{\mu}(x^n) = f_{\mu}(n) x^{|\mu|}$, we arrive at the first equation of 
the Proposition. 
Similar treatment on the $(n-k)$ derivative of the Chebyshev-Bernstein 
basis at the parameter $b$ lead to the second equation of the proposition
\end{proof}
 
\subsection{A Descent Construction of the Chebyshev-Bernstein Basis}
In the following, we will show that Theorem \ref{bernsteinderivative}
will allow us to relate the Chebyshev-Bernstein bases of \muntz spaces 
associated with two different partitions under a condition on the 
partitions that we state in the following definition  

\begin{definition}
Let $\lambda$ be a partition of length at most $n$. A partition 
$\mu$ of length at most $(n+1)$ is said to be a dimension elevation partition
of $\lambda$ if, and only if the \muntz spaces associated with the partitions
$\lambda$ and $\mu$ satisfy
$$
\mathcal{E}_{\lambda}(n) \subset \mathcal{E}_{\mu}(n+1).
$$
\end{definition}

We can characterize all the dimension elevation partitions of a 
given partition $\lambda$ as follows :

%%%%%%%%%%%%%%%%%%%%%%%%%
\begin{proposition}\label{elevationproposition}
Let $\lambda = (\lambda_1,...,\lambda_n)$ be a partition of length at 
most $n$. Then every dimension elevation partition 
$\mu$ is of the form 
\begin{equation}\label{elevation1}
\mu = (r+\lambda_1,r+\lambda_2,...,r + \lambda_n,r), 
\quad {\textnormal{with}} \quad r \geq 0,
\end{equation}
or of the form 
\begin{equation}\label{elevation2}
\mu = (\lambda_1-1,\lambda_2-1,...,\lambda_s-1,\rho,
\lambda_{s+1},...,\lambda_n) 
\quad {\textnormal{with}} \quad  1 \leq s \leq n,
\end{equation}
under the condition that $\mu$ is a partition. 
\end{proposition}
%%%%%%%%%%%%%%%%%%%%%%%%
\begin{proof}
Let $\phi$ be the Chebyshev function associated with the partition $\lambda$.
Then $\phi$ is given by 
$$
\phi(t) = \left( t^{\lambda_1-\lambda_2 +1},t^{\lambda_1-\lambda_3+2},...,
t^{\lambda_1-\lambda_n + (n-1)},t^{\lambda_1+n} \right)^{T}.
$$
There is two ways to supplement the function $\phi$ with an extra component 
of the form $t^{m}$ with $m$ a positive integer.
%%%%%%%%%
\vskip 0.1cm
\noindent{\bf{The first way: }}We can add a function 
$t^{m}$ such that $m$ is strictly larger than all the exponents in the components
of the function $\phi$, i.e., $m > \lambda_1 + n$.
In this case, if we denote by 
$\mu = (\mu_1,\mu_2,...,\mu_{n+1})$ the partition 
associated with the obtained space, then we have   
\begin{equation}\label{localpartition}
m = \mu_1 + (n+1), \;
\lambda_1 - \lambda_i = \mu_1 - \mu_i, 
\; for \quad i=2,...,n,
\; \textnormal{and} \; \mu_1 - \mu_{n+1} = \lambda_1.
\end{equation}   
If we denote by $r =  \mu_1 - \lambda_1$, then $r \geq 0$ as 
$m > \lambda_1 +n$. Moreover, equation (\ref{localpartition})
shows that we can write 
$\mu_i$ 
as $r + \lambda_i$ for $i=1,...,n$ and $\mu_{n+1} = r$, thereby leading to 
the form of the partition in (\ref{elevation1}).
%%%%%%%%
\vskip 0.1cm
\noindent{\bf{The second way: }}We can add a function of the form
$t^m$ in which the exponent $m$ lies between two exponents of 
the components of the function $\phi$, i.e., for an $1 \leq s \leq n$, we insert $t^m$ between 
$t^{\lambda_1-\lambda_s + (s-1)}$ and $t^{\lambda_1-\lambda_{s+1}+s}$ ($\lambda_{n+1} = 0$). Note that 
this is possible only if $\lambda_s > \lambda_{s+1}$. In this case, we would have 
$\lambda_1 + n = \mu_1 + (n+1)$ and therefore, $\mu_1 = \lambda_1-1$.
By using the condition that we should have
$$
\lambda_1 - \lambda_i = \mu_1 - \mu_i, 
\quad for \quad i=2,...,s
$$
and 
$$
\lambda_1 - \lambda_i -1 = \mu_1 - \mu_{i+1},
\quad for \quad i=s+1,...,n.
$$
we arrive at a partition of the form (\ref{elevation2}) with $\rho = \lambda_1 + s - (m+1)$.
\end{proof}  

Consider, now, a partition $\lambda$ of length at most $n$, and 
let $\mu$ a dimension elevation partition of $\lambda$. 
As an element of $\mathcal{E}_{\mu}(n+1)$,
the Chebyshev-Bernstein basis $B_{k,\lambda}^n, k=0,...,n,$ of 
$\mathcal{E}_{\lambda}(n)$ over an interval $[a,b]$ can be expressed 
as linear combination of the the Chebyshev-Bernstein basis  
$B_{k,\mu}^{n+1}, k=0,...,n+1,$ of $\mathcal{E}_{\mu}(n+1)$ over the same interval.
Exhibiting the vanishing properties of the Chebyshev-Bernstein bases as 
expressed in Theorem \ref{mazuretheorem} shows that \cite{Aldaz} 
\begin{equation}\label{derivativerecurence}
B_{k,\lambda}^n(t) = \frac{{B_{k,\lambda}^n}^{(k)}(a)}
{{B_{k,\mu}^{n+1}}^{(k)}(a)} B_{k,\mu}^{n+1}(t) 
+ \frac{{B_{k,\lambda}^n}^{(n-k)}(b)}{{B_{k+1,\mu}^{n+1}}^{(n-k)}(b)}
B_{k+1,\mu}^{n+1}(t).
\end{equation}
As Proposition \ref{bernsteinderivative} gives explicit expressions
of all the needed derivatives in the last equation,
we can express the Chebyshev-Bernstein basis associated with 
the partition $\lambda$ in term of the Chebyshev-Bernstein basis associated with 
a dimension elevation partition $\mu$. To write the expression in a more compact 
fashion we define the following factors 

\begin{definition}
Let $\lambda$ (resp. $\mu$) be a partition of length at most $n$
(resp. at most $(n+1)$). Denote by $\lambda^{(0)}$ 
(resp. $\mu^{(0)}$) the bottom partition of $\lambda$ (resp. $\mu$). 
For $k=0,...,n$, we define the following factors 
\begin{equation}\label{factor1}
\Gamma_{\lambda}^{\mu}(n,k) = 
\frac{S_{\lambda^{(0)}}(a^{n-k},b^k)
S_{\mu}(a^{n+2-k},b^k)}
{S_{\lambda}(a^{n+1-k},b^k) S_{\mu^{(0)}}(a^{n+1-k},b^k)}
\end{equation}
and 
\begin{equation}\label{factor2}
\Delta_{\lambda}^{\mu}(n,k)=
\frac{S_{\lambda^{(0)}}(a^{n-k},b^k)
S_{\mu }(a^{n-k},b^{k+2})}
{S_{\lambda}(a^{n-k},b^{k+1}) S_{\mu^{(0)}}(a^{n-k},b^{k+1})}
\end{equation}
and 
\begin{equation}\label{factor3}
{\binom{\mu}{\lambda}}_{\hskip -0.1cm {n}} = 
\frac{f_{\lambda}(n+1) f_{\mu^{(0)}}(n+1)}{f_{\lambda^{(0)}}(n) f_{\mu}(n+2)}.
\end{equation}
\end{definition}

Now using equation (\ref{derivativerecurence}) 
and in which we insert the value of the needed derivatives 
from Proposition \ref{bernsteinderivative}
we arrive at the following
\begin{theorem}\label{elevationtheorem}
Let $\lambda = (\lambda_1,...,\lambda_n)$ be a partition of 
length at most $n$ and let 
$\mu =(\mu_1,...,\mu_{n+1})$ be a dimension elevation partition 
of $\lambda$. 
Denote by $(B_{0,\lambda}^{n},...,B_{n,\lambda}^{n})$ (resp 
$(B_{0,\mu}^{n+1},...,B_{n+1,\mu}^{n+1})$) the 
Chebyshev-Bernstein basis of $\mathcal{E}_{\lambda}(n)$(resp. 
$\mathcal{E}_{\mu}(n+1)$) over an interval $[a,b]$. Then, we have
$$
B_{k,\lambda}^n(t) =
\frac{(n+1-k)}{n+1} {\binom{\mu}{\lambda}}_{\hskip -0.1cm n} a^{\rho}
\Gamma_{\lambda}^{\mu}(n,k)  B_{k,\mu}^{n+1}(t) 
+\frac{(k+1)}{n+1} {\binom{\mu}{\lambda}}_{\hskip -0.1cm n}
b^{\rho} \Delta_{\lambda}^{\mu}(n,k) B_{k+1,\mu}^{n+1}(t).
$$
where $\rho = \lambda_1 - \mu_1$.
\end{theorem}   

To illustrate the use of the last Theorem as a mean of finding explicit expression
for the Chebyshev-Bernstein basis, consider the $r$th elementary \muntz space
i.e., the \muntz space associated with the partition $\lambda = (1^{r})$.
The partition $\mu = (0)$ is a dimension elevation partition of $\lambda$ whose
Chebyshev-Bernstein basis over an interval $[a,b]$ is given by the classical
polynomial Bernstein basis of order $n+1$ over the interval $[a,b]$. We have 
$\lambda^{(0)} = (1^{r-1})$, $\mu = \mu^{(0)} = (0)$ and  
${\binom{\lambda}{\mu}}_{\hskip -0.1cm n} = (n+1)/r$.
Therefore, applying Theorem \ref{elevationtheorem} leads to
$$
B_{k,(1^r)}^n(t) = \frac{e_{r-1}(a^{n-k},b^{k})}{r} \left( \frac{(n+1-k)a}{e_r(a^{n+1-k},b^k)}
B_{k}^{n+1}(t) + 
\frac{(k+1)b}{e_r(a^{n-k},b^{k+1})}
B_{k+1}^{n+1}(t) \right).   
$$
This illustrative example prompt us to consider the following 
algorithm for computing the Chebyshev-Bernstein basis associated
with a partition $\lambda$. We can construct a sequence 
of nested \muntz spaces $\mathcal{E}_{\mu_{(j)}}(n+j), j=0,...,m$ such that
\begin{equation}\label{nested}
\mathcal{E}_{\lambda}(n) = \mathcal{E}_{\mu _{(0)}}(n)
\subset \mathcal{E}_{\mu _{(1)}}(n+1) \subset ....\subset 
\mathcal{E}_{\mu_{(m-1)}}(n+m-1) \subset
\mathcal{E}_{\mu _{(m)}}(n+m),
\end{equation}
where the partition $\mu _{(m)} =0$.    
As the Chebyshev-Bernstein basis of the space 
$\mathcal{E}_{\mu_{(m)}}(n+m)$ is the classical 
Bernstein basis of degree $n+m$, we can construct iteratively
the Chebyshev-Bernstein  bases starting from the space 
$\mathcal{E}_{\mu_{(m-1)}}(n+m-1)$ until reaching the space 
$\mathcal{E}_{\lambda}(n)$ using Theorem \ref{elevationtheorem}.

There is several sequences of nested spaces that satisfies (\ref{nested}),
starting from the space $\mathcal{E}_{\lambda}(n)$ and in accordance with the constraints 
of Proposition \ref{elevationproposition}. For example we have the following two sequences 
of nested \muntz spaces
$$
\mathcal{E}_{\tiny{\yng(3,1)}}(n) \subset \mathcal{E}_{\tiny{\yng(2)}}(n+1)
\subset \mathcal{E}_{\tiny{\yng(1)}}(n+2)
\subset \mathcal{E}_{\emptyset}(n+3)
$$ 
and
$$
\mathcal{E}_{\tiny{\yng(3,1)}}(n) \subset \mathcal{E}_{\tiny{\yng(2,2,1)}}(n+1)
\subset \mathcal{E}_{\tiny{\yng(1,1)}}(n+2)
\subset \mathcal{E}_{\emptyset}(n+3)
$$
which correspondent respectively to the following situation of nested \muntz spaces
\begin{equation*}
\begin{split}
& \qquad (1,t^3,t^5,t^6,...,t^{n+3})\subset span(1,t^3,t^4,t^5,t^6,...,t^{n+3}) \subset \\ 
& span(1,t^2,t^3,t^4,t^5,t^6,...,t^{n+3})
\subset span(1,t,t^2,t^3,t^4,t^5,t^6,...,t^{n+3})
\end{split}
\end{equation*}
and 
\begin{equation*}
\begin{split}
& \qquad (1,t^3,t^5,t^6,...,t^{n+3})\subset span(1,t,t^3,t^5,t^6,...,t^{n+3}) \subset \\ 
& span(1,t,t^3,t^4,t^5,t^6,...,t^{n+3})
\subset span(1,t,t^2,t^3,t^4,t^5,t^6,...,t^{n+3}).
\end{split}
\end{equation*}

\begin{remark} 
In some circumstances, it is not necessary to have a full descent of nested
\muntz spaces as in (\ref{nested}) to compute the Chebyshev-Bernstein basis of a 
specific \muntz space.
We can sometimes use the staircase \muntz spaces as a short-cut space for the 
computation. For example, consider the \muntz space $E = span(1,t^2,t^6,t^8)$.
The partition associated with this space is given by $\lambda = (5,3,1)$.
To compute its associated Bernstein-Chebyshev basis, we can make the dimension elevation 
$E = span(1,t^2,t^6,t^8) \subset F = span(1,t^2,t^4,t^6,t^8)$.
As the Chebyshev-Bernstein basis of $F$ over an interval $[a,b]$ is known 
in terms the classical Bernstein basis over the interval $[a^2,b^2]$, 
we can use Theorem \ref{elevationtheorem} to find the Chebyshev-Bernstein 
basis of the space $F$ in a single iteration. In principle, we can use this trick 
to compute the Chebyshev-Bernstein basis of any \muntz space whose components
are a ``reparametrization'', in the sense of (\ref{reparametrizationfunction}),
of an already studied \muntz space.  
\end{remark} 

In the following, we will show that there is a particularly convenient choice of a 
sequence of nested \muntz spaces in which an inductive argument along the sequence 
will give us the explicit expression of the Chebyshev-Bernstein basis.    
%%%%%
%%%%%
\begin{definition}
Let $\lambda$ be a partition of length at most $n$. We define the border 
complement $\eta$ of the partition $\lambda$ as the partition obtained by 
removing the first column of $\lambda$. In other word, if $\lambda$ is given 
by $\lambda = (\lambda_1,\lambda_2,...,\lambda_s,0,..0)$ 
where $\lambda_i \geq 1$ for $i=1,...,s$, then its border complement 
is given by $\eta = (\lambda_1-1,\lambda_2-1,...,\lambda_s-1,0,...,0).$
\end{definition}
%%%%%
%%%%%
It is clear from proposition \ref{elevationproposition} 
that if $\lambda$ is a partition of length at  most $n$ 
and $\eta$ its border complement then $\eta$ is a dimension elevation 
partition of $\lambda$ i.e.,  
$\mathcal{E}_{\lambda}(n) \subset \mathcal{E}_{\eta}(n+1).$ 
A first hint of the usefulness of this choice is the following combinatorial 
lemma 
\begin{lemma}\label{lemma2}
Let $\lambda$ be a non-empty partition of length at most $n$ and let 
$\eta$ be it border complement. Then, we have 
$$
{\binom{\eta}{\lambda}}_{\hskip -0.1cm {n}} = \frac{f_{\lambda}(n+1) f_{\eta^{(0)}}(n+1)}
{f_{\eta}(n+2) f_{\lambda^{(0)}}(n)} = 
\frac{n+1}{h_{\lambda}(1,1)},
$$
where $h_{\lambda}(1,1)$ is the content of the first square
of the partition $\lambda$.
\end{lemma}

\begin{proof}
From Lemma \ref{hooklemma}, we have 
$$
\frac{f_{\lambda}(n+1)}{f_{\lambda^{(0)}}(n)} = 
\prod_{j=1}^{\lambda_1} \frac{(n+1) + c_{\lambda}(1,j)}{h_{\lambda}(1,j)} 
$$ 
and 
$$
\frac{f_{\eta}(n+2)}{f_{\eta^{(0)}}(n+1)} = 
\prod_{j=1}^{\lambda_1-1} \frac{(n+2) + c_{\eta}(1,j)}{h_{\eta}(1,j)}.
$$
Moreover, from the definition of $\eta$, we have for $j=1,..,\lambda_1-1$ 
$$
c_{\eta}(1,j) = c_{\lambda}(1,j+1)-1 
\quad \textnormal{and} \quad
h_{\eta}(1,j) = h_{\lambda}(1,j+1).
$$
Using this extra information in the computation 
leads to a proof of the lemma.
\end{proof} 

With a border complement partition as a dimension elevation partition, 
Theorem \ref{elevationtheorem} takes the simpler form 

\begin{proposition}\label{elevationproposition2}
Let $\lambda$ be a partition of length at most $n$, and let 
$\eta$ be its border complement. Then, the Chebyshev-Bernstein basis 
associated with $\lambda$ and $\eta$ over an interval $[a,b]$ are related by   
$$
B_{k,\lambda}^n(t) =
\frac{(n+1-k)a}{h_{\lambda}(1,1)} \Gamma_{\lambda}^{\eta}(n,k)  B_{k,\eta}^{n+1}(t) 
+\frac{(k+1)b}{h_{\lambda}(1,1)} \Delta_{\lambda}^{\eta}(n,k) 
B_{k+1,\eta}^{n+1}(t).
$$
\end{proposition}
  
We need the following proposition and the subsequent corollary in order 
to give a proof of the main Theorem \ref{bernsteintheorem}
\begin{proposition}\label{propositionfrommuntz}
Let $\lambda$ be a partition of length at most $n$ and $\mu$ 
it border complement.  
Then, for any real numbers $U=(u_1,...,u_{n-1})$, and 
real numbers $x$ and $y$, we have
$$
S_{\mu^{(0)}}(U,x) S_{\lambda}(U,y) - S_{\mu^{(0)}}(U,y) S_{\lambda}(U,x) =
(y-x) S_{\mu}(U,x,y) S_{\lambda^{(0)}}(U),
$$
where $\lambda^{(0)}$ (resp. $\mu^{(0)}$) the bottom partition of $\lambda$ (resp. $\mu$).
\end{proposition}    

\begin{proof}
Let us first assume that the partition $\lambda$ is of exact length $n$. 
Using (\ref{schurplusones}), we have 
$$
S_{\lambda}(U,y) = y \left( \prod_{i=1}^{n-1} u_i \right) S_{\mu}(U,y), \quad
S_{\lambda}(U,x) = x \left(\prod_{i=1}^{n-1} u_i \right) S_{\mu}(U,x),
$$
and
$$
S_{\lambda^{(0)}}(U) = \left(\prod_{i=1}^{n-1} u_i \right) S_{\mu^{(0)}}(U).
$$
In this case the statement of the proposition is nothing by Proposition 
\ref{condensationproposition}. Now, let us assume that $\lambda$ 
is of exact length $k < n$ i.e., $\lambda = (\lambda_1,...,\lambda_k,0,...,0)$
with $\lambda_k \geq 1$.
Consider the \muntz tableau $(\mu^{(0)},\mu^{(1)},...,\mu^{(n)})$ associated 
with the partition $\mu$. Then, we have, $\mu^{(k)} = \lambda$. Therefore, applying
Proposition \ref{condensationmuntztableau} 
to the partition $\mu$ leads to a proof of the proposition. 
\end{proof}

\begin{corollary}\label{corollarymain}
Let $\lambda= (\lambda_1,\lambda_2,...,\lambda_n)$ be a partition of length at most $n$ and $\mu$ 
it border complement.  
Then, for any positive real numbers $U=(u_1,...,u_{n-1})$, and 
positive real numbers $a, b$ and $t$, we have
$$
a(b-t) t^{\lambda_1-1} S_{\mu}(U,a,\frac{ab}{t}) S_{\lambda}(U,b) +
b(t-a) t^{\lambda_1-1} S_{\mu}(U,b,\frac{ab}{t}) S_{\lambda}(U,a) = 
$$
$$
(b-a) t^{\lambda_1} S_{\mu}(U,a,b) S_{\lambda}(U,\frac{ab}{t}).
$$
\end{corollary}

\begin{proof}
Applying Proposition \ref{propositionfrommuntz},
with $x=a$ and $y = ab/t$ and multiplying both sides 
by $t^{\lambda_1}$ give
$$
a(b-t) t^{\lambda_1 - 1} S_{\mu}(U,a,ab/t) S_{\lambda^{(0)}}(U) =
$$
$$
t^{\lambda_1} \left( S_{\mu^{(0)}}(U,a) S_{\lambda}(U,ab/t) - 
S_{\mu^{(0)}}(U,ab/t) S_{\lambda}(U,a) \right).
$$ 
Therefore, we have 
$$
a(b-t) t^{\lambda_1 - 1} S_{\mu}(U,a,ab/t) S_{\lambda}(U,b) =
$$
\begin{equation}\label{localequation1}
t^{\lambda_1} \frac{S_{\lambda}(U,b)}{S_{\lambda^{(0)}}(U)}
\left( S_{\mu^{(0)}}(U,a) S_{\lambda}(U,ab/t) - 
S_{\mu^{(0)}}(U,ab/t) S_{\lambda}(U,a) \right).
\end{equation}
Similarly, by applying Proposition \ref{propositionfrommuntz}
with $x=b$ and $y = ab/t$, we have 
$$
b(t-a) t^{\lambda_1 - 1} S_{\mu}(U,b,ab/t) S_{\lambda}(U,a) =
$$
\begin{equation}\label{localequation2}
t^{\lambda_1} \frac{S_{\lambda}(U,a)}{S_{\lambda^{(0)}}(U)}
\left(  S_{\mu^{(0)}}(U,ab/t) S_{\lambda}(U,b)-
S_{\mu^{(0)}}(U,b) S_{\lambda}(U,ab/t) \right).
\end{equation}
Summing up the expressions (\ref{localequation1}) and 
(\ref{localequation2}) shows that the right hand 
side of the equation in the corollary is given by 
$$
\frac{t^{\lambda_1} S_{\lambda}(U,ab/t)}{S_{\lambda^{(0)}}(U)} 
\left( S_{\lambda}(U,b) S_{\mu^{(0)}} (U,a) -
S_{\lambda}(U,a) S_{\mu^{(0)}}(U,b) \right).
$$
Using again Proposition \ref{propositionfrommuntz},
to the quantity between the parenthesis in the last equation, 
with $x = a$ and $y = b$, leads to the desired formulas
\end{proof} 

\bigskip
{\noindent{{\bf Proof of the main Theorem \ref{bernsteintheorem}:}}} 
\begin{proof}
We will proceed by induction on the number of boxes in the partition
$\lambda$. For an empty partition, the Chebyshev-Bernstein basis is 
given by the classical Bernstein basis which is consistent 
with the formula given by the Theorem.
Let us assume that the Theorem is true for any 
partition with less than $m$ boxes. For a given partition $\lambda$ with 
$m$ boxes, let us denote by $\mu$ its border complement partition.
Then, necessary the number of boxes in $\mu$ is less than $m$. 
By Proposition \ref{elevationproposition2}, we have 
\begin{equation}\label{localmain}
\begin{split}
B^{n}_{k,\lambda}(t) & = \frac{(n+1-k)a}{h_\lambda(1,1)}
\frac{S_{\lambda^{(0)}}(a^{n-k},b^{k}) S_{\mu}(a^{n+2-k},b^{k})}
{S_{\mu^{(0)}}(a^{n+1-k},b^{k}) S_{\lambda}(a^{n+1-k},b^{k})}
B^{n+1}_{k,\mu}(t) +  \\
& \frac{(k+1)b}{h_\lambda(1,1)}
\frac{S_{\lambda^{(0)}}(a^{n-k},b^{k}) S_{\mu}(a^{n-k},b^{k+2})}
{S_{\mu^{(0)}}(a^{n-k},b^{k+1}) S_{\lambda}(a^{n-k},b^{k+1})}
B^{n+1}_{k+1,\mu}(t).
\end{split}
\end{equation}
Using the induction hypothesis (\ref{bernstein}) on the Chebyshev-Bernstein functions 
$B^{n+1}_{k+1,\mu}(t)$ and $B^{n+1}_{k,\mu}(t)$ and carrying out all the obvious  
simplifications as well as using Lemma \ref{lemma2}, we find that 
the first term in (\ref{localmain}) is given by 
\begin{equation}
\frac{f_{\lambda}(n+1)}{f_{\lambda^{(0)}}(n)}
\frac{a(b-t)}{b-a} B^{n}_{k}(t) t^{\lambda_1-1}
\frac{S_{\lambda^{(0)}}(a^{n-k},b^{k}) S_{\mu}(a^{n+1-k},b^{k},ab/t)}
{S_{\lambda}(a^{n+1-k},b^{k}) S_{\mu}(a^{n+1-k},b^{k+1})},
\end{equation}
while the second term in (\ref{localmain}) is given by
\begin{equation}
\frac{f_{\lambda}(n+1)}{f_{\lambda^{(0)}}(n)}
\frac{b(t-a)}{b-a} B^{n}_{k}(t) t^{\lambda_1-1}
\frac{S_{\lambda^{(0)}}(a^{n-k},b^{k}) S_{\mu}(a^{n-k},b^{k+1},ab/t)}
{S_{\lambda}(a^{n-k},b^{k+1}) S_{\mu}(a^{n+1-k},b^{k+1})}.
\end{equation}
Summing the two last equations leads to 
\begin{equation}\label{localbernstein}
B_{k,\lambda}^n(t) = 
\frac{f_{\lambda}(n+1)}{f_\lambda^{(0)}(n)} B^{n}_{k}(t) 
\frac{S_{\lambda^{(0)}}(a^{n-k},b^{k})}
{S_{\mu}(a^{n+1-k},b^{k+1})} \mathfrak{U}(a,b,t),
\end{equation}
where $\mathfrak{U}$ is given by
\begin{equation*}
\begin{split}
\mathfrak{U}(a,b,t) = &   
\frac{t^{\lambda_1 - 1}}{b-a} ( 
\frac{a(b-t) S_{\mu}(a^{n+1-k},b^{k},ab/t) 
S_{\lambda}(a^{n-k},b^{k+1})}{S_{\lambda}(a^{n+1-k},b^{k}) 
S_{\lambda}(a^{n-k},b^{k+1})} + \\ 
& \frac{b(t-a) S_{\mu}(a^{n-k},b^{k+1},ab/t) 
S_{\lambda}(a^{n+1-k},b^{k})}{S_{\lambda}(a^{n+1-k},b^{k}) 
S_{\lambda}(a^{n-k},b^{k+1})} ).
\end{split}
\end{equation*}
Using Corollary \ref{corollarymain}, with $U = (a^{n-k},b^{k})$ (note that here we view $\lambda$ as a partition
of length at most $(n+1)$ to be able to take $U$ with $n$ components) shows that 
\begin{equation}
\mathfrak{U}(a,b,t) =
\frac{ t^{\lambda_1} S_{\mu}(a^{n-k+1},b^{k+1}) 
S_{\lambda}(a^{n-k},b^{k},ab/t)}{S_{\lambda}(a^{n+1-k},b^{k}) 
S_{\lambda}(a^{n-k},b^{k+1})}.  
\end{equation}  
Inserting the last term into equation (\ref{localbernstein}) result 
in the proof of the Theorem.
\end{proof}

\begin{remark}
One we have guessed the explicit expression of the Chebyshev-Bernstein 
basis, it is in principle, possible to find a simpler proof than the one
given here. For instance, according to the characterization of the Chebyshev-Bernstein 
basis given in \cite{Mazure2} and in view of Proposition \ref{bernsteinderivative},
we need only to show that every element $B^{n}_{k,\lambda}$ expressed in (\ref{bernstein})
is an element of the \muntz space $\mathcal{E}_{\lambda}(n)$. However, the advantage of 
our proof lies in demonstrating the elegant combinatorics beneath the relations of 
Chebyshev-Bernstein bases associated with different partitions.
\end{remark}

%%%%%%%%%%%%%%%%%%%%%%%%%%%%%%%%%%%%%%%%%%%%%%%%%
%%%%%%%%%%%%%%%%%%%%%%%%%%%%%%%%%%%%%%%%%%%%%%%%%
\subsection{Dimension elevation process}
%%%%%%%%%%%%%%%%%%%%%%%%%%%%%%%%%%%%%%%%%%%%%%%%%
%%%%%%%%%%%%%%%%%%%%%%%%%%%%%%%%%%%%%%%%%%%%%%%%%
Let $\lambda = (\lambda_1,\lambda_2,...,\lambda_n)$ be a partition 
of length at most $n$ and let $\eta = (\eta_1,\eta_2,...,\eta_{n+1})$
be a dimension elevation partition of $\lambda$. 
Consider a $\mathcal{E}_{\lambda}(n)$-function $P$ written 
in the Chebyshev-Bernstein bases associated with the partitions
$\lambda$ and $\eta$ over an interval $[a,b]$ as  
\begin{equation}\label{expansion1}
P(t) = \sum_{k=0}^{n} B_{k,\lambda}^{n}(t) P_{k} =
\sum_{k=0}^{n+1} B_{k,\eta}^{n+1}(t) \tilde{P}_{k}.
\end{equation}
Using Theorem \ref{elevationtheorem} to detect the coefficients of $B_{k,\eta}^{n+1}(t)$
in the expansion (\ref{expansion1}), we readily find 
%%%%%%%%%%%%%%%%%%%%%%%%%%%%%%%%%%%%%%%
%%%%%%%%%%%%Figure 1 %%%%%%%%%%%%%%%%%% 
\begin{figure}
\hskip 2.5 cm
\includegraphics[height=6.cm]{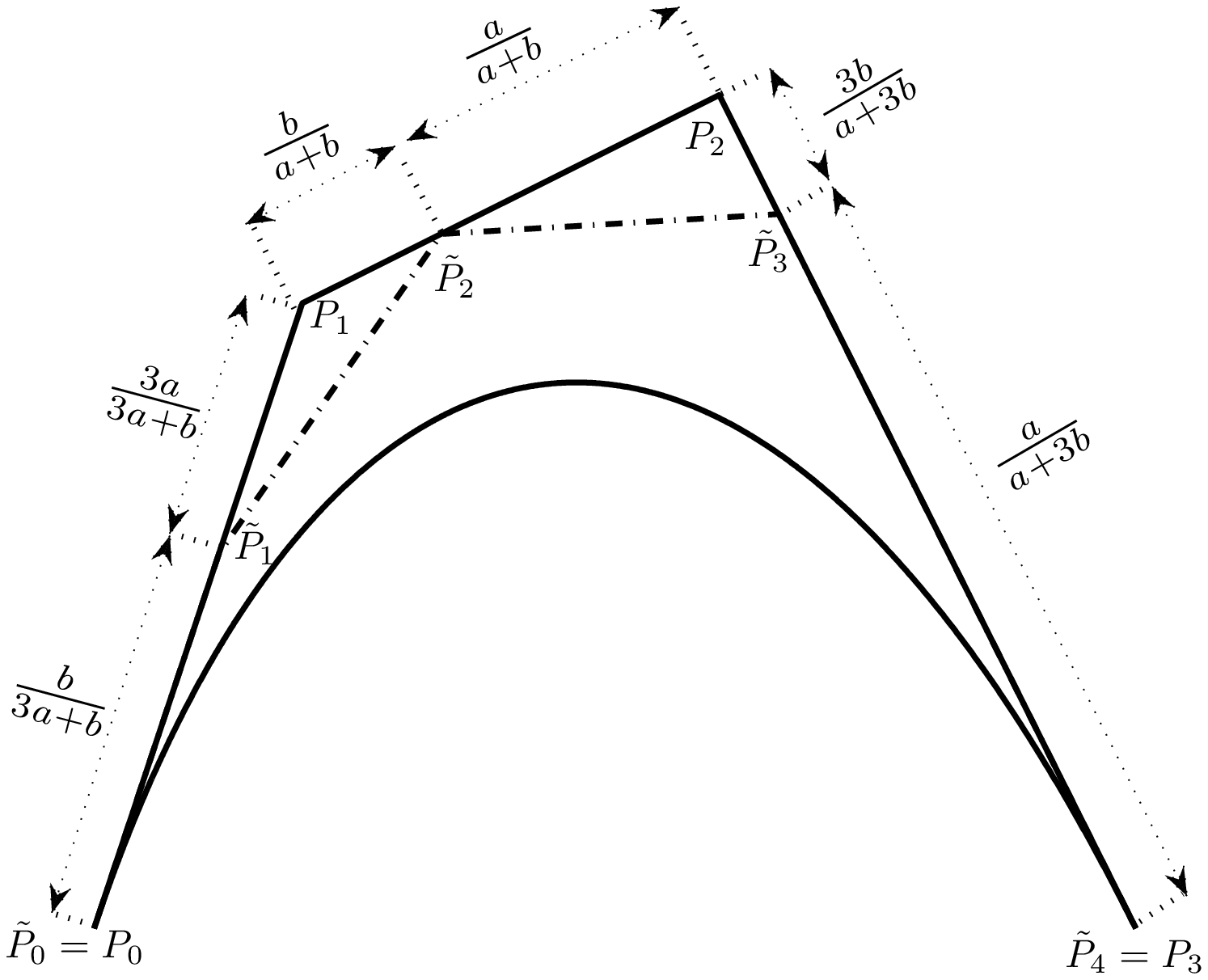}
\caption{The dimension elevation process
$\mathcal{E}_{(1)}(3) = span(1,t^2,t^3,t^4) \subset 
\mathcal{E}_{\emptyset}(4) = span(1,t,t^2,t^3,t^4)$. 
$(P_0,P_1,P_2,P_3)$ are the Chebyshev-B\'ezier points of a $\mathcal{E}_{(1)}(3)$-function,
while $(\tilde{P}_0,\tilde{P}_1,\tilde{P}_2,\tilde{P}_3,\tilde{P}_4)$
are the Chebyshev-B\'ezier points of the same function viewed as a 
$\mathcal{E}_{\emptyset}(4)$-function. 
(see example 11).}
\label{fig:figure1}
\end{figure}
%%%%%%%%%%%%%%%%%%%%%%%%%%%%%%%%%%%%%%%
%%%%%%%%%%%%%%%%%%%%%%%%%%%%%%%%%%%%%%%
\begin{theorem}\label{elevationtheorem2}
The Chebyshev-B\'ezier points $\tilde{P_k}$ in (\ref{expansion1})
are related to the Chebyshev-B\'ezier points $P_{k}$ by the relations 
$$
\tilde{P}_0 = P_0, \quad \tilde{P}_{n+1} = P_{n},
$$
and for $k=1,2,...,n$
\begin{equation}\label{elevationequation}
\tilde{P}_k = \rho_{[\lambda,\eta]}(n,k-1) \; P_{k-1} + \xi_{[\lambda,\eta]}(n,k) \; P_{k},
\end{equation}
where $\xi_{[\lambda,\eta]}(n,k)$  and $\rho_{[\lambda,\eta]}(n,k)$ are given by 
$$
\xi_{[\lambda,\eta]}(n,k) = \frac{(n+1-k) a^{\lambda_1-\eta_1}}{n+1}
{\binom{\eta}{\lambda}}_{\hskip -0.1cm {n}} \Gamma_{\lambda}^{\eta}(n,k)
$$
and 
$$
\rho_{[\lambda,\eta]}(n,k) = \frac{(k+1) b^{\lambda_1-\eta_1}}{n+1}
{\binom{\eta}{\lambda}}_{\hskip -0.1cm {n}} \Delta_{\lambda}^{\eta}(n,k),
$$
where $\Gamma_{\lambda}^{\eta}(n,k)$ and $\Delta_{\lambda}^{\eta}(n,k)$ are
defined in (\ref{factor1}) and (\ref{factor2}).
\end{theorem}

\begin{remark}
As the relation (\ref{elevationequation}) is independent of the 
$\mathcal{E}_{\lambda}(n)$-function $P$ and in view of 
(\ref{generalelevation}), we have 
$\rho_{\lambda,\eta}(n,k-1) + \xi_{\lambda,\eta}(n,k) = 1.$
This relation can also be directly proven (with rather great efforts)
using Proposition \ref{propositionfrommuntz}.
\end{remark}

\begin{example}
Let $P$ be a $\mathcal{E}_{(1^r)}(n)$-function. The partition $\eta = (0)$
is a dimension elevation partition of $\lambda$. Therefore,
the function $P$ can be expressed as 
\begin{equation*}
P(t) = \sum_{k=0}^{n} B_{k,(1^{r})}^{n}(t) P_{k} =
\sum_{k=0}^{n+1} B_{k}^{n+1}(t) \tilde{P}_{k}.
\end{equation*}
In this case, we have 
$$
\rho_{[1^{(r)},(0)]}(n,k-1) = \frac{k b}{r} 
\frac{e_{r-1}(a^{n-k+1},b^{k-1})}{e_{r}(a^{n-k+1},b^{k})} 
$$
and 
$$
\xi_{[1^{(r)},(0)]}(n,k) = \frac{(n+1-k) a}{r} 
\frac{e_{r-1}(a^{n-k},b^{k})}{e_{r}(a^{n-k+1},b^{k})}. 
$$
Therefore, from Theorem \ref{elevationtheorem2}, we have
$\tilde{P}_{0} = P_{0}$ and $\tilde{P}_{n+1} = P_{n}$ and for $k=1,...,n$, 
$$
\tilde{P}_k = 
\frac{k b e_{r-1}(a^{n-k+1},b^{k-1})} {r e_{r}(a^{n-k+1},b^{k})} P_{k-1} +
\frac{(n+1-k)a e_{r-1}(a^{n-k},b^{k})}{{r e_{r}(a^{n-k+1},b^{k})} } P_{k}.
$$
The case $r = 1$ provide us with the following simple relationships 
$$
\tilde{P}_k = \frac{kb}{(n+1-k)a + kb} P_{k-1} + \frac{(n+1-k)a}{(n+1-k)a + kb} P_{k}.
$$
Figure \ref{fig:figure1} shows an example of dimension elevation process 
for the case $r = 1$.
\end{example}
%%%%%%%%%%%%%%%%%%%%%%%%%%%%%%%%%%%%%%%
%%%%%%%%%%%%Figure 2 %%%%%%%%%%%%%%%%%% 
\begin{figure}
\hskip 2.5 cm
\includegraphics[height=6.cm]{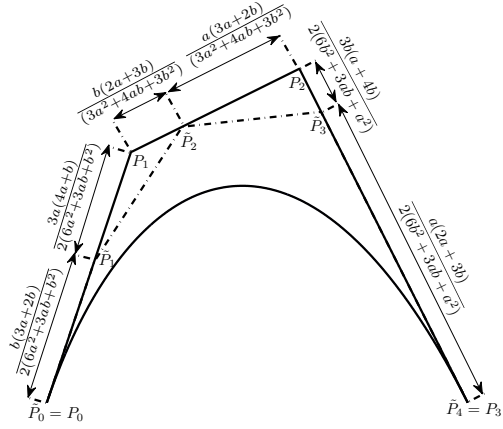}
\caption{The dimension elevation process
$\mathcal{E}_{(1,1)}(3) = span(1,t^3,t^4,t^5) \subset 
\mathcal{E}_{(1)}(4) = span(1,t^2,t^3,t^4,t^5)$. 
$(P_0,P_1,P_2,P_3)$ are the Chebyshev-B\'ezier points of a $\mathcal{E}_{(1,1)}(3)$-function,
while $(\tilde{P}_0,\tilde{P}_1,\tilde{P}_2,\tilde{P}_3,\tilde{P}_4)$
are the Chebyshev-B\'ezier points of the same function viewed as a 
$\mathcal{E}_{(1)}(4)$-function. 
(see example 12).}
\label{fig:figure2}
\end{figure}
%%%%%%%%%%%%%%%%%%%%%%%%%%%%%%%%%%%%%%%
%%%%%%%%%%%%%%%%%%%%%%%%%%%%%%%%%%%%%%%
%%%%%
%%%%%
\begin{example}
Let us consider the dimension elevation process
$$
\mathcal{E}_{(l)}(n) \subset \mathcal{E}_{(l-1)}(n+1).
$$
Then, if we write 
\begin{equation}
P(t) = \sum_{k=0}^{n} B_{k,(l)}^{n}(t)  P_{k} =
\sum_{k=0}^{n+1} B_{k,(l-1)}^{n+1}(t) \tilde{P}_{k},
\end{equation}
we would have
$$
\rho_{[(l),(l-1)]}(n,k-1) = \frac{k b}{l} 
\frac{h_{l-1}(a^{n-k+1},b^{k+1})}{h_{l}(a^{n-k+1},b^{k})} 
$$
and 
$$
\xi_{[(l),(l-1)]}(n,k) = \frac{(n+1-k) a}{l} 
\frac{h_{l-1}(a^{n+2-k},b^{k})}{h_{l}(a^{n-k+1},b^{k})}. 
$$
Therefore, we have $\tilde{P}_{0} = P_{0}$ and $\tilde{P}_{n+1} = P_{n}$ and for $k=1,...,n$, 
$$
\tilde{P}_k = 
\frac{k b h_{l-1}(a^{n-k+1},b^{k+1})} {l h_{l}(a^{n-k+1},b^{k})} P_{k-1} +
\frac{(n+1-k)a h_{l-1}(a^{n+2-k},b^{k})}{{l h_{l}(a^{n-k+1},b^{k})} } P_{k}.
$$
Figure \ref{fig:figure2} shows an example of the dimension elevation process
for the case $l=2$.
\end{example}
\section{Toward shaping with Young diagram}
A polygon $\mathcal{P} = (P_0,P_1,...,P_n)$ can be viewed as the control 
polygon of a $\mathcal{E}_{\lambda}(n)$-function, where $\lambda$
is a partition of length at most $n$. Therefore, by varying the partition 
$\lambda$, the curve associated with the control polygon will also vary 
accordingly. In such circumstances, the Young diagram can be viewed as 
a shape parameter. It would, therefore, be interesting to study the effect of 
standard operations on a fixed partition $\lambda$, such as adding a box, 
removing a box, adding a row or column and so on, on the shape of the curve.
The problem is rather challenging and we will content ourself, here,
with a simple experimental example. In Figure \ref{fig:figure3}, we show the effect of
adding boxes to the first row of the partition $\lambda = (2,1)$. 
Adding boxes to the first row seems to have the effect of making the curve
more and more far from the control polygon. However, adding the same number 
of boxes to every column seems to have the opposite effect as shown in Figure \ref{fig:figure4}.
%%%%%%%%%%%%%%%%%%%%%%%%%%%%%%%%%%%%%%%
%%%%%%%%%%%%Figure 3 %%%%%%%%%%%%%%%%%% 
\begin{figure}
%\vskip -4cm
\includegraphics[height=3.8cm]{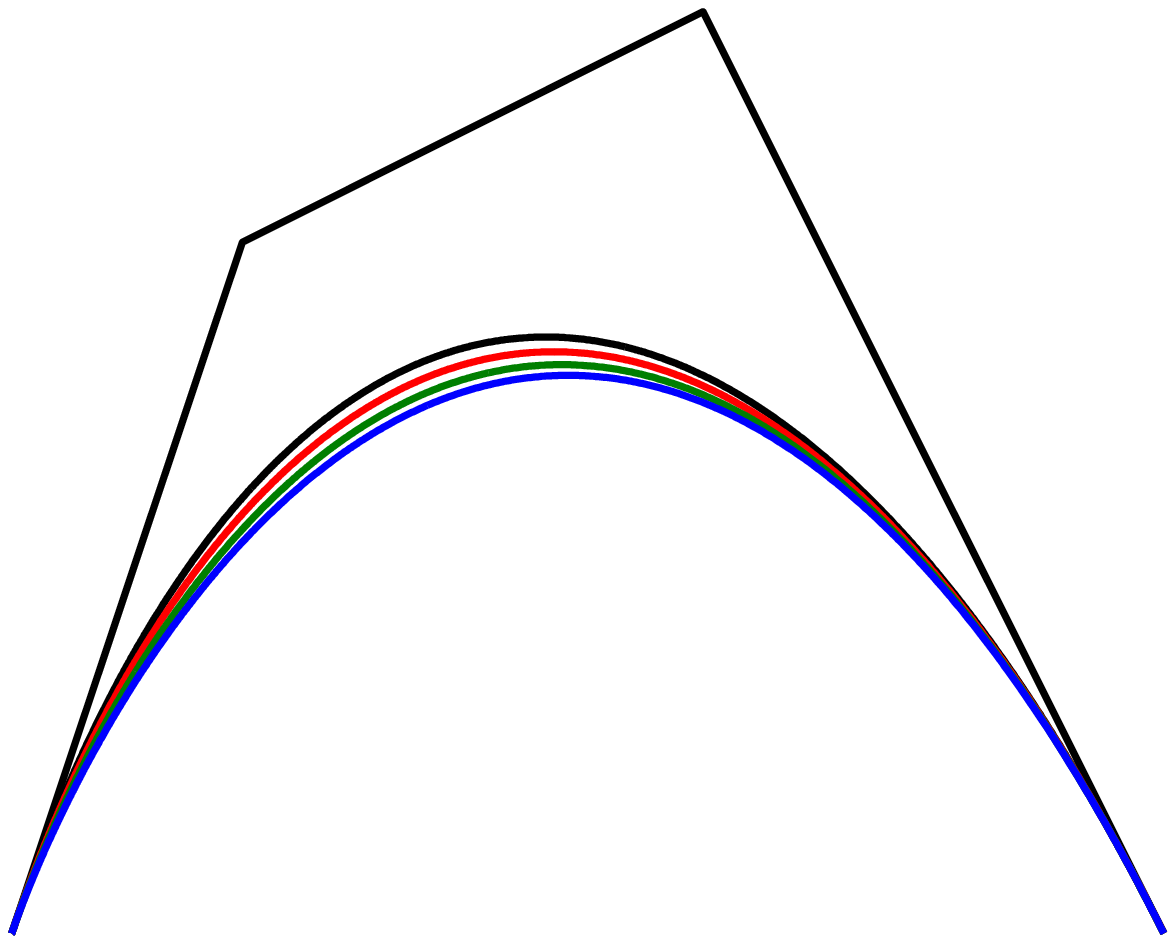}
\hskip 1.8cm
\includegraphics[height=3.6cm]{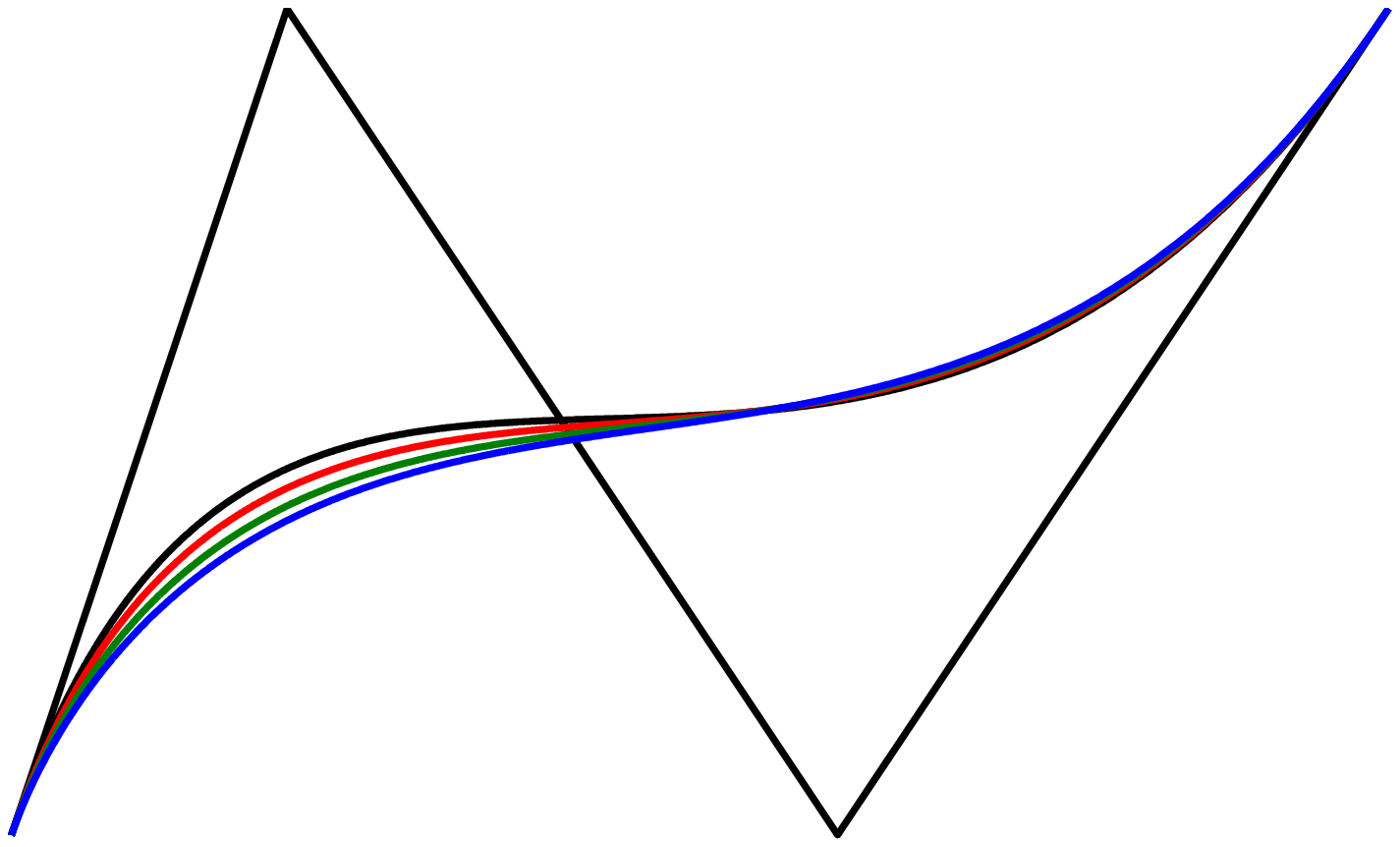}
%\vskip -2cm
\caption{The effect of iteratively adding boxes to the first row of the partition 
$\lambda = (2,1)$. The black curve refers to the Chebyshev-B\'ezier curve over 
the interval $[a,b] = [1,4]$ associated with the partition $\lambda$,
the red curve correspond to the partition $(3,1)$, the green curve to the partition
$(4,1)$ and the blue curve to the partition $(5,1)$.} 
\label{fig:figure3}
\end{figure}
%%%%%%%%%%%%%%%%%%%%%%%%%%%%%%%%%%%%%%%%%%%
%%%%%%%%%%%%%%%%%%%%%%%%%%%%%%%%%%%%%%%%%%%

We can also define the tensor-product surfaces based on the Chebyshev-Bernstein basis
associated with two different partitions. Namely, 
we can define a surface $\Gamma_{\lambda,\mu}$ by the parametric equation 
\begin{equation}\label{surf}
\Gamma_{\lambda,\mu}(s,t) = \sum_{i,j=1}^{n} B_{i,\lambda}^{n}(t) B_{j,\mu}^{n}(s) P_{ij},
\end{equation}
where $\lambda$ and $\mu$ are partitions of length at most $n$, 
$P_{ij}$ are points in $\mathbb{R}^3$ and $(s,t) \in [a,b] \times [c,d]$.
Figure \ref{fig:figure5} shows an example of surfaces obtained from (\ref{surf}).     

%%%%%%%%%%%%%%%%%%%%%%%%%%%%%%%%%%%%%%%
%%%%%%%%%%%%Figure 4 %%%%%%%%%%%%%%%%%% 
\begin{figure}
\vskip 0.2 cm
\includegraphics[height=3.9cm]{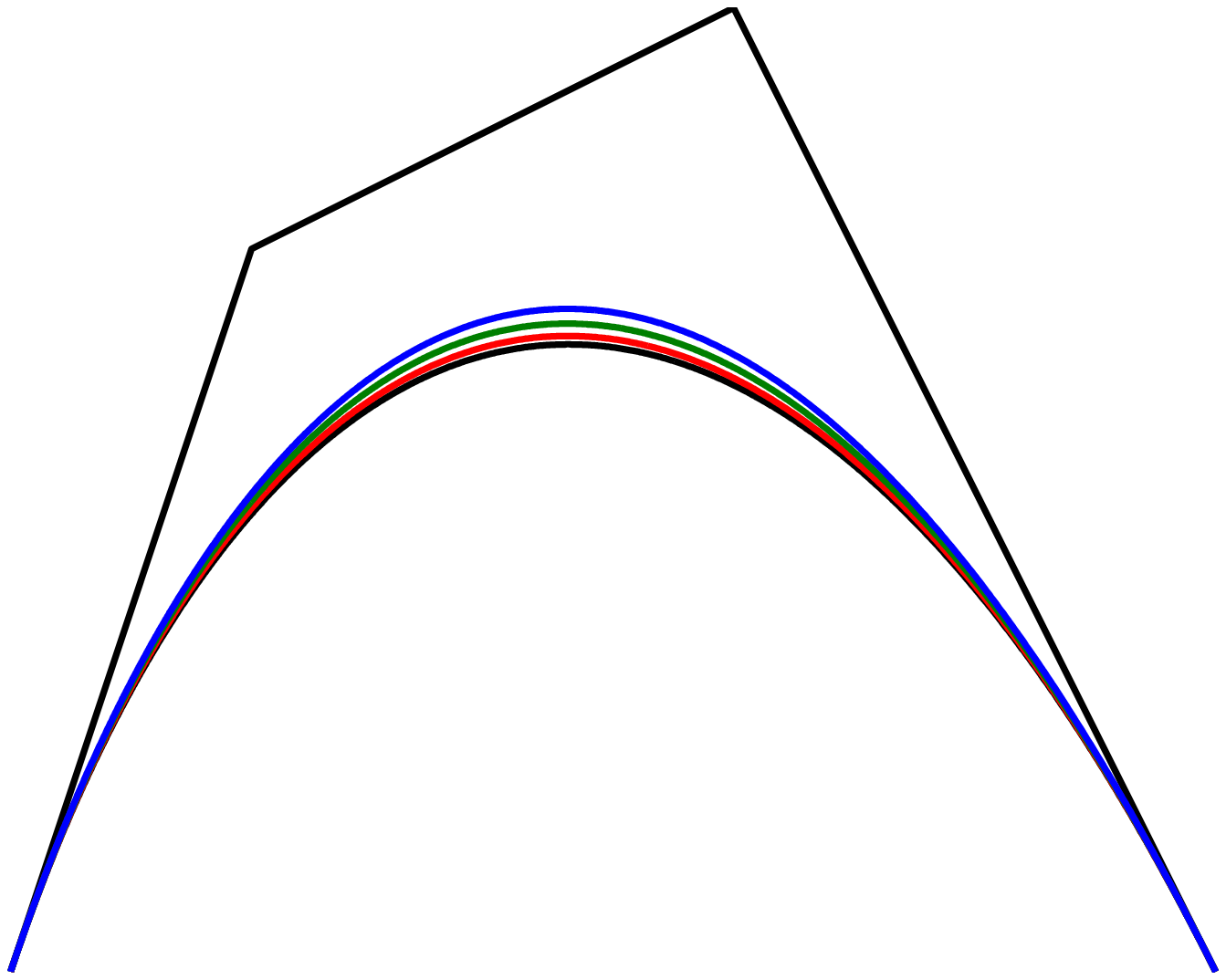}
\hskip 1.cm
\includegraphics[height=3.6cm]{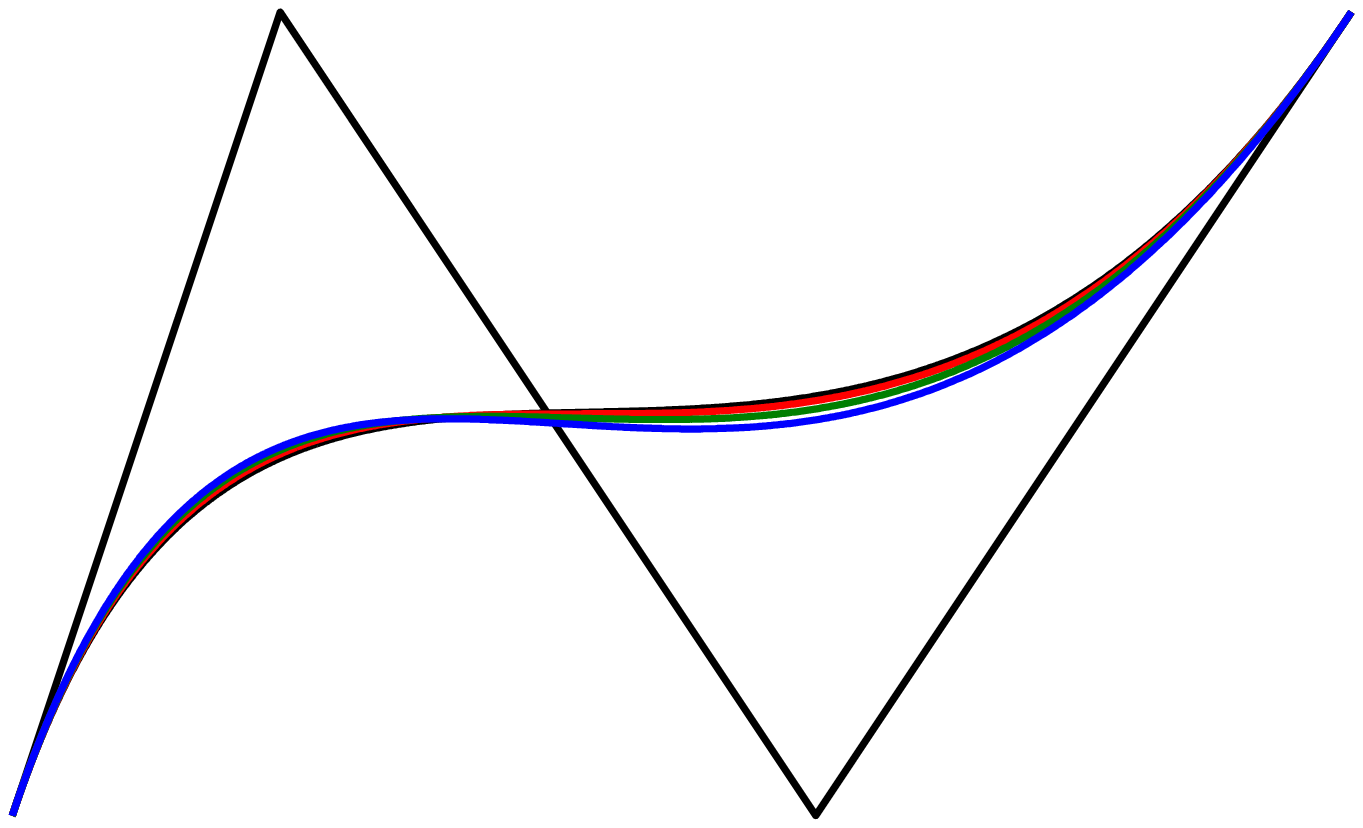}
\caption{The effect of iteratively adding a column to the partition 
$\lambda = (2,1)$. The black curve refers to the Chebyshev-B\'ezier
curve over the interval $[a,b] = [1,4]$ associated with the partition $\lambda$,
the red curve correspond to the partition $(3,2,1)$, the green curve to the partition $(4,3,2)$
and the blue curve to the partition $(5,4,3)$.}
\label{fig:figure4}
\end{figure}
%%%%%%%%%%%%%%%%%%%%%%%%%%%%%%%%%%%%%%%
%%%%%%%%%%%%%%%%%%%%%%%%%%%%%%%%%%%%%%%

%%%%%%%%%%%%%%%%%%%%%%%%%%%%%%%%%%%%%%%
%%%%%%%%%%%%Figure 5 %%%%%%%%%%%%%%%%%% 
\begin{figure}
\hskip 2 cm
\includegraphics[height=4cm]{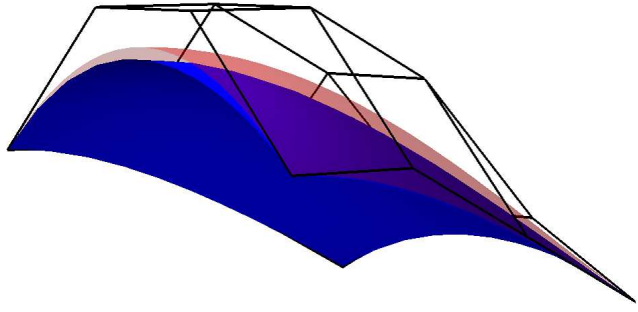}
\caption{Tensor-product surfaces obtained using equation (\ref{surf}): 
the transparent surface is associated with the parameters $\lambda = (2,1)$, $\mu = (1,1)$ and $[a,b]=[c,d]=[3,4]$,
while the blue surface is associated with the parameters $\lambda = (5,1)$, $\mu = (5,1)$ and $[a,b]=[c,d]=[1,6]$}. 
\label{fig:figure5}
\end{figure}
%%%%%%%%%%%%%%%%%%%%%%%%%%%%%%%%%%%%%%%
%%%%%%%%%%%%%%%%%%%%%%%%%%%%%%%%%%%%%%%
     
\subsection{Derivative the Chebyshev-Bernstein Basis}
Consider a partition $\lambda= (\lambda_1,\lambda_2,...,\lambda_n)$
of length at most $n$ in which we assume that 
\begin{equation}\label{condition1}
\lambda_1 = \lambda_2.
\end{equation}
Under the condition (\ref{condition1}), the derivative of the Chebyshev-Bernstein element 
$B^{n}_{k,\lambda}$ is an element of the Chebyshev space $\mathcal{E}_{\lambda^{(0)}}(n-1)$
where $\lambda^{(0)}$ is the bottom partition of $\lambda$. Therefore, the derivative 
of $B^{n}_{k,\lambda}$ can be written as linear combination of the Chebyshev-Bernstein 
basis of $\mathcal{E}_{\lambda^{(0)}}(n-1)$. Using the vanishing property of the
Chebyshev-Bernstein bases stated in Theorem \ref{mazuretheorem}, we derive as in (\ref{derivativerecurence}) 
the following relationship
\begin{equation}\label{localder}
\frac{d B_{k,\lambda}^n(t)}{dt} = 
\frac{{B_{k,\lambda}^n}^{(k)}(a)}
{{B_{k-1,\lambda^{(0)}}^{n-1}}^{(k-1)}(a)} B_{k-1,\lambda^{(0)}}^{n-1}(t) 
+ \frac{{B_{k,\lambda}^n}^{(n-k)}(b)}{{B_{k,\lambda^{(0)}}^{n-1}}^{(n-k-1)}(b)}
B_{k,\lambda^{(0)}}^{n-1}(t).
\end{equation}
in which we adopt the convention that  
$B_{-1,\lambda^{(0)}}^{n-1} \equiv B_{n,\lambda^{(0)}}^{n-1} \equiv 0$.
Let us denote by $\eta$ the partition $\eta = (\lambda_3,...,\lambda_n)$. 
The partition $\eta$ is the bottom partition of $\lambda^{(0)}$ and can 
well be written as $\eta = \left(\lambda^{(0)}\right)^{(0)}$, but for simplicity we will 
refer to this partition as $\eta$. Inserting in (\ref{localder}) the value of the derivatives from
Theorem \ref{bernsteinderivative}, we find 
$$
\frac{{B_{k,\lambda}^n}^{(k)}(a)}
{{B_{k-1,\lambda^{(0)}}^{n-1}}^{(k-1)}(a)} = 
\frac{n}{b-a} \frac{f_{\lambda}(n+1) f_{\eta}(n-1)}{f_{\lambda^{(0)}}(n)^2} 
\frac{S_{\lambda^{(0)}}(a^{n-k},b^{k}) S_{\lambda^{(0)}}(a^{n+1-k},b^{k-1})}
{S_{\lambda}(a^{n+1-k},b^{k}) S_{\eta}(a^{n-k},b^{k-1})} 
$$
and 
$$
\frac{{B_{k,\lambda}^n}^{(n-k)}(b)}
{{B_{k,\lambda^{(0)}}^{n-1}}^{(n-k-1)}(b)} = 
-\frac{n}{b-a} \frac{f_{\lambda}(n+1) f_{\eta}(n-1)}{f_{\lambda^{(0)}}(n)^2} 
\frac{S_{\lambda^{(0)}}(a^{n-k},b^{k}) S_{\lambda^{(0)}}(a^{n-1-k},b^{k+1})}
{S_{\lambda}(a^{n-k},b^{k+1}) S_{\eta}(a^{n-k-1},b^{k})}. 
$$
To write the formula for the derivative in a compact form, we define 
\begin{definition}
Let $\lambda = (\lambda_1,\lambda_2,...,\lambda_n)$
be a partition of length $n \geq 2$. Let $a,b$ be real numbers. For $k=0,...,n$,
we define the factor 
\begin{equation}\label{Requation}
R_{\lambda}(k,n) =  \frac{S_{\lambda^{(0)}}(a^{n-k-1},b^{k+1})
S_{\lambda^{(0)}}(a^{n-k},b^{k})}
{S_{\lambda}(a^{n-k},b^{k+1}) 
S_{\eta}(a^{n-k-1},b^{k})}, 
\end{equation}
where $\lambda^{(0)}=(\lambda_2,...,\lambda_n$ is the bottom partition of $\lambda$ and 
$\eta = (\lambda_3,...,\lambda_n$ is the bottom partition of $\lambda^{(0)}$.
\end{definition}

\noindent From the last definition and by noticing that 
$$
\frac{f_{\lambda}(n+1) f_{\eta}(n-1)}{f_{\lambda^{(0)}}(n)^2} = 
\frac{1}{\binom{\lambda}{\lambda^{(0)}}}_{n-1} 
$$ 
we obtain 
\begin{theorem}\label{case1lambda}
Let $\lambda = (\lambda_1,\lambda_2,\lambda_3,...,\lambda_n)$
be a partition of length at most $n$ such that $\lambda_1 = \lambda_2$.
Then the derivative of the Chebyshev-Bernstein basis associated
with the partition $\lambda$ over an interval $[a,b]$ satisfies 
$$
\frac{d B_{k,\lambda}^n(t)}{dt} = 
\frac{n}{(b-a) \binom{\lambda}{\lambda^{(0)}}}_{n-1}
\left(R_{\lambda}(k-1,n)  B_{k-1,\lambda^{(0)}}^{n-1}(t) - 
R_{\lambda}(k,n)  B_{k,\lambda^{(0)}}^{n-1}(t) \right),
$$
where $\lambda^{(0)}$ is the bottom partition of $\lambda$ 
and $R_{\lambda}(k,n)$ is defined in equation (\ref{Requation})
and in which $\eta$ is the partition $\eta = (\lambda_3,...,\lambda_n)$.
We adopt here the convention that 
$B_{-1,\lambda^{(0)}}^{n-1} \equiv B_{n,\lambda^{(0)}}^{n-1} \equiv 0$. 
\end{theorem}
Note that in the case the partition $\lambda$ is the empty partition, we recover 
the classical formulas for the derivative of the polynomial Bernstein 
basis. Let $\lambda = (\lambda_1,\lambda_2,...,\lambda_n)$ be a partition 
of length at most $n$ such that $\lambda_1 = \lambda_2$ and 
Consider a $\mathcal{E}_{\lambda}(n)$-function $P$ written 
in the Chebyshev-Bernstein basis over an interval $[a,b]$ as 
$$
P(t) = \sum_{k=0}^{n} B_{k,\lambda}^{n}(t) P_{k}.
$$
Using Theorem \ref{bernsteinderivative}, we can express 
the derivative of the function $P$ in term of the Chebyshev-Bernstein 
basis over the interval $[a,b]$ of the space $\mathcal{E}_{\lambda^{(0)}}(n-1)$.
Doing so leads to the following

\begin{theorem}
Let $\lambda = (\lambda_1,...,\lambda_n)$ be a partition 
such that $\lambda_1 = \lambda_2$ and let $P$ be a 
$\mathcal{E}_{\lambda}(n)$-function, 
written in the Chebyshev-Bernstein basis as 
$$
P(t) = \sum_{k=0}^{n} B_{k,\lambda}^{n}(t) P_{k}.
$$
Then, we have 
$$
P'(t) = \frac{n}{(b-a)\binom{\lambda}{\lambda^{(0)}}}_{n-1} 
\sum_{k=0}^{n-1} R_{\lambda}(k,n) 
B_{k,\lambda^{(0)}}^{n-1}(t) \Delta P_{k},
$$  
where $\Delta P_{i} = P_{i+1}-P_{i}$ and  $R_{\lambda}(k,n)$ is given in 
(\ref{Requation}).
\end{theorem}

\noindent In particular, we have
\begin{equation}\label{c1continuity}
\begin{split}
P'(a)  & = \frac{n a^{\lambda_1}}{b-a} \frac{f_{\lambda}(n+1)}{f_{\lambda^{(0)}}(n)}
\frac{S_{\lambda^{(0)}}(a^{n-1},b)}{S_{\lambda}(a^{n},b)} \left( P_1 - P_0 \right), \\
P'(b)  & = \frac{n b^{\lambda_1}}{b-a} \frac{f_{\lambda}(n+1)}{f_{\lambda^{(0)}}(n)}
\frac{S_{\lambda^{(0)}}(a,b^{n-1})}{S_{\lambda}(a,b^{n})} \left( P_n - P_{n-1} \right).
\end{split}
\end{equation}
As from Theorem \ref{bernsteinderivative}, we know the derivatives $(B_{1,\lambda}^{n})'(a)$ and 
$(B_{n-1,\lambda}^{n})'(b)$ and by using the fact that the segments 
$[P_0,P_1]$ and $[P_{n-1},P_{n}]$ are tangents to the curve at the 
point $P(a)$ and $P(b)$ respectively, we can show that the equations (\ref{c1continuity}) are 
true independently if the partition $\lambda$ has it first two parts equals or not.
It is rather interesting that computing the derivative $(B_{0,\lambda}^{n})'(a)$ 
or $(B_{n,\lambda}^{n})'(a)$ using the explicit expression of the Chebyshev-Bernstein 
basis (\ref{bernstein}) reveal to be difficult.    
Equations (\ref{c1continuity}) can be used to achieve the $C^{1}$ continuity between 
two Chebyshev-B\'ezier curves associated with two different partitions, as follows 

\begin{corollary}\label{c1corollary}
Let $\lambda=(\lambda_1,\lambda_2,...,\lambda_n$ and 
$\mu = (\mu_1,\mu_2,...,\mu_n)$ be two partitions of length
at most $n$ and let $\mathcal{P} = (P_0,P_1,...,P_n)$ 
(resp. $(\mathcal{Q} = Q_0,Q_1,...,Q_n)$) be the Chebyshev-B\'ezier 
points of a $\mathcal{E}_{\lambda}(n)$ (resp. a $\mathcal{E}_{\mu}(n)$)
function over an interval $[a,b]$ (resp. $[b,c]$). 
Then the composite curve $\gamma$ formed by the two curves associated with the two control polygons 
$\mathcal{P}$ and $\mathcal{Q}$ is $C^1$ at $b$ if and only if $P_{n} = Q_{0}$ and 
\small \begin{equation*}
\frac{n b^{\lambda_1}}{b-a} \frac{f_{\lambda}(n+1)}{f_{\lambda^{(0)}}(n)}
\frac{S_{\lambda^{(0)}}(a,b^{n-1})}{S_{\lambda}(a,b^{n})} \left( P_n - P_{n-1} \right) 
= \frac{n b^{\mu_1}}{c-b} \frac{f_{\mu}(n+1)}{f_{\mu^{(0)}}(n)}
\frac{S_{\mu^{(0)}}(c,b^{n-1})}{S_{\mu}(c,b^{n})} \left( Q_1 - Q_0 \right).
\end{equation*}
\end{corollary}

\begin{example}
Consider the Chebyshev-B\'ezier curve $\Gamma_1$ of order $n$ associated with 
the partition $\lambda =(1^k)$ with $k \leq n$ and control polygon 
$(P_{0},P_{1},....,P_{n})$ over an interval $[a,b]$. Consider another 
Chebyshev-B\'ezier curve $\Gamma_2$ of order $n$ associated 
with the empty partition and control polygon $(Q_{0},Q_{1},....Q_n)$ 
over an interval $[b,c]$. From equations (\ref{c1continuity}),
a necessary and sufficient condition for 
the two curves $\Gamma_1$ and $\Gamma_2$ to be $C^1$ at the point $P_n$ is that 
$$
P_{n} = Q_{0} 
\quad \textnormal{and} \quad
\frac{n(n+1)b}{k(b-a)} \frac{e_{k-1}(a,b^{n-1})}{e_k(a,b^n)} (P_{n}-P_{n-1}) = 
\frac{n}{c-b} (Q_{1}-Q_{0}).
$$
If we denote by $\rho$ the positive number such that 
$P_{n}-P_{n-1} =  \rho (Q_{1}-Q_{0})$, then, from the last equation, 
in order to achieve the $C^1$ continuity at $P_n$,
we should choose the number $c$ as 
\begin{equation}\label{cparameter}
c = b + \frac{k(b-a) e_{k}(a,b^n)}{(n+1) b \rho e_{k-1}(a,b^{n-1})}.
\end{equation}
Figure \ref{fig:figure6} shows the case $n = 3$ in this example, while Figure \ref{fig:figure7} shows another example
of the application of Corollary \ref{c1corollary} with the partitions $\lambda = (2,1)$ and $\mu = (1,1)$.   
\end{example}

%%%%%%%%%%%%%%%%%%%%%%%%%%%%%%%%%%%%%%%
%%%%%%%%%%%%Figure 6 %%%%%%%%%%%%%%%%%% 
\begin{figure}
\hskip 3 cm
\includegraphics[height=5.7cm]{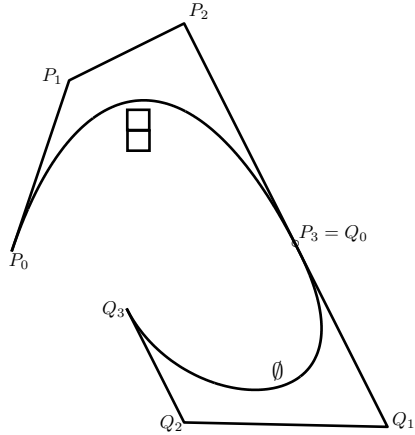}
\caption{$C^1$ continuity at the point $P_3$ between two Chebyshev-B\'ezier curves
associated with two differents partitions. The Chebyshev-B\'ezier curve with Chebyshev-B\'ezier 
points $(P_0,P_1,P_2,P_3)$ is associated to the partition $(1,1)$ and parametrized over the interval 
$[a,b] = [1,3]$. The Chebyshev-B\'ezier curve with Chebyshev-B\'ezier points $(Q_0,Q_1,Q_2,Q_3)$ is associated 
to the empty partition and parametrized over the interval $[3,c]$, where the parameter $c$ was computed 
using equation (\ref{cparameter}) to achieve the $C^1$ continuity. (see example 13)} 
\label{fig:figure6}
\end{figure}
%%%%%%%%%%%%%%%%%%%%%%%%%%%%%%%%%%%%%%%
%%%%%%%%%%%%%%%%%%%%%%%%%%%%%%%%%%%%%%%

\begin{remark}
If a partition $\lambda = (\lambda_1,\lambda_2,...,\lambda_n)$ 
of length at most $n$ satisfies $\lambda_1 = \lambda_2 = ...=\lambda_s,$
$s \leq n$, then we can iterate Theorem \ref{case1lambda} to compute the 
derivatives up to order $s-1$ of the Chebyshev-Bernstein basis.  
\end{remark}

%%%%%%%%%%%%%%%%%%%%%%%%%%%%%%%%%%%%%%%
%%%%%%%%%%%%Figure 6 %%%%%%%%%%%%%%%%%% 
\begin{figure}
\hskip 3 cm
\includegraphics[height=6.cm]{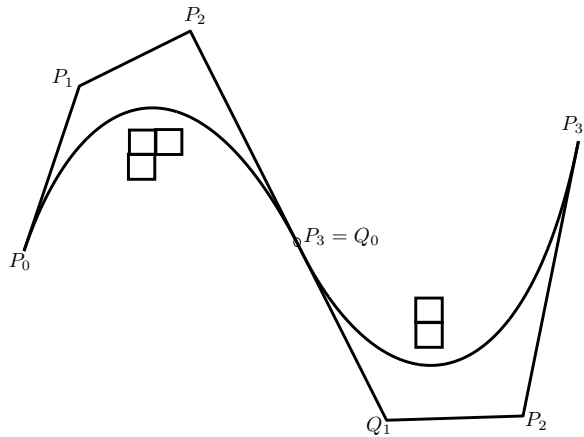}
\caption{$C^1$ continuity at the point $P_3$ between two Chebyshev-B\'ezier curves
associated with two differents partitions. The Chebyshev-B\'ezier curve with Chebyshev-B\'ezier 
points $(P_0,P_1,P_2,P_3)$ is associated to the partition $(2,1)$ and parametrized over the interval 
$[a,b] = [1,3]$. The Chebyshev-B\'ezier curve with Chebyshev-B\'ezier points $(Q_0,Q_1,Q_2,Q_3)$ is associated 
to the partition $(1,1)$ and parametrized over the interval $[3,c]$, where the parameter $c$ was computed 
using the conditions of corollary \ref{c1corollary} to achieve the $C^1$ continuity.} 
\label{fig:figure7}
\end{figure}
%%%%%%%%%%%%%%%%%%%%%%%%%%%%%%%%%%%%%%%
%%%%%%%%%%%%%%%%%%%%%%%%%%%%%%%%%%%%%%%
Consider, now, a partition $\lambda= (\lambda_1,\lambda_2,...,\lambda_n)$
of length at most $n$ in which we assume this time that 
$\lambda_1 \neq \lambda_2$. Under this condition,
the derivative of the Chebyshev-Bernstein element $B^{n}_{k,\lambda}$ over an 
interval $[a,b]$ is an element of the Chebyshev space $\mathcal{E}_{\mu}(n)$ where $\mu$ 
is the partition $\mu = (\lambda_1-1,\lambda_2,...,\lambda_n)$.  
The derivative of $B^{n}_{k,\lambda}$ can be written as linear combination
of the Chebyshev-Bernstein basis of $\mathcal{E}_{\mu}(n)$. However, if we use 
the vanishing properties of the Chebyshev-Bernstein bases,
we arrive to a three-term recurrence relation between the two Chebyshev-Bernstein bases and 
in which Theorem \ref{bernsteinderivative} does not allow for an 
easy way to compute the necessary coefficients. To solve this problem,
we can instead proceed as follows : As $\lambda_1 \neq \lambda_2$, 
we have necessarily $\lambda_1 > \lambda_2$. Therefore, the partition 
$\eta = (\lambda_1-1,\lambda_1-1,\lambda_2,...,\lambda_n)$ is a dimension elevation 
partition of $\lambda$. We can compute the Chebyshev-Bernstein basis of 
$\mathcal{E}_{\eta}(n+1)$ as a function of the Chebyshev-Bernstein basis of 
$\mathcal{E}_{\lambda}(n)$ according to Theorem \ref{elevationtheorem}.
Now the partition $\eta$ satisfy the condition that its first 
two parts are equals, and therefore, we can use Theorem \ref{case1lambda}
to compute the derivative. Proceeding along these two steps, in which 
we omit the computation as they can be readily done, we find  

\begin{theorem}\label{case2lambda}
Let $\lambda = (\lambda_1,\lambda_2,\lambda_3,...,\lambda_n)$
be a partition of length at most $n$ such that $\lambda_1 \neq \lambda_2$.
Then the derivative of the Chebyshev-Bernstein basis associated
with the partition $\lambda$ over an interval $[a,b]$ satisfies 
{\small \begin{equation*}
\frac{d B_{k,\lambda}^n(t)}{dt} = 
\frac{ \binom{\eta}{\lambda}_{n}}{(b-a) \binom{\eta}{\mu}_{n}}
\left( G_1(k,n) B_{k-1,\mu}^{n}(t)  + G_2(k,n) B_{k,\mu}^{n}(t)
+ G_3(k,n) B_{k+1,\mu}^{n}(t) \right),  
\end{equation*}}
where 
$$
G_1(k,n) = (n+1-k) a \Gamma_{\lambda}^{\eta}(n,k) R_{\eta}(k-1,n+1),
$$
$$
G_2(k,n) =R_{\eta}(k,n+1) \big( (k+1) b \Delta_{\lambda}^{\eta}(n,k) -
(n+1-k) a \Gamma_{\lambda}^{\eta}(n,k) \big),
$$
and 
$$
G_3(k,n) = -(k+1) b \Delta_{\lambda}^{\eta}(n,k) R_{\eta}(k+1,n+1).
$$
and the partition $\eta$ and $\mu$ are given by 
$\eta = (\lambda_1-1,\lambda_1-1,\lambda_2,...,\lambda_n)$ and 
$\mu = (\lambda_1-1,\lambda_2,...,\lambda_n)$, the factors 
$\Delta_{\lambda}^{\eta}$, $\Gamma_{\lambda}^{\eta}(n,k)$ and
$R_{\eta}$ are defined in (\ref{factor1}), (\ref{factor2}) and (\ref{Requation})
respectively. 
We adopt the convention that $B_{-1,\mu}^{n} \equiv B_{n+1,\mu}^{n} \equiv 0$.
\end{theorem}

Let $\lambda = (\lambda_1,\lambda_2,...,\lambda_n)$ be a partition 
of length at most $n$ such that $\lambda_1 \neq \lambda_2$ and 
consider a $\mathcal{E}_{\lambda}(n)$-function $P$ written 
in the Chebyshev-Bernstein basis over an interval $[a,b]$ as 
$$
P(t) = \sum_{k=0}^{n} B_{k,\lambda}^{n}(t)  P_{k}.
$$
Using Theorem \ref{case2lambda}, we can express 
the derivative of the function $P$ in term of the Chebyshev-Bernstein 
basis over the interval $[a,b]$ of the space $\mathcal{E}_{\mu}(n)$.
Doing so leads to the following

\begin{theorem}
Let $\lambda = (\lambda_1,...,\lambda_n)$ be a partition 
such that $\lambda_1 \neq \lambda_2$ and let $P$ be a 
$\mathcal{E}_{\lambda}(n)$-function, 
written in the Chebyshev-Bernstein basis over an interval $[a,b]$ as 
$$
P(t) = \sum_{k=0}^{n} B_{k,\lambda}^{n}(t) P_{k}.
$$
Then, we have 
{\small {\begin{equation*}
P'(t) = \frac{ \binom{\eta}{\lambda}_{n}}{(b-a) \binom{\eta}{\mu}_{n}} 
\sum_{k=0}^{n} \left( G_{3}(k-1,n) P_{k-1} + G_{2}(k,n) P_{k} + G_{1}(k+1,n) P_{k+1} \right),  
B_{k,\mu}^{n}(t)
\end{equation*}}} 
where the factors $G_{1},G_{2}, G_{3}$ and the partitions $\eta$ and $\mu$ are defined in Theorem \ref{case2lambda}.  
\end{theorem}

\section{Conclusions}
In this paper, we carried out a comprehensive study of the notion of Chebyshev blossom in \muntz 
spaces, thereby showing their adequacy in free-form design schemes. An interesting 
aspect of the work was the followed methodology in providing for an explicit expression of
the Chebyshev-Bernstein basis. Most of the steps in the proof were combinatorial 
in nature. Similar arguments, therefore, could be applied to any extended Chebyshev
space constructed from weight functions, in the sense that the condensation formula
will provide us with the pseudo-affinity factor and in which the combinatorics of the de 
Casteljau paths can be employed to give extra-information on the derivatives of the 
Chebyshev-Bernstein basis. Such a program will be the object of a forthcoming
contribution. Moreover, the problem of higher order continuity and the issue of 
shaping with Young diagrams lead to interesting problems for future work.  
 
\vskip 0.2 cm         

\subsubsection*{Acknowledgement : This work was partially supported by the MEXT Global COE project.} 

\vskip 0.2 cm

\label{references}


\begin{thebibliography}{4}

\bibitem{Ait-Haddou1} R. Ait-Haddou, T. Nomura and L. Biard,  A refinement of the variation 
diminishing property of B\'ezier curves, Comput. Aided Geom. Design,  Volume 27, Issue 2, 202--211, 2010 

\bibitem{Aldaz} J.M. Aldaz, O. Kounchev, H. Render, Shape preserving properties of 
Bernstein operators on extended Chebyshev spaces, Numer. Math. 114 (1) (2009) 1–25.

\bibitem{Bowman}D. Bowman and D.M. Bradley, The Algebra and Combinatorics of Shuffles and Multiple Zeta Values
J. Combin. Theory Ser. A , Vol 97, Issue 1,(2002),43-61 

\bibitem{Chen1} K.T.Chen, Iterated Path Integrals, Bull. AMS 83 (1977) 831-879.
  
\bibitem{Farin} G. Farin, Curves and Surfaces for CAGD. A practical Guide, fifth ed. 
The Morgan Kaufmann Series in Computer Graphics Series, (2002).

\bibitem{Kra} C. Krattenthaler, Advanced determinant calculus, S\'eminaire Lotharingien Combin. 42 (1999).

\bibitem{Macdonald} I.G. Macdonald, Symmetric functions and Hall polynomials, Oxford Math. Monographs, (1979).

\bibitem{Mazure1} M.-L. Mazure, Chebyshev blossoming, RR 953M IMAG, Universit\'e Joseph Fourier, Grenoble
(January 1996).

\bibitem{Mazure2} M.-L. Mazure, Chebyshev–Bernstein bases, Comput. Aided Geom. Design 16 (1999) 291–315.

\bibitem{Mazure3} M.-L. Mazure, Chebyshev spaces with polynomial blossoms, Adv. Comput. Math. 10 (1999) 219--238.

\bibitem{Mazure4} M.-L. Mazure, Bernstein bases in \muntz spaces, Num. Algorithms 22 (1999) 285–304

\bibitem{Mazure5} M.-L. Mazure, Blossoms of generalized derivatives in Chebyshev spaces, J. Approx. Theo.
Vol 131, Issue 1 (2004)  47-58 

\bibitem{Pottmann}H. Pottmann, The geometry of Tchebycheffian splines, Comput. Aided Geom. Design 10 (1993)
181--210.

\bibitem{Ramshaw} L. Ramshaw, Blossoms are polar forms, Comput. Aided Geom. Design, Vol 6, no
4, 323--358, (1989).

\end{thebibliography}
\end{document}